\documentclass[10pt]{amsart}
\usepackage{amsmath}
\usepackage{amsthm}
\usepackage{amsfonts}
\usepackage{amssymb}
\usepackage{amscd}
\usepackage[all]{xy}
\usepackage[percent]{overpic}
\usepackage{graphicx}
\usepackage{xcolor}
\usepackage{enumitem}
\usepackage[backref=page]{hyperref}
\usepackage{tikz}
\usepackage{tikz-cd}
\usetikzlibrary{decorations.pathmorphing, decorations.markings, calc, cd}
\usepackage{pgf}
\usepackage{overpic}

\hypersetup{
    colorlinks=true,       
    linkcolor=blue,          
    citecolor=blue,        
    filecolor=blue,      
    urlcolor=blue           
}

    \newtheorem{Lem}{Lemma}[section]
    \newtheorem{Lem-Def}[Lem]{Lemma-Definition}
		\newtheorem{CorDef}[Lem]{Corollary-Definition}
    \newtheorem{Prop}[Lem]{Proposition}

    \newtheorem{Thm}[Lem]{Theorem}  
		\newtheorem*{Thm*}{Theorem}
    \newtheorem{Cor}[Lem]{Corollary}

\theoremstyle{definition}
\font\smallsc=cmcsc10
\font\smallsl=cmsl10

    \newtheorem{Def}[Lem]{Definition}
    \newtheorem{Exa}[Lem]{Example}
    \newtheorem{Rem}[Lem]{Remark}
    \newtheorem{Cons}[Lem]{Construction}
    \newtheorem{Not}[Lem]{Notation}

\newcommand{\ora}[1]{\overrightarrow{#1}}

\newcommand{\E}{\mathcal E}

\newcommand{\Spec}{\text{Spec}\,}
\newcommand{\red}{\text{\rm red}}

\newcommand{\A}{\textnormal{Adm}}

\newcommand{\F}{\mathcal F}
\newcommand{\Z}{\mathcal Z}
\newcommand{\I}{\mathcal I}
\newcommand{\M}{\mathcal M}

\renewcommand{\L}{\mathcal L}
\renewcommand{\O}{\mathcal O}
\newcommand{\C}{\mathcal C}

\newcommand{\X}{\mathcal X}
\newcommand{\Y}{\mathcal Y}
\newcommand{\D}{\mathcal D}
\newcommand{\R}{\mathbb R}

\newcommand{\col}{\colon}

\newcommand{\ra}{\rightarrow}
\newcommand{\ol}{\overline}
\newcommand{\supp}{\text{supp}}

\newcommand{\wh}{\widehat}

\newcommand{\J}{\mathcal{J}}

\newcommand{\lra}{\longrightarrow}

\DeclareMathOperator{\spa}{span}

\DeclareMathOperator{\st}{st}
\DeclareMathOperator{\Hom}{Hom}
\DeclareMathOperator{\codim}{codim}
\DeclareMathOperator{\DivMW}{\textbf{Div}}
\DeclareMathOperator{\aj}{aj}
\newcommand{\Tor}{\textnormal{TV}}
\renewcommand{\l}{\ell}

\newcommand{\Pos}{\mathcal{QD}_{v_0,\mu}}

\newcommand{\Aut}{\textnormal{Aut}}
\newcommand{\cone}{\textnormal{cone}}
\newcommand{\val}{\text{val}}

\renewcommand{\div}{\textnormal{div}}
\newcommand{\Div}{\text{Div}}
\newcommand{\Prin}{\text{Prin}}
\renewcommand{\Im}{\text{Im}}

\newcommand{\grap}{\mathcal{SG}}
\newcommand{\ord}{\textnormal{ord}}

\newcommand{\trop}{{\operatorname{trop}}}
\newcommand{\an}{{\operatorname{an}}}

\title[The resolution of the universal Abel map]{The resolution of the universal Abel map via tropical geometry and applications}

\author{Alex Abreu and Marco Pacini}

\begin{document}

\begin{abstract}
Let $g$ and $n$ be nonnegative integers and $\mathcal A=(a_0,\dots,a_n)$ a sequence of $n+1$  integers summing up to $d$. Let $\ol{\M}_{g,n+1}$ be the moduli space of $(n+1)$-pointed stable curves of genus $g$ and  $\ol{\J}_{\mu,g}\ra \ol{\M}_{g,1}$ be the Esteves' universal Jacobian, where $\mu$ is a universal genus-$g$ polarization of degree $d$. We give an explicit resolution of the universal Abel map $\alpha_{\mathcal A,\mu}\col \ol{\M}_{g,n+1}\dashrightarrow \ol{\J}_{\mu,g}$, taking a pointed curve $(X,p_0,\dots,p_n)$ to $\O_X(\sum_{0\le i\le n} a_ip_i)$. The blowup of $\ol{\M}_{g,n+1}$ giving rise to the resolution is  inspired by the  resolution of the tropical analogue of the map $\alpha_{\mathcal A,\mu}$ (in the category of generalized cone complexes). 
As an application, we describe the double ramification cycle in terms of the universal sheaf inducing the resolution of the map $\alpha_{\mathcal A,\mu}$.
\end{abstract}

\maketitle

\tableofcontents

\section{Introduction}

\subsection{History and motivation}
The Abel map of a smooth curve encodes many geometric properties of the curve. For instance, the complete linear systems and the Brill-Noether varieties $W_{g,d}^0$  are, respectively, the fibers and the images of the Abel map. On the other hand, a powerful tool  for studying smooth curves consists in taking degenerations to singular curves: it was through this technique that the Brill-Noether and Gieseker-Petri theorems were proved, see \cite{G} and \cite{GH}. It is natural to expect that a construction of the Abel map for singular curves could be very useful to study the degeneration process from smooth curves to singular ones. \par

  Several works have been dedicated to Abel maps of singular curves. To our knowledge, the first construction appeared in \cite{AK} for irreducible curves. The reducible case is more complicated, principally due to the fact that, in general, the compactified Jacobian does not parametrize all the invertible sheaves on the curve. In this case it is usual to resort to a one-parameter deformation to smooth curves, and construct an Abel map as a limit of the Abel maps of the smooth fibers. This has been started in degree one in \cite{CE}, in degree two in \cite{CEP} and \cite{P}, for curves of compact type and any degree in \cite{CP}, and for curves with two components and any degree in \cite{AbCP}. The general problem is still open.\par 
	
	 Instead of considering a one-parameter deformation, one can consider the universal deformation. More precisely, let $g$ and $n$ be nonnegative integers, and $d$ be an integer. Let $\M_{g,n+1}$ be the moduli stack of  smooth $(n+1)$-pointed curves of genus $g$ and $\J_{d,g}\to \M_g$ be the universal Jacobian parametrizing degree-$d$ invertible sheaves on smooth curves of genus $g$. Given a sequence of  integers $\mathcal{A}:=(a_0,\ldots,a_n,m)$ such that $d=\sum a_i+m(2g-2)$, one can define the universal Abel map $\alpha_{\mathcal{A}}\col \M_{g,n+1}\to \J_{d,g}$ taking $[(X,p_0,\ldots,p_n)]$ to $[(X,\omega_X^m\otimes\O_X(a_0p_0+\ldots+a_np_n))]$. In the classical setting, it is usual to consider $\mathcal{A}=(1,1,\ldots,1,0)$ or $\mathcal{A}=(-n,1,\ldots,1,0)$. \par 

	Many interesting classes in $\M_{g,n+1}$ can be defined in terms of the universal Abel map such as, for instance, the double ramification cycle, which is the pullback via $\alpha_{\mathcal{A}}$ of the zero section in $\J_{0,g}$ for $d=m=0$. A natural question is how to extend these classes to the compactification $\overline{\M}_{g,n+1}$ in a meaningful way, and express them in terms of the tautological classes. In the case of the double ramification cycle, an extension of the class was defined in \cite{L1,L2, GV} and computed in  \cite{JPPZ}; previous partial results were achieved in \cite{Hai, GZ, GZ1}. 	One possible approach to extend these classes is to define an Abel map over  (a modification of) $\ol{\M}_{g,n+1}$, the Deligne-Mumford-Knudsen compactification of  $\M_{g,n+1}$. This requires the choice of  a compactified Jacobian as a target. \par

The first construction of a universal compactified Jacobian over $\overline{\M}_g$ is due to Caporaso  \cite{C} and \cite{C08}; a version over $\overline{\M}_{g,n+1}$ was recently introduced by Melo \cite{M11}. Another remarkable construction was given by Esteves \cite{Es01} for families of (pointed) curves; this construction can be carried through to the universal setting, giving rise to a universal compactified Jacobian $\ol{\J}_{\mu,g,n+1}$ over $\overline{\M}_{g,n+1}$ or, when $n=1$, simply $\ol{\J}_{\mu,g}$ over $\ol{\M}_{g,1}$. This compactified Jacobian depends on the choice of a universal genus-$g$ polarization $\mu$ of degree $d$ (see \cite{M15}). \par

	Recently the problem of resolving the universal Abel map or, in some cases, the Abel section $\M_{g,n+1}\to \J_{g,n+1}$,  has attracted a lot of attention, principally motivated, as already mentioned, by the connection with the double ramification cycle. In \cite{KP} and \cite{M15} the authors found a suitable polarization $\mu$ depending on $\mathcal{A}$ such that the Abel section extends to a map $\overline{\M}_{g,n+1}\to\ol{\J}_{\mu,g,n+1}$, while in \cite{H} and \cite{MW} a modification of the source $\overline{\M}_{g,n+1}$ was introduced to define a partial resolution of the Abel map. We refer to Section \ref{sec:comparison} for a more precise discussion. \par

	This paper is dedicated to the construction of an explicit resolution of the universal Abel map $\alpha_{\mathcal A,\mu}\col \ol{\M}_{g,n+1}\dashrightarrow\ol{\J}_{\mu,g}$. First, we resolve the analogue problem in tropical geometry. More precisely, let $M_{g,n+1}^{\trop}$ be the moduli space of stable $(n+1)$-pointed tropical curves of genus $g$, and $J_{\mu,g}^{\trop}\ra M_{g,1}^{\trop}$ be the universal tropical Jacobian, recently introduced in \cite{AP2}. We define the universal tropical Abel map $\alpha^{\trop}_{\mathcal A,\mu}\col M_{g,n+1}^{\trop}\ra J_{\mu,g}^{\trop}$, and introduce an explicit refinement of $M_{g,n+1}^{\trop}$ whose composition with $\alpha^{\trop}_{\mathcal A,\mu}$ is a morphism of generalized cone complexes. The upshot is that this refinement induces a natural blowup $\ol{\M}_{\mathcal{A},\mu}$ of $\ol{\M}_{g,n+1}$ resolving  $\alpha_{\mathcal A,\mu}$. 
As an application, we describe the double ramification cycle in terms of the universal sheaf over $\ol{\M}_{\mathcal{A},\mu}$ inducing the resolution of $\alpha_{\mathcal A,\mu}$. \par

 It seems to us that the results about the universal tropical Abel map are interesting on their own. The tropical approach already appears in \cite{MW} and implicitly in \cite{H} where the resolution of the Abel section is described via toric geometry. This paper is inspired by the work of Holmes \cite{H}: our stack $\ol{\M}_{\mathcal{A},\mu}$ is a compactification of the stack $\M^\diamond$ appearing there. 

\subsection{The results} 
Let us give more details about our results. We start with the tropical setting. Let $(X,p_0)$ be a pointed tropical curve and $J^\trop(X)$ its tropical Jacobian. There is a piecewise-linear map $X^{n+1}\to J^\trop(X)$ sending $(p_0,\ldots,p_n)$ to the class of the divisor $p_1+\ldots+p_n-np_0$. This map was extensively studied in \cite{MZ} and \cite{BF}. One can also consider the universal setting as follows. \par
 Given a sequence $\mathcal A=(a_0, \ldots , a_n , m)$ of integers and a universal genus-$g$ polarization $\mu$ of degree-$d$  such that $d=\sum a_i+m(2g-2)$, we define the universal tropical Abel map $\alpha_{\mathcal A,\mu}^{\trop}\col M_{g,n+1}^\trop\to J^\trop_{\mu,g}$ taking an $(n+1)$-pointed tropical curve $(X,p_0,\ldots,p_n)$ to the pair $(\st(X,p_0),\D)$ where $\st(X,p_0)$ is the stable model of $(X,p_0)$ (see Section \ref{sec:tropicalcurves}) and $\D$ is the unique $(p_0,\mu)$-quasistable divisor equivalent to $m\omega_X+\sum a_ip_i$ (see Theorem \ref{thm:quasistable}). The map $\alpha_{\mathcal A,\mu}^{\trop}$ is continuous but it is not, in general, a morphism of generalized cone complexes. This is the tropical analogue of the fact that the Abel map $\alpha_{\mathcal A,\mu}\col \overline{\M}_{g,n+1}\dashrightarrow\ol{\J}_{\mu,g}$ is just  rational. In Section \ref{sec:abeltrop} we construct a resolution of the map $\alpha_{\mathcal A,\mu}^{\trop}$ in the category of generalized cone complexes, in the sense of the following theorem.

\begin{Thm*}[\ref{thm:abeltrop}]
There exists an explicit refinement $\beta_{\mathcal A,\mu}^{\trop}\col M_{\mathcal{A},\mu}^\trop\to M_{g,n+1}^\trop$ such that $\alpha_{\mathcal A,\mu}^{\trop}\circ\beta_{\mathcal A,\mu}^{\trop}\col M_{\mathcal{A},\mu}^\trop\to J_{\mu,g}^\trop$ is a morphism of generalized cone complexes.
\end{Thm*}

	The refinement $\beta_{\mathcal A,\mu}^{\trop}$ is explictly given in terms of acyclic flows on graphs (see Theorem \ref{thm:fan}). We note that refinements of $M_{g,n+1}^\trop$ correspond to blow-ups of $\ol{\M}_{g,n+1}$ and can be used to reveal geometric properties of birational modifications of $\ol{\M}_{g,n+1}$. This idea already appeared in \cite{AP1} and \cite{H}. In fact, the refinement $\beta_{\mathcal A}^{\trop}$ is closely related to the fan defined by Holmes \cite{H}: the difference is that Holmes only considers flows arising from the zero divisor, while we consider flows arising from every quasistable divisor (see Section \ref{sec:comparison}). We get the following result.
	
	\begin{Thm*}[\ref{thm:mainglobal}]
Let $\beta_{\mathcal{A},\mu}\col \ol{\M}_{\mathcal{A},\mu}\to \ol{\M}_{g,n+1}$ be the  birational morphism induced by $M^{\trop}_{\mathcal{A},\mu}$. Then $\ol{\M}_{\mathcal{A},\mu}$ is the normalization of the closure of the image of the Abel section $\ol{\M}_{g,n+1}\dashrightarrow \ol{\J}_{\mu,g,n+1}$ and the rational map $\alpha_{\mathcal{A},\mu}\circ\beta_{\mathcal{A},\mu}\col \ol{\M}_{\mathcal{A},\mu}\dashrightarrow  \ol{\J}_{\mu,g}$ is defined everywhere, i.e., it is a morphism of Deligne-Mumford stacks.  
	\end{Thm*}

  Our approach first deals with the local case (see Theorem \ref{thm:mainlocal}). The local case could be used to resolve variants of Abel maps for other moduli stacks of curves endowed with marked points and universal sheaves such as, for instance, the moduli stack of spin curves or stable maps. The main technical tool is Proposition \ref{prop:Icap}, which translates the combinatorial properties of $M_{\mathcal{A},\mu}^\trop$ to algebraic properties. In particular, it tells us that certain ideal sheaves on the universal family over $\ol{\M}_{\mathcal{A},\mu}$ are relative torsion-free and rank-$1$, which is a fundamental property satisfied by the sheaves parametrized by the compactified Jacobian $\ol{\J}_{\mu,g}$.\par

Finally, let $d=0$ and $\mu=0$, and let $\I$ be the universal sheaf over the universal family $\pi\col\C_{g,n+1}\times_{\ol{\M}_{g,n+1}}\ol{\M}_{\mathcal{A},\mu}\to \ol{\M}_{\mathcal{A},\mu}$ inducing the morphism $\alpha_{\mathcal A,\mu}\circ\beta_{\mathcal A,\mu}$. We get the following formula, following the work of Dudin \cite{D}. 

\begin{Thm*}[\ref{thm:drc}]
Let $DRC_{\mathcal{A}}$ be the double ramification cycle. Then 
\[
DRC_{\mathcal A}=(\beta_{\mathcal{A},\mu})_*(c_g(R^1\pi_*(\I))).
\]
\end{Thm*}

In a nutshell, in Sections \ref{sec:divgraph} we introduce some preliminaries on divisors on graphs and tropical curves, including quasistability. In Section \ref{sec:abeltrop} we construct the moduli space $M_{\mathcal{A},\mu}^\trop$ and show that the induced map $M_{\mathcal{A},\mu}^\trop\to J_{\mu,g}^\trop$ is a morphism of generalized cone complexes. Section \ref{sec:torbin} is devoted to toric geometry and binomial ideals, while Section \ref{sec:combalg} manly translates the combinatorial properties of $M_{\mathcal{A},\mu}^\trop$ to geometric and algebraic results. In Section \ref{sec:Abelres} we construct the local and the universal resolutions of the Abel map, and Section \ref{sec:app} is dedicated to the application to the double ramification cycle and to a discussion on some future research directions. In Section \ref{sec:comparison}, we compare the results achieved in this paper with previous related works.

\section{Divisors on graphs and tropical curves}\label{sec:divgraph}

\subsection{Graphs}
Given a graph $\Gamma$, we denote by $E(\Gamma)$ the set of edges  and by $V(\Gamma)$ the set of vertices of $\Gamma$. If $\E\subset E(\Gamma)$ and $v$ is a vertex of $\Gamma$, the \emph{valence of $v$ in $\E$} is
\begin{equation}\label{eq:valE}
\val_\E(v):=\#\{e\in \E ; e \text{ is incident to } v\},
\end{equation}
with loops counting twice. 
If $\E=E(\Gamma)$, we simply write $\val(v)$ and call it the \emph{valence} of $v$. We say that $\Gamma$ is \emph{$k$-regular} if every vertex of $\Gamma$ has valence $k$; we say that $\Gamma$ is a \emph{circular graph} if $\Gamma$ is $2$-regular and connected. A \emph{cycle} on $\Gamma$ is a circular subgraph of $\Gamma$. \par
	A \emph{digraph} (directed graph) is a graph $\Gamma$ where each edge has an orientation, i.e., there are functions $s,t\col E(\Gamma)\to V(\Gamma)$ called the \emph{source} and \emph{target} and each edge is oriented from the source to the target. We denote a digraph by $\overrightarrow{\Gamma}$ and  its underlying graph by $\Gamma$. We let $E(\overrightarrow{\Gamma})$ be the set of oriented edges of $\overrightarrow{\Gamma}$ and we set $V(\overrightarrow{\Gamma}):=V(\Gamma)$.\par
 Let $\Gamma$ be a graph. Fix disjoint subsets $V, W\subset V(\Gamma)$. We define $E(V,W)$ as the set of edges joining a vertex in $V$ with one in $W$. More generally if $V$ and $W$ have nonempty intersection, we define $E(V,W):=E(V\setminus W, W\setminus V)$. We set $V^c:=V(\Gamma)\setminus V$.  If $E(V,V^c)$ is nonempty, it is called a \emph{cut} of $\Gamma$. We set 
\begin{equation}\label{eq:delta}
\delta_{\Gamma,V}:=|E(V,V^c)|.
\end{equation}
 Sometimes, we simply write $\delta_V$ instead of $\delta_{\Gamma,V}$. \par

		Fix a subset $\E\subset E(\Gamma)$. We define the graphs $\Gamma/\E$ and $\Gamma_\E$ as the graphs obtained, respectively, by the contraction of the edges in $\E$ and by the removal of edges in $\E$. We say that $\E$ is \emph{nondisconnecting} if $\Gamma_\E$ is connected. Note that there is a natural surjection $V(\Gamma)\to V(\Gamma/\E)$ and a natural identification $E(\Gamma/\E)=E(\Gamma)\setminus\E$. Moreover, $V(\Gamma_\E)=V(\Gamma)$ and $E(\Gamma_\E)=E(\Gamma)\setminus\E$.\par
	
We let $b_i(\Gamma)$ be the $i$-th Betti number of $\Gamma$ for $i=0,1$, i.e., $b_0(\Gamma)$ is the number of connected components of $\Gamma$ and $b_1(\Gamma):=|E(\Gamma)|-|V(\Gamma)|+b_0(\Gamma)$.	
		
\begin{Prop}
\label{prop:b0b1}
Let $\Gamma$ be a connected graph and let $E(\Gamma)=\E\coprod\F$ be a partition of $E(\Gamma)$. Then
\[
b_1(\Gamma/\E)=|\F|-b_0(\Gamma_\F)+1.
\]
\end{Prop}
\begin{proof} 
We compute:
\begin{align*}
b_1(\Gamma/\E)&=|E(\Gamma/\E)|-|V(\Gamma/\E)|+1\\
              &=|E(\Gamma)|-|\E|-|V(\Gamma/\E)|+1\\
							&=|\F|-b_0(\Gamma_\F)+1,
\end{align*}
from which the result follows.
\end{proof}

A graph $\Gamma$ \emph{specializes} to a graph $\Gamma'$ if there is $\E\subset E(\Gamma)$ such that $\Gamma'$ is isomorphic to $\Gamma/\E$. We denote a specialization of $\Gamma$ to $\Gamma'$ by $\iota\col \Gamma\ra\Gamma'$. A specialization $\iota\col \Gamma\ra\Gamma'$ comes equipped with a surjective map $\iota^V\col V(\Gamma)\ra V(\Gamma')$ and an injective map $\iota^E\col E(\Gamma')\ra E(\Gamma)$. We usually write $\iota=\iota^V$ and see $E(\Gamma')$ as a subset of $E(\Gamma)$ via $\iota^E$. A similar notion of specialization can be given for digraphs. \par 
   
We define the graph $\Gamma^\E$ as the graph obtained from $\Gamma$ by adding exactly one vertex in the interior of each edge in $\E$. We call $\Gamma^\E$ the \emph{$\E$-subdivision} of $\Gamma$. Note that there is a natural inclusion $V(\Gamma)\subset V(\Gamma^\E)$. We call a vertex in $V(\Gamma^\E)\setminus V(\Gamma)$ an \emph{exceptional vertex}. We set $\Gamma^{(2)}:=\Gamma^{E(\Gamma)}$. \par

	More generally, a graph $\Gamma'$ is a \emph{refinement} of a graph $\Gamma$ if $\Gamma'$ is obtained from $\Gamma$ by successive subdivisions. In particular, there is an inclusion $a\col V(\Gamma)\to V(\Gamma')$ and a surjection $b\col E(\Gamma')\to E(\Gamma)$ such that for any edge $e$ of $\Gamma$ there are distinct vertices $x_0,\ldots, x_n\in V(\Gamma')$ and distinct edges $e_1,\ldots, e_n\in E(\Gamma')$ such that
	\begin{enumerate}
\item $x_0=a(v_0)$, $x_n=a(v_1)$ and $x_{i}\notin\Im(a)$ for every $i=1,\ldots,n-1$, where $v_0$ and $v_1$ are precisely the vertices of $\Gamma$ incident to $e$;
\item $b^{-1}(e)=\{e_1,\ldots, e_n\}$;
\item the vertices in $V(\Gamma')$ incident to $e_i$ are precisely $x_{i-1}$ and $x_i$, for $i=1,\dots,n$.
\end{enumerate}
We say that the edges $e_1,\ldots, e_n$ (respectively, $x_1,\ldots, x_{n-1}$) are the edges (respectively, the exceptional vertices) \emph{over} $e$.\par
One can also give a similar notion of refinement for digraphs, where in addition we ask in (1) that the vertices $v_0$ and $v_1$ are respectively the source and the target of $e$, and in  (3) that $x_{i-1}$ and $x_i$ are respectively the source and the target of $e_i$ for $i=1,\dots,n$.\par

	Let $\iota\col \Gamma\to\Gamma'$ be a specialization of graphs, and $\E'\subset E(\Gamma')$, $\E\subset E(\Gamma)$ be sets such that $\E'\subset\E\cap E(\Gamma')$. We call a specialization  $\iota^\E\col\Gamma^\E\ra\Gamma'^{\E'}$ \emph{compatible with $\iota$} if the following diagrams
\[
\SelectTips{cm}{11}
\begin{xy} <16pt,0pt>:
\xymatrix{ V(\Gamma)\ar[d]\ar[r]^{\iota} &V(\Gamma')\ar[d]&& E(\Gamma')\ar[r]^\iota &E(\Gamma)\\
             V(\Gamma^\E)\ar[r]^{\iota^\E}& V(\Gamma'^{\E'})&&E(\Gamma'^{\E'})\ar[u]\ar[r]^{\iota^\E} & E(\Gamma^\E)\ar[u]
}
\end{xy}
\]
are commutative. If $\iota$ is the identity of $\Gamma$, we just call $\iota^\E$ \emph{compatible}. Note that there are $2^{|\E\setminus\E'|}$ compatible specializations of $\Gamma^\E$ to $\Gamma^{\E'}$.

	A \emph{(vertex-)weighted graph} is a pair $(\Gamma,w)$, where $\Gamma$ is a connected graph and $w$ is a function $w\colon V(\Gamma)\to\mathbb{Z}_{\geq0}$. The genus of a weighted graph $(\Gamma,w)$ is defined as $g(\Gamma,w):=\sum_{v\in V(\Gamma)} w(v)+b_1(\Gamma)$, where $b_1(\Gamma)$ is the first Betti number of $\Gamma$. \par

	A \emph{graph with legs} indexed by the finite set $L$ (the set of legs) is the data of a graph $\Gamma$ and a map $\text{leg}_\Gamma\col L\to V(\Gamma)$. We denote by $L(v)$ the set of legs incident to $v$, i.e., $L(v):=\text{leg}_\Gamma^{-1}(v)$.  Usually, we will simply write $\Gamma$ for a weighted graph with legs and we denote by $L(\Gamma)$ its set of legs. A \emph{graph with $n$ legs} is a  graph with legs $\Gamma$ such that $L(\Gamma)=I_n:=\{0,1,\ldots,n-1\}$. If $\Gamma$ is a graph with $n$ legs, we will always set $v_0:=\text{leg}_\Gamma(0)\in V(\Gamma)$. We will denote by $(\Gamma,v_0)$ a graph with $1$ leg. 
	
	If $(\Gamma,w,\text{leg}_{\Gamma})$ and $(\Gamma',w',\text{leg}_{\Gamma'})$ are weighted graphs with legs, we say that a specialization $\iota\col\Gamma\ra\Gamma'$ is a \emph{specialization of weighted graphs with legs} if, for every $v'\in V(\Gamma')$, we have  $w'(v')=g(\iota^{-1}(v))$, and $\text{leg}_{\Gamma}=\text{leg}_{\Gamma'}\circ \iota^V$.

We let $\grap_{g,n}$ be the poset of the isomorphism classes of weighted graphs with $n$ legs of genus $g$, where the partial order is given by specialization of graphs.

	A weighted graph with $n$ legs $\Gamma$ is \emph{stable} if $\val(v)+2w(v)+|L(v)|\geq3$ for every vertex $v\in V(\Gamma)$. 
A \emph{tree} is a connected graph whose first Betti number is $0$ or, equivalently, a connected graph whose number of edges is equal to the number of its vertices minus one.

 Every weighted graph $\Gamma$ with $n$ legs has a \emph{stable reduction} $\st(\Gamma)$, which is the weighted stable graph with $n$ legs endowed with a unique refinement $\widehat{\st(\Gamma)}$ and a unique specialization $\text{red}_\Gamma\col\Gamma\ra \widehat{\st(\Gamma)}$ such that for every specialization $\iota\col\Gamma\to\widehat{\Gamma'}$ where $\widehat{\Gamma'}$ is a refinement of a weighted stable graph $\Gamma'$ with $n$ legs, $\iota$ factors through $\text{red}_\Gamma$. We note that $\text{red}_\Gamma$ is a sequence of specializations starting from $\Gamma$, in each step contracting an edge  incident to a vertex $v$ of valence $1$, weight $0$ and where there is at most one leg attached to $v$. Then $\st(\Gamma)$ is obtained by just removing from $\widehat{\st(\Gamma)}$ any vertex with valence $2$, weight $0$ and with no legs attached to it. We illustrate this process in the following example.\par
\begin{Exa}
\label{exa:stgraph}
Let $(\Gamma,v_0)$ be the leftmost weighted graph with 1  leg in Figure \ref{fig:stgraph}. Then $\st(\Gamma)$ is the rightmost graph, while $\widehat{\st(\Gamma)}$ is the graph in the middle.
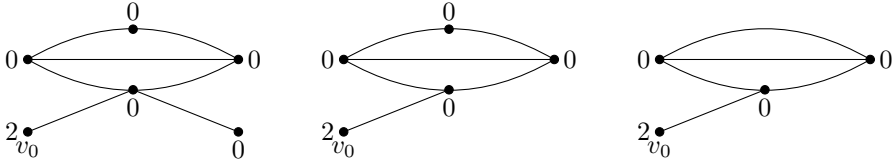
\begin{figure}[h]
\begin{tikzpicture}[scale=2.8]
\begin{scope}[shift={(0,0)}]
\draw (0,0) to [out=30, in=150] (1,0);
\draw (0,0) to (1,0);
\draw (0,0) to [out=-30, in=-150] (1,0);
\draw (0.5,-0.144) to (0,-0.344);
\draw (0.5,-0.144) to (1,-0.344);
\draw[fill] (0,0) circle [radius=0.02];
\node[left] at (0,0) {$0$};
\draw[fill] (1,0) circle [radius=0.02];
\node[right] at (1,0) {$0$};
\draw[fill] (0.5,0.144) circle [radius=0.02];
\node[above] at (0.5,0.144) {$0$};
\draw[fill] (0.5,-0.144) circle [radius=0.02];
\node[below] at (0.5,-0.144) {$0$};
\draw[fill] (0,-0.344) circle [radius=0.02];
\node[left] at (0,-0.344) {$2$};
\node[below] at (0,-0.344) {$v_0$};
\draw[fill] (1,-0.344) circle [radius=0.02];
\node[below] at (1,-0.344) {$0$};
\end{scope}
\begin{scope}[shift={(1.5,0)}]
\draw (0,0) to [out=30, in=150] (1,0);
\draw (0,0) to (1,0);
\draw (0,0) to [out=-30, in=-150] (1,0);
\draw (0.5,-0.144) to (0,-0.344);
\draw[fill] (0,0) circle [radius=0.02];
\node[left] at (0,0) {$0$};
\draw[fill] (1,0) circle [radius=0.02];
\node[right] at (1,0) {$0$};
\draw[fill] (0.5,0.144) circle [radius=0.02];
\node[above] at (0.5,0.144) {$0$};
\draw[fill] (0.5,-0.144) circle [radius=0.02];
\node[below] at (0.5,-0.144) {$0$};
\draw[fill] (0,-0.344) circle [radius=0.02];
\node[left] at (0,-0.344) {$2$};
\node[below] at (0,-0.344) {$v_0$};
\end{scope}
\begin{scope}[shift={(3,0)}]
\draw (0,0) to [out=30, in=150] (1,0);
\draw (0,0) to (1,0);
\draw (0,0) to [out=-30, in=-150] (1,0);
\draw (0.5,-0.144) to (0,-0.344);
\draw[fill] (0,0) circle [radius=0.02];
\node[left] at (0,0) {$0$};
\draw[fill] (1,0) circle [radius=0.02];
\node[right] at (1,0) {$0$};
\draw[fill] (0.5,-0.144) circle [radius=0.02];
\node[below] at (0.5,-0.144) {$0$};
\draw[fill] (0,-0.344) circle [radius=0.02];
\node[left] at (0,-0.344) {$2$};
\node[below] at (0,-0.344) {$v_0$};
\end{scope}
\end{tikzpicture}
\caption{Stable reduction of a weighted graph with legs.}
\label{fig:stgraph}
\end{figure}
\end{Exa}

\subsection{Divisors and flows on graphs}

Let $\Gamma$ be a graph. A divisor $D$ on $\Gamma$ is a function $D\colon V(\Gamma)\to \mathbb{Z}$. The degree of $D$ is the integer $\deg D:=\sum_{v\in V(\Gamma)}D(v)$. The set of divisors on $\Gamma$ form an Abelian group denoted by $\Div(\Gamma)$. Given a subset $\E\subset E(\Gamma)$ and a divisor $D$ on $\Gamma$, we define $D^\E$ as the divisor on $\Gamma^\E$ such that 
\[
D^\E(v)=\begin{cases}
         \begin{array}{ll}
				  D(v),&\text{ if }v\in V(\Gamma);\\
					0,&\text{ if }v\in V(\Gamma^\E)\setminus V(\Gamma).
					\end{array}
					\end{cases}
\]
We also define a divisor $D_\E$ on $\Gamma_\E$ as $D_\E(v):=D(v)$, for every $v\in V(\Gamma_\E)=V(\Gamma)$.\par

	A \emph{pseudo-divisor} on $\Gamma$ is a pair $(\E,D)$ where $\E\subset E(\Gamma)$ and $D$ is a divisor on $\Gamma^\E$ such that $D(v)=-1$ for every exceptional vertex $v\in V(\Gamma^\E)$. If $\E=\emptyset$, then $(\E,D)$ is just a divisor of $\Gamma$. Since every divisor $D$ on $\Gamma^\E$ can be lifted to a divisor on $\Gamma^{(2)}$, it is equivalent to define a pseudo-divisor on $\Gamma$ as a divisor $D$ on $\Gamma^{(2)}$ such that $D(v)=0,-1$ for every exceptional vertex $v\in V(\Gamma^{(2)})$.\par
	Let $\Gamma$ and $\Gamma'$ be graphs. Given a specialization $\iota\col\Gamma\ra\Gamma'$ and a divisor $D$ on $\Gamma$, we define the divisor $\iota_*(D)$ on $\Gamma'$ taking $v'\in V(\Gamma')$ to 
\[
\iota_*(D)(v'):=\sum_{v\in\iota^{-1}(v')}D(v).
\] 
	We say that a pair $(\Gamma,D)$ \emph{specializes} to a pair $(\Gamma',D')$, where $D$ is a divisor on $\Gamma$ and $D'$ is a divisor on $\Gamma'$, if there is a specialization of graphs $\iota\col\Gamma\ra\Gamma'$ such that $D'=\iota_*(D)$; we denote by $\iota\col(\Gamma,D)\ra(\Gamma',D')$ a specialization of pairs. Given a specialization $\iota\col \Gamma\ra\Gamma'$ and a subset $\E$ of $E(\Gamma)$, there exists an induced specialization $\iota^\E\col \Gamma^\E\to\Gamma'^{\E'}$, where $\E':=\E\cap E(\Gamma')$; in this case, if $(\E,D)$ is a pseudo-divisor on $\Gamma$, we define the pseudo-divisor $\iota_*(\E,D)$ on  $\Gamma'$ as $\iota_*(\E,D):=(\E',\iota_*^\E(D))$.
	Given pseudo-divisors $(\E,D)$ on $\Gamma$ and $(\E',D')$ on $\Gamma'$, we say that $(\Gamma,\E, D)$ \emph{specializes} to $(\Gamma',\E',D')$ if there is a specialization $\iota\col\Gamma\to\Gamma'$ such that $\E'\subset \E\cap E(\Gamma')$, and a specialization $\iota^\E\col \Gamma^\E\ra\Gamma'^{\E'}$ compatible with $\iota$ such that $\iota^\E_*(D)=(D')$. We denote by $\iota\col (\Gamma,\E,D)\ra (\Gamma',\E',D')$ such a specialization. \par

	Let $\overrightarrow{\Gamma}$ be a digraph. A \emph{directed path} on $\overrightarrow{\Gamma}$ is a sequence 
\[
v_1, e_1, v_2,\ldots, e_n, v_{n+1}
\]
 such that $s(e_i)=v_i$ and $t(e_i)=v_{i+1}$ for every $i=1,\dots,n$. A \emph{directed cycle} on $\overrightarrow{\Gamma}$ is a directed path on $\overrightarrow{\Gamma}$ such that $v_{n+1}=v_1$.  A \emph{cycle} on $\overrightarrow{\Gamma}$ is just a cycle on $\Gamma$. 
A \emph{source} (respectively, \emph{sink}) of $\overrightarrow{\Gamma}$ is a vertex in $V(\overrightarrow{\Gamma})$ such that $t(e)\neq v$ (respectively, $s(e)\neq v$) for every $e\in E(\overrightarrow{\Gamma})$. We say that $\overrightarrow{\Gamma}$ is \emph{acyclic} if it has no directed cycles. It is a well known result that every acyclic (finite) digraph has at least a source and a sink.\par

 A \emph{flow} on $\overrightarrow{\Gamma}$ is a function $\phi \col E(\overrightarrow{\Gamma})\to \mathbb{Z}_{\geq 0}$. We say that $\phi$ is \emph{acyclic} if the digraph $\ora{\Gamma}/S$ is acyclic, where $S:=\{e\in E(\ora{\Gamma});\phi(e)=0\}$.  Moreover, we say that a flow $\phi$ is \emph{positive} if $\phi(e)>0$ for all $e\in E(\ora{\Gamma})$. Abusing terminology, we will say that a flow $\phi$ on a graph $\Gamma$ is a pair $(\ora{\Gamma},\phi)$ made by an orientation on $\Gamma$ and a flow $\phi$ on $\ora{\Gamma}$. Given a flow $\phi$ on $\ora{\Gamma}$, we define the divisor $\div(\phi)$ on $\Gamma$ associated to $\phi$ as 
\[
\div(\phi)(v)=\underset{t(e)=v}{\sum_{e\in E(\Gamma)}}\phi(e)-\underset{s(e)=v}{\sum_{e\in E(\Gamma)}}\phi(e).
\]\par
Let $\ora{\Gamma}$ and $\ora{\Gamma'}$ be digraphs. Given a specialization $\iota\col \ora{\Gamma}\ra \ora{\Gamma}'$ and a flow $\phi$ on $\Gamma$, we define $\iota_*(\phi)$ as the flow on $\ora{\Gamma}'$ taking $e\in E(\ora{\Gamma}')\subset E(\ora{\Gamma})$ to $\iota_*(\phi)(e):=\phi(e)$. Given flows $\phi$ on $\ora{\Gamma}$ and $\phi'$ on $\ora{\Gamma'}$, we say that the pair $(\ora{\Gamma},\phi)$ \emph{specializes} to $(\ora{\Gamma}',\phi')$ if there is a specialization $\iota:\ora{\Gamma}\ra \ora{\Gamma}'$ such that $\phi'=\iota_*(\phi)$.\par
	Given a graph $\Gamma$ and a ring $A$, we define 
\[
C_0(\Gamma,A):=\bigoplus_{v\in V(\Gamma)}A\cdot v\quad\text{and}\quad C_1(\Gamma,A):=\bigoplus_{e\in E(\Gamma)}A\cdot e.
\]
Fix a orientation on $\Gamma$, i.e., choose a digraph $\ora{\Gamma}$ with $\Gamma$ as underlying graph. We define the differential operator $d\col C_0(\ora{\Gamma},A)\to C_1(\ora{\Gamma}, A)$ as the linear operator taking a generator $v$ of $C_0(\ora{\Gamma},A)$ to
\[
d(v):=\underset{t(e)=v}{\sum_{e\in{E(\ora{\Gamma})}}}e-\underset{s(e)=v}{\sum_{e\in{E(\ora{\Gamma})}}}e.
\]
The adjoint of $d$ is the linear operator $d^*\col C_1(\ora{\Gamma},A)\to C_0(\ora{\Gamma},A)$ taking a generator $e$ of $C_1(\ora{\Gamma},A)$ to
\[
d^*(e):=t(e)-s(e).
\]

There is a natural identification between $C_0(\Gamma,\mathbb{Z})$ and $\Div(\Gamma)$. The composition $d^*d\col C_0(\Gamma,A)\to C_0(\Gamma,A)$ does not depend on the choice of the orientation. The group of principal divisors on $\Gamma$ is the subgroup $\Prin(\Gamma):=\Im(d^*d)$ of $\Div(\Gamma)$. Given $D,D'\in \Div(\Gamma)$, we say that $D$ is \emph{equivalent} to $D'$ if $D-D'\in \Prin(\Gamma)$.

The space of $1$-cycles of $\ora{\Gamma}$ (over $A$) is defined as $H_1(\ora{\Gamma},A):=\ker(d^*)$. 
 Note that if $\ora{\Gamma}'$ is a refinement of $\ora{\Gamma}$, then there is a canonical isomorphism between $H_1(\ora{\Gamma},A)$ and $H_1(\ora{\Gamma}',A)$. 

Let $\gamma$ be a cycle on $\ora{\Gamma}$. Let $v_1,e_1,v_2\ldots,e_n,v_{n+1}$ be a sequence such that $v_1,\dots,v_{n+1}$ are the vertices of $\gamma$, where $v_{n+1}=v_1$ and $v_i\neq v_j$ if $i\neq j$ and $\{i,j\}\neq\{1,n+1\}$, and where $e_i$ is the edge of $\gamma$ connecting $v_i$ with $v_{i+1}$ for $i=1,\dots,n$. For every $e\in E(\ora{\Gamma})$ we define 
\begin{equation}
\label{eq:gamma}
\gamma(e):=\begin{cases}
            \begin{array}{ll}
						0,&\text{ if } e\neq e_i \text{ for all $i=1,\ldots, n$};\\
						1,&\text{ if } e=e_i \text{ and } s(e)=v_i \text{ for some $i=1,\ldots,n$};\\
						-1,&\text{ if } e=e_i \text{ and } t(e)=v_i \text{ for some $i=1,\ldots,n$}.
						\end{array}
						\end{cases}
\end{equation}
The definition of $\gamma(e)$ depends on the choice of the sequence. However a different choice possibly changes the signs of all $\gamma(e)$, and in what follows a change of sign on all $\gamma(e)$ for a given cycle $\gamma$ is not relevant. So we will implicitly fix a choice of such a sequence for every cycle $\gamma$.\par

\begin{Prop}
\label{prop:flow}
Let $\ora{\Gamma}$ be an acyclic digraph and $D$ be a degree-$0$ divisor on $\Gamma$. Then there are finitely many flows $\phi$ such that $\div(\phi)=D$.
\end{Prop}

\begin{proof}
The proof is by induction on the number of vertices of $\ora{\Gamma}$. If $\ora{\Gamma}$ has only one vertex the result is clear, so we can assume that $\ora{\Gamma}$ has at least two vertices.

Let $v$ be a sink of $\ora{\Gamma}$. If $\phi$ is a flow on $\ora{\Gamma}$ such that $\div(\phi)=D$, then 
\[
D(v)=\underset{t(e)=v}{\sum_{e\in E(\ora{\Gamma})}}\phi(e).
\]
Since $\phi(e)\geq0$, then there is a finite number of possibilities for $(\phi(e))_{t(e)=v}$. For each one of these possibilities $(\phi(e))_{t(e)=v}$, we remove the vertex $v$ together with all edges $e$ such that $t(e)=v$. This gives rise to an acyclic digraph $\ora{\Gamma'}$ with $V(\ora{\Gamma'})$ strictly contained in $V(\ora{\Gamma})$, and to the degree-0 divisor $D'$ on $\ora{\Gamma'}$ taking a vertex $v'\in V(\ora{\Gamma'})$ to
\[
D'(v'):=D(v')-\underset{t(e)=v, s(e)=v'}{\sum_{e\in E(\ora{\Gamma})}}\phi(e).
\]
By the induction hypothesis there are finitely many flows $\phi'$ on $\ora{\Gamma'}$ such that $\div(\phi')=D'$, and hence finitely many ways to complete $\phi'$ to a flow $\phi$ on $\ora{\Gamma}$ such that $\div(\phi)=D$.
\end{proof}

\subsection{Quasistability on graphs}

We introduce the key notion of quasistability for pseudo-divisors on graphs and the poset of quasistable pseudo-divisors. These notions will be crucial to define the universal tropical Jacobian in Section \ref{sec:tropmod}.

Let $\Gamma$ be a graph. Given an integer $d$, a \emph{degree-$d$ polarization} on $\Gamma$ is a function $\mu\col V(\Gamma)\to\R$ such that $\sum_{v\in V(\Gamma)}\mu(v)=d$. 

Let $\mu$ be a degree-$d$ polarization on $\Gamma$. For every subset $V\subset V(\Gamma)$, we set $\mu(V):=\sum_{v\in V}\mu(v)$. 
Given a degree-$d$ divisor $D$ on $\Gamma$, we define 
\[
\beta_D(V):=\deg(D|_V)-\mu(V)+\frac{\delta_V}{2}.
\]
 If $\iota\col\Gamma\ra\Gamma'$ is a specialization, then there is an induced degree-$d$ polarization $\iota_*(\mu)$ on $\Gamma'$ defined as 
\[
\iota_*(\mu)(v'):=\underset{v \in \iota^{-1}(v')}{\sum_{v\in V(\Gamma)}}\mu(v).
\]
 
If $\E\subset E(\Gamma)$ is a nondisconnecting subset of edges, we can define a degree-$(d+|\E|)$ polarization $\mu_\E$ on $\Gamma_\E$ (the graph induced by removing the edges $\E$ in $E(\Gamma)$) by 
\[
\mu_\E(v):=\mu(v)+\frac{1}{2}\val_\E(v).
\]
 Given a subdivision $\Gamma^{\E}$ of $\Gamma$ for some $\E\subset E(\Gamma)$, there is an induced degree-$d$ polarization $\mu^\E$ on $\Gamma^\E$ given as
\[
\mu^\E(v):=
\begin{cases}
\begin{array}{ll}
\mu(v)& \text{if}\;v\in V(\Gamma);\\
0&\text{otherwise.}
\end{array}
\end{cases}
\]

Let $\Gamma$ be a graph and $\mu$ a degree-$d$ polarization on $\Gamma$. A degree-$d$ divisor $D$ on $\Gamma$ is $\mu$-\emph{semistable} if $\beta_D(V)\geq0$ for every $V\subset V(\Gamma)$. Given $v_0\in V(\Gamma)$, we say that a degree-$d$ divisor $D$ on $\Gamma$ is $(v_0,\mu)$-\emph{quasistable} if $\beta_D(V)\geq0$ for every $V\subsetneq V(\Gamma)$, with strict inequality if $v_0\in V$, or, equivalently, interchanging $V$ and $V^c$, $\beta_D(V)\leq\delta_V$, with strict inequality if $v_0\notin V$. 
If $\mu$ is a degree-$d$ polarization on a graph with $1$ leg $(\Gamma,v_0)$, we will call $(\Gamma,v_0,\mu)$ a \emph{degree-$d$ polarized graph with $1$ leg}.

We say that a pseudo-divisor $(\E,D)$ on $\Gamma$  is \emph{$\mu$-semistable} (respectively, $(v_0,\mu)$-\emph{quasistable}) if $D$ is $\mu^\E$-semistable (respectively, $(v_0,\mu^\E)$-quasistable) on $\Gamma^\E$. Clearly every $(v_0,\mu)$-quasistable pseudo-divisor is $\mu$-semistable.\par

Let us recall some important properties of quasistable pseudo-divisors on a graph.

\begin{Prop}
\label{prop:spec}
Let $(\Gamma,v_0,\mu)$ be degree-$d$ polarized graph with $1$ leg. Assume that $\iota\col\Gamma\ra\Gamma'$ is a specialization of graphs. For every pseudo-divisor $(\E,D)$ on $\Gamma$ of  degree-$d$, the following properties hold:
\begin{enumerate}
\item[(i)] if $(\E,D)$ is $(v_0,\mu)$-quasistable then the  pseudo-divisor $\iota_*(\E,D)$ on $\Gamma'$ is $(\iota(v_0),\iota_*(\mu))$-quasistable;
\item[(ii)] $(\E,D)$ is $(v_0,\mu)$-quasistable if and only if $\E$ is nondisconnecting and the divisor $D_\E$ on $\Gamma_\E$ is $(v_0,\mu_\E)$-quasistable.
\end{enumerate}
\end{Prop}

\begin{proof}
See \cite[Proposition 4.6]{AP2}.
\end{proof}

\begin{Rem}
\label{rem:subdivision}
It is easy to see that if a divisor $D$ on $\Gamma^\E$ is $(v_0,\mu^\E)$-quasistable, then $D(v)=0, -1$ for every exceptional vertex $v\in V(\Gamma^\E)$. Moreover if $\widehat{\Gamma}$ is a refinement of $\Gamma$ then $\mu$ induces a polarization $\widehat{\mu}$ in $\widehat{\Gamma}$ given by 
\[
\widehat{\mu}(v)=\begin{cases}
                           \begin{array}{ll}
													  0,&\text{ if $v$ is exceptional;}\\
														\mu(v),& \text{ if $v\in V(\Gamma)$},
														\end{array}
									\end{cases}
									\]
where we view $V(\Gamma)$ as a subset of $V(\widehat{\Gamma})$ via the injection $a\col V(\Gamma)\ra V(\wh{\Gamma})$ induced by the refinement. If a divisor $D$ on $\widehat{\Gamma}$ is $(v_0,\widehat{\mu})$-quasistable, then for every edge $e\in E(\Gamma)$ we have $D(v)=0$ for all but at most one exceptional vertex $v$ over $e$; if such a vertex $v$ over $e$ exists, then $D(v)=-1$. Hence, every $(v_0,\widehat{\mu})$-quasistable divisor $D$ of $\widehat{\Gamma}$ induces a $(v_0,\mu)$-quasistable pseudo-divisor $(\E',D')$ on $\Gamma$.
\end{Rem}

Let $\Pos(\Gamma)$ be the set of $(v_0,\mu)$-quasistable pseudo-divisors on $\Gamma$. Note that $\Pos(\Gamma)$ is a poset where the partial order is $(\E,D)\geq (\E',D')$ if $(\E,D)$ specializes to $(\E',D')$. 
In Figure \ref{fig:poset}, we illustrate the poset $\Pos(\Gamma)$, where $\Gamma$ is a graph with 2 vertices $v_0,v_1$ and 3 edges, and $\mu$ is the trivial polarization of degree 0. 
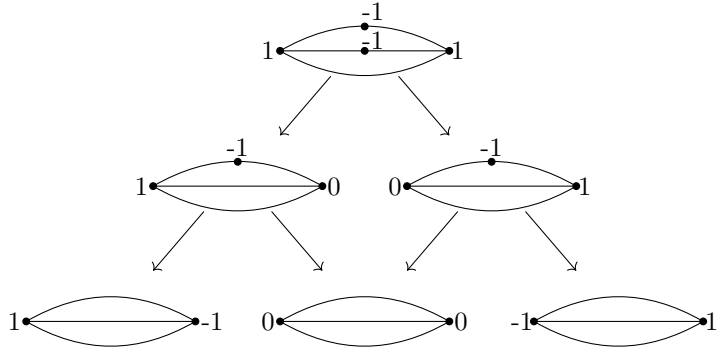
\begin{figure}[h!]
\begin{tikzpicture}[scale=2.25]
\begin{scope}[shift={(0,0)}]
\draw (0,0) to [out=30, in=150] (1,0);
\draw (0,0) to (1,0);
\draw (0,0) to [out=-30, in=-150] (1,0);
\draw[fill] (0,0) circle [radius=0.02];
\draw[fill] (1,0) circle [radius=0.02];
\draw[fill] (0.5,0.144) circle [radius=0.02];
\draw[fill] (0.5,0) circle [radius=0.02];
\node at (-0.07,0) {1};
\node at (1.05,0) {1};
\node at (0.55,0.23) {-1};
\node at (0.55,0.06) {-1};
\end{scope}
\draw[->] (0.3,-0.15) to (0, -0.5);
\draw[->] (0.7,-0.15) to (1, -0.5);
\begin{scope}[shift={(-0.75,-0.8)}]
\draw (0,0) to [out=30, in=150] (1,0);
\draw (0,0) to (1,0);
\draw (0,0) to [out=-30, in=-150] (1,0);
\draw[fill] (0,0) circle [radius=0.02];
\draw[fill] (1,0) circle [radius=0.02];
\draw[fill] (0.5,0.144) circle [radius=0.02];
\node at (-0.07,0) {1};
\node at (1.07,0) {0};
\node at (0.5,0.23) {-1};
\draw[->] (0.3,-0.15) to (0, -0.5);
\draw[->] (0.7,-0.15) to (1, -0.5);
\end{scope}
\begin{scope}[shift={(+0.75,-0.8)}]
\draw (0,0) to [out=30, in=150] (1,0);
\draw (0,0) to (1,0);
\draw (0,0) to [out=-30, in=-150] (1,0);
\draw[fill] (0,0) circle [radius=0.02];
\draw[fill] (1,0) circle [radius=0.02];
\draw[fill] (0.5,0.144) circle [radius=0.02];
\node at (-0.07,0) {0};
\node at (1.05,0) {1};
\node at (0.5,0.23) {-1};
\draw[->] (0.3,-0.15) to (0, -0.5);
\draw[->] (0.7,-0.15) to (1, -0.5);
\end{scope}
\begin{scope}[shift={(-1.5,-1.6)}]
\draw (0,0) to [out=30, in=150] (1,0);
\draw (0,0) to (1,0);
\draw (0,0) to [out=-30, in=-150] (1,0);
\draw[fill] (0,0) circle [radius=0.02];
\draw[fill] (1,0) circle [radius=0.02];
\node at (-0.07,0) {1};
\node at (1.1,0) {-1};
\end{scope}
\begin{scope}[shift={(0,-1.6)}]
\draw (0,0) to [out=30, in=150] (1,0);
\draw (0,0) to (1,0);
\draw (0,0) to [out=-30, in=-150] (1,0);
\draw[fill] (0,0) circle [radius=0.02];
\draw[fill] (1,0) circle [radius=0.02];
\node at (-0.07,0) {0};
\node at (1.07,0) {0};
\end{scope}
\begin{scope}[shift={(1.5,-1.6)}]
\draw (0,0) to [out=30, in=150] (1,0);
\draw (0,0) to (1,0);
\draw (0,0) to [out=-30, in=-150] (1,0);
\draw[fill] (0,0) circle [radius=0.02];
\draw[fill] (1,0) circle [radius=0.02];
\node at (-0.07,0) {-1};
\node at (1.05,0) {1};
\end{scope}
\end{tikzpicture}
\caption{The poset of pseudo-divisors.}
\label{fig:poset}
\end{figure}

We say that $\mu$ is a \emph{universal genus-$g$ polarization (of degree $d$)}, if $\mu$ is a collection of polarizations $\mu_\Gamma$ of degree $d$ for every genus-$g$ weighted stable graph with $1$ leg $\Gamma$, such that $\mu_{\Gamma'}=\iota_*(\mu_\Gamma)$ for every specialization $\iota\col\Gamma\to\Gamma'$.  This polarization extends to every genus-$g$ semistable graph, since every genus-$g$ semistable graph is a subdivision of a stable graph. 

If $\mu$ is a universal genus-$g$ polarization, we define the poset $\mathcal{QD}_{\mu,g}$  as
\[
\mathcal{QD}_{\mu,g}:=\frac{\left\{(\Gamma,\E,D);\begin{array}{c}
                      \Gamma \text{ is a genus-$g$ stable graph with $1$ leg}\\
											(\E,D)\text{ is a }(v_0,\mu_\Gamma)\text{-quasistable pseudo-divisor on }\Gamma\end{array}\right\}}{\sim}
\]
where the ordering is given by specializations, and $(\Gamma,\E,D)\sim (\Gamma',\E',D')$ if there exists an isomorphism $\iota\col\Gamma\ra\Gamma'$ such that $(\E',D')=\iota_*(\E,D)$.

\subsection{Tropical curves}
\label{sec:tropicalcurves}

 A \emph{metric graph} is a pair $(\Gamma,\l)$ where $\Gamma$ is a graph and $\l$ is a function $\l\col E(\Gamma)\to \mathbb{R}_{>0}$, called \emph{length function}. If $\ora{\Gamma}$ is an orientation on $\Gamma$, we define the \emph{tropical curve} $X$ associated to $(\ora{\Gamma},\ell)$ as
\[
X=\frac{\left(\bigcup_{e\in E(\ora{\Gamma})}I_e\cup V(\ora{\Gamma})\right)}{\sim}
\]
where $I_e=[0,\l(e)]\times\{e\}$ and $\sim$ is the equivalence relation generated by $(0,e)\sim s(e)$  and $(\l(e),e)\sim t(e)$. The tropical curve $X$ has a natural topology and its connected components have a natural structure of metric space. Note that the definition of $X$ does not depend on the chosen orientation. For two points $p,q\in I_e$ we denote by $\overline{pq}$ (respectively, $\overrightarrow{pq}$) the interval (respectively, oriented interval) $[p,q]\subset I_e$. \par

  We say that the tropical curves $X$ and $Y$ are \emph{isomorphic} if there is a bijection between $X$ and $Y$ that is an isometry over each connected component of $X$. Moreover, a graph $\Gamma$ (respectively, $\overrightarrow{\Gamma}$) is a  \emph{model} (respectively, \emph{directed model}) of a tropical curve $X$ if there exists a length function $\l$ on $\Gamma$ such that $X$ and the tropical curve associated to $(\Gamma,\l)$ are isomorphic. Sometimes, we will use interchangeably the notion of tropical curve and of its isomorphism class. If $X$ and $Y$ are tropical curves and $\Gamma_X$ and $\Gamma_Y$ are models for $X$ and $Y$, we say that the pairs $(X,\Gamma_X)$ and $(Y,\Gamma_Y)$ are \emph{isomorphic} if there exists an isomorphism $f\col X\to Y$ of tropical curves such that $f(V(\Gamma_X))=V(\Gamma_Y)$ and $f$ induces an isomorphism of graphs between $\Gamma_X$ and $\Gamma_Y$.  \par

  We say that a tropical curve $Y$ is a \emph{tropical subcurve} of $X$ if there is an injection $Y\subset X$ that is an isometry onto its image. In this case, if $\Gamma_Y$ is a model of $Y$, one can choose a model $\Gamma_X$ of $X$ such that $\Gamma_Y$ is a subgraph of $\Gamma_X$. Conversely, if $\Gamma'$ is a subgraph of a model $\Gamma_X$ of $X$, then $\Gamma'$ induces a tropical subcurve $Y$ of $X$. Note that a tropical subcurve can be a single point. If $Y$ and $Z$ are tropical subcurves of $X$ then $Y\cap Z$ and $Y\cup Z$ are also tropical subcurves of $X$. 
  From now on  all tropical curves will be connected (and hence pure-dimensional), while we will allow possibly nonconnected tropical subcurves (and hence possibly non pure dimensional). \par
	
Let $X$ be a tropical curve with a model $\Gamma_X$ and $Y\subset X$ be a tropical subcurve of $X$. Then, there is a minimal refinement $\Gamma_{X,Y}$ of $\Gamma_X$ such that $Y$ is induced by a subgraph $\Gamma_Y$ of $\Gamma_{X,Y}$. We define
\[
\delta_{X,Y}:=\sum_{v\in V(\Gamma_Y)}\val_{E(\Gamma_{X,Y})\setminus E(\Gamma_Y)}(v)
\]
The definition of $\delta_{X,Y}$ does not depend on the choice of the model $\Gamma_X$ of $X$. When no confusion may arise, we will simply write $\delta_Y$ instead of $\delta_{X,Y}$.

  Let $X$ be a tropical curve, $\Gamma_X$ a model of $X$, and $\l$ the induced metric on $\Gamma_X$. A specialization $\iota\col\Gamma_X\ra\Gamma'$ induces a metric $\l'$ on $\Gamma'$ defined as
\[
\l'\colon E(\Gamma')\stackrel{\iota_*}{\hookrightarrow} E(\Gamma)\stackrel{\l}{\to} \R_{>0}.
\]
Let $Y$ be the tropical curve associated to $(\Gamma',\l')$. Then there is an induced function $\iota\col X\to Y$ that is constant on the edges of $\Gamma_X$ contracted by $\iota$. We call this function a \emph{specialization of $X$ to $Y$}.\par

  An \emph{$n$-pointed tropical curve} is a pair $(X,\text{leg})$ where $X$ is a tropical curve and $\text{leg}\col I_n=\{0,\dots,n-1\} \to X$ is a function. For every $p\in X$, we set $\l(p):=\#\text{leg}^{-1}(p)$. A \emph{weight} on a tropical curve $X$ is a function $w\col X\to \mathbb{Z}_{\geq0}$ such that $w(p)\neq 0$ only for finitely many points $p\in X$. 

 An $n$-pointed weighted tropical curve $(X,w,\text{leg})$ is \emph{stable} if $\delta_{X,p}+2w(p)+\l(p)\ge 3$ for every point $p\in X$ such that $\delta_{X,p}\leq1$.
The genus $g(X,w)$ of a weighted tropical curve $(X,w)$ is defined by the formula
\[
2g(X,w)-2:=\sum_{p\in X}(2w(p)-2+\delta_{X,p}).
\]
The sum on right-hand side of the last formula is finite, since $(\delta_{X,p}, w(p))\neq (2,0)$ only for finitely many $p\in X$. We often abuse notation and denote by $X$ the pair $(X,w)$ leaving the weight implicit.

\subsection{Divisors on tropical curves}

  Let $X$ be a tropical curve. A \emph{divisor} on $X$ is a map $\D\col X\to \mathbb{Z}$ such that $\D(p)\neq0$ for finitely many points $p\in X$.  \par

Let $\D$ be a divisor on $X$. The \emph{degree} of $\D$ is the integer $\deg \D:=\sum_{p\in X} \D(p)$. The \emph{support} of $\D$, written $\supp(\D)$, is the set of points $p$ of $X$ such that $\D(p)\neq0$. We say that $\D$ is \emph{effective} if $\D(p)\ge0$ for every $p\in X$. We let $\Div(X)$ be the Abelian group of divisors on $X$ and $\Div^d(X)$ the subset of degree-$d$ divisors on $X$. \par
   Given a weighted tropical curve $(X,w)$, the \emph{canonical divisor} $\omega_X$ of $X$ is defined as $\omega_X(p)=2w(p)+\delta_{X,p}-2$. Clearly, $\deg(\omega_X)=2g(X)-2$.\par
  If $\iota\col X\to Y$ is a specialization (or an inclusion $\iota\col X\hookrightarrow Y$) of tropical curves, we define $\iota_*(\D)$ as 
\[
\iota_*(\D)(q):=\sum_{p\in \iota^{-1}(q)}\D(p)
\]
for every $q\in Y$. Note that $\deg(\D)=\deg(\iota_*(\D))$.\par

Let $\Gamma_X$ be a model of $X$. Then every divisor on $\Gamma_X$ can be seen as a divisor on $X$. Given a divisor $\D$ on $X$, we define $\Gamma_{X,\D}$ as the smallest refinement of $\Gamma_X$ such that $\supp(\D)\subset V(\Gamma_{X,\D})$.   
If $X$ and $Y$ are tropical curves and $\D$ and $\D'$ are divisors on $X$ and $Y$, respectively, we say that the pairs $(X,\D)$ and $(Y, \D')$ are \emph{isomorphic} if there is an isomorphisms $f\col X\to Y$ of tropical curves such that $\D(p)=\D'(f(p))$ for every $p\in X$.

   A \emph{rational function} on $X$ is a continuous, piecewise-linear function $f\col X\ra \R$ with integer slopes. 
  If $(\ora{\Gamma},\l)$ is a directed model of $X$, then for every $e\in E(\ora{\Gamma})$  and $a\in\R$ with $0\le a\le\l(e)$, we let $p_{a,e}$ be the point of $X$ lying on $e$ whose distance from $s(e)$ is $a$. 
 We say that a rational function $f$ on $X$ has \emph{slope $b$ over} $\ora{p_{e,a_1},p_{e,a_2}}$, for $a_1,a_2\in\R$ and $0\leq a_1< a_2\leq \ell(e)$, if the restriction of $f$ to the locus of points $p_{a,e}$ for $a_1\leq a \leq a_2$ is linear and has slope $b$.

	A \emph{principal divisor} on $X$ is a divisor of the form
\[
\div_X(f):=\sum_{p\in X} \ord_p(f) p\in \Div(X),
\]
where $f$ is a rational function on $X$ and $\ord_p(f)$ is the sum of the incoming slopes of $f$ at $p$. A principal divisor has degree zero. The \emph{support} of a rational function $f$ on $X$ is defined as $\supp(f)=\{p\in X;\ord_p(f)\neq 0\}$. 
  We denote by $\Prin(X)$ the subgroup of $\Div(X)$ of principal divisors. Given divisors $\D_1,\D_2\in\Div(X)$, we say that $\D_1$ and $\D_2$ are \emph{equivalent} if $\D_1-\D_2\in \Prin(X)$. \par
 
	Let $f$ be a rational function on the tropical curve $X$ and $\Gamma$ be a model of $X$ such that $\supp(f)\subset V(\Gamma)$. Then $f$ is linear over each edge of $\Gamma$. If $f$ is nowhere constant, then it induces an orientation $\ora{\Gamma}$ on $\Gamma$, such that $f$ has always positive  slopes. In this case, we can define a positive flow $\phi_f$ on $\ora{\Gamma}$ where $\phi_f(e)$ is equal to the slope of $f$ over $e$, for every $e\in E(\ora\Gamma)$. It is clear that $\div_X(f)=\div(\phi_f)$, where $\div(\phi_f)$ is seen as a divisor on $X$. Note that the orientation $\ora{\Gamma}$ on $\Gamma$ induced by a rational function $f$ is acyclic, because there are no strictly increasing functions on the circle $S^1$. If $f$ is constant on a subset $\E\subset E(\Gamma)$, then we can contract all edges in $\E$ and get a specialization $\iota\col X\to Y$ of tropical curves. Clearly, $f$ induces a nowhere constant rational function on $Y$. Hence $f$ induces an acyclic orientation $\ora{\Gamma/\E}$ on $\Gamma/\E$ and a positive flow $\phi_f$ on $\Gamma/\E$. We abuse notation and denote by $\phi_f$ also the nonnegative flow induced by $f$ on $\Gamma$.\par

Every $n$-pointed weighted tropical curve $X$ has a \emph{stable reduction} $\st(X)$, which is a stable $n$-pointed weighted tropical  curve, such that there is a unique specialization $\text{red}_X\col X\to \st(X)$ satisfying the following condition: if $\iota\col X\to Y$ is any specialization, where $Y$ is a stable $n$-pointed weighted tropical curve, then $\iota$ factors through $\text{red}_X$.  Similar to the discussion before Example \ref{exa:stgraph}, we note that $\text{red}_X$ is induced  by a sequence of specializations starting from $X$, in each step contracting a segment of $X$ that has an endpoint of valence $1$, weight $0$ and at most one leg attached to it. In this way, we can view $\st(X)$ as a tropical subcurve of $X$. Let $\iota\col \st(X)\to X$ be the natural inclusion. Note that every divisor $\D$ on $X$ is equivalent to the divisor $\text{red}_{X*}(\D)$, seen as a divisor in $X$ in the natural way via the pushforward via $\iota$.

\begin{Exa}
\label{exa:sttrop}
Let $(X,p_0)$ be the leftmost pointed weighted tropical curve in Figure \ref{fig:sttrop}. Then, the stable reduction $\st(X)$ is the rightmost pointed weighted tropical curve  in Figure \ref{fig:sttrop}. Note that there is a natural specialization $X\to \st(X)$ and $\st(X)$ can be seen as a subcurve of $X$.
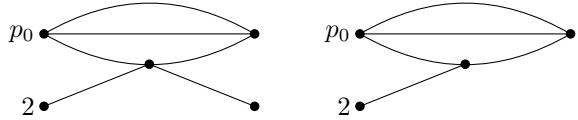
\begin{figure}[h]
\begin{tikzpicture}[scale=2.8]
\begin{scope}[shift={(0,0)}]
\draw (0,0) to [out=30, in=150] (1,0);
\draw (0,0) to (1,0);
\draw (0,0) to [out=-30, in=-150] (1,0);
\draw (0.5,-0.144) to (0,-0.344);
\draw (0.5,-0.144) to (1,-0.344);
\draw[fill] (0,0) circle [radius=0.02];
\draw[fill] (1,0) circle [radius=0.02];
\draw[fill] (0.5,-0.144) circle [radius=0.02];
\draw[fill] (0,-0.344) circle [radius=0.02];
\node[left] at (0,-0.344) {$2$};
\node[left] at (0,0) {$p_0$};
\draw[fill] (1,-0.344) circle [radius=0.02];
\end{scope}
\begin{scope}[shift={(1.5,0)}]
\draw (0,0) to [out=30, in=150] (1,0);
\draw (0,0) to (1,0);
\draw (0,0) to [out=-30, in=-150] (1,0);
\draw (0.5,-0.144) to (0,-0.344);
\draw[fill] (0,0) circle [radius=0.02];
\draw[fill] (1,0) circle [radius=0.02];
\draw[fill] (0.5,-0.144) circle [radius=0.02];
\draw[fill] (0,-0.344) circle [radius=0.02];
\node[left] at (0,-0.344) {$2$};
\node[left] at (0,0) {$p_0$};
\end{scope}
\end{tikzpicture}
\caption{Stable reduction of a pointed weighted tropical curve.}
\label{fig:sttrop}
\end{figure}
\end{Exa}

\subsection{Quasistability on tropical curves}

In analogy with graphs, we introduce the notion of quasistability for a divisor on a tropical curve.

Let $X$ be a tropical curve. A degree-$d$ polarization on $X$ is a function $\mu\col X\to\R$ such that $\mu(p)=0$ for all, but finitely many $p\in X$, with $\sum_{p\in X}\mu(p)=d$. We define the \emph{support} of $\mu$ as 
\[
\supp(\mu):=\{p\in X;\mu(p)\neq 0\}.
\]

   Let $\mu$ be a degree-$d$ polarization on $X$. For every tropical subcurve $Y\subset X$, we define $\mu(Y):=\sum_{p\in Y}\mu(p)$. 
For any divisor $\D$ on $X$ and every tropical subcurve $Y\subset X$, we set 
\[
\beta_\D(Y):=\deg(\D|_Y)-\mu(Y)+\frac{\delta_Y}{2}.
\]
A divisor $\D$ on a tropical curve $X$ is \emph{$\mu$-semistable}  if $\beta_\D(Y)\geq0$, for every tropical subcurve $Y\subset X$. Given a point $p_0$ of $X$, we say that $\D$ is \emph{$(p_0,\mu)$-quasistable} if it is $\mu$-semistable and $\beta_\D(Y)>0$ for every proper tropical subcurve $Y\subset X$ such that $p_0\in Y$.\par

If $\mu$ is a degree-$d$ polarization on $X$ and $p_0$ a point on $X$, 
we will call $(X,p_0,\mu)$ a \emph{degree-$d$ polarized pointed tropical curve}. 

In the next proposition we establish a connection between quasistability on graphs and tropical curves. Note that, given a polarization $\mu$ on a tropical curve $X$ and a model $\Gamma_X$ of $X$ such that $\supp(\mu)\subset V(\Gamma_X)$, we can view $\mu$  as a polarization on $\Gamma_X$.

\begin{Prop}
\label{prop:quasiquasi}
Let $(X,p_0,\mu)$ be a degree-$d$ polarized pointed tropical curve. A degree-$d$ divisor $\D$ on $X$ is $(p_0,\mu)$-quasistable if and only if $D$ is $(p_0,\mu)$-quasistable on $\Gamma_{X,\D}$, where $D$ is the divisor $\D$ seen as a divisor on $\Gamma_{X,\D}$. 
\end{Prop}

\begin{proof}
See \cite[Proposition 5.3]{AP2}.
\end{proof}

Combining  Remark \ref{rem:subdivision} with Proposition \ref{prop:quasiquasi}, we deduce the following result.

\begin{CorDef}
\label{cor:quasiquasi}
Let $\D$ a $(p_0,\mu)$-quasistable degree-$d$ divisor on $X$. Then $\Gamma_{X,\D}$ is an $\E$-subdivision of $\Gamma_X$ for some $\E\subset E(\Gamma_X)$, and the pair $(\E,D)$ is a $(p_0,\mu)$-quasistable degree-$d$ pseudo-divisor on $\Gamma_X$, where $D$ is the divisor $\D$ seen as a divisor on $\Gamma_X^\E$. We call $(\E,D)$ the pseudo-divisor on $\Gamma_X$ induced by $\D$.
\end{CorDef}

\begin{proof}
See \cite[Proposition 5.4]{AP2}.
\end{proof}

The following useful result tells us that in every class of divisors on a tropical curve there is a canonical representative which is $(p_0,\mu)$-quasistable.

\begin{Thm}
\label{thm:quasistable}
Let $(X,p_0,\mu)$ be a degree-$d$ polarized pointed tropical curve. Given a divisor $\D$ on $X$ of degree $d$, there exists a unique degree-$d$ divisor on $X$ equivalent to $\D$ which is $(p_0,\mu)$-quasistable.
\end{Thm}

\begin{proof}
See \cite[Theorem 5.6]{AP2}.
\end{proof}

\section{The universal tropical Abel map}
\label{sec:abeltrop}

\subsection{Tropical moduli spaces}\label{sec:tropmod}
In this section we introduce the tropical moduli spaces which we will be dealing with, focusing on the universal tropical Jacobian.

 Given a finite set $S\subset \R^n$ we define 
\[
\cone(S):=\left\{\sum_{s\in S}\lambda_ss|\lambda_s\in \mathbb R_{\ge0}\right\}.
\]
A subset $\sigma\subset \R^n$ is called a \emph{polyhedral cone} if $\sigma=\cone(S)$ for some finite set $S\subset \mathbb{R}^n$. If there exists $S\subset \mathbb{Z}^n$ with $\sigma=\cone(S)$ then $\sigma$ is called \emph{rational}.\par

Let $\sigma\subset \R^n$ be a polyhedral cone. Then $\sigma$ is the intersection of finitely many closed half-spaces. We denote by $\text{span}(\sigma)$ the minimal linear subspace containing $\sigma$. The \emph{dimension} of $\sigma$, written $\dim(\sigma)$, is the dimension of $\text{span}(\sigma)$.
The \emph{relative interior} $\sigma^\circ$ is the interior of $\sigma$ inside $\text{span}(\sigma)$. A \emph{face} of $\sigma$ is either $\sigma$ itself or the intersection of $\sigma$ with some linear subspace $H\subset \mathbb{R}^n$ of codimension one such that $\sigma$ is contained in one of the closed half-spaces determined by $H$. A face of $\sigma$ is also a polyhedral cone. If $\tau$ is a face of $\sigma$, then we write $\tau\prec \sigma$. In the sequel we will use the word \emph{cone} to mean rational polyhedral cone. \par

	 A \emph{morphism} $f\col \tau\to\sigma$ between cones $\tau\subset \R^m$ and $\sigma\subset \R^n$ is the restriction to $\tau$ of an integral linear transformation $T\col\R^n\to \R^m$ such that $T(\tau)\subset\sigma$. 
We say that $f$ is an \emph{isomorphism} if there exists an inverse morphism $f^{-1}\col\sigma\to\tau$.  A morphism $f\col\tau\to\sigma$ is called a \emph{face morphism} if $f$ is an isomorphism between $\tau$ and a (not necessarily proper) face  of $\sigma$.\par

	\begin{Rem}
	\label{rem:primitive}
	If the integral linear transformation $T$ is injective and primitive, i.e., $T^{-1}(\mathbb{Z}^m)=\mathbb{Z}^n$, and $f$ is a surjection, then $f$ is an isomorphism. Indeed, since $f$ is a bijection, the image of the restriction of $T$ to $\spa(\tau)$ is precisely $\spa(\sigma)$. Moreover $T|_{\spa(\tau)}\col\spa(\tau)\to\spa(\sigma)$ is an integral isomorphism, because $T$ is primitive, hence it admits an integral inverse $T'\col\spa(\sigma)\to\spa(\tau)$ with $T'(\sigma)=\tau$. Upon choosing a basis $e_1,\ldots, e_m$ of $\mathbb{Z}^m$ such that $e_1,\ldots,e_k$ is a basis of $\mathbb{Z}^m\cap\spa(\sigma)$, we can extend the map $T'$ to $\mathbb{R}^m$ by $T'(e_j)=0$ for every $j=k+1,\ldots,m$. This gives rise to the inverse $f^{-1}$.\end{Rem}

 A \emph{fan} $\Sigma$ is a set of cones such that:
\begin{enumerate}
\item if $\sigma\in \Sigma$ and $\tau\prec\sigma$, then $\tau\in\Sigma$;
\item if $\sigma,\tau\in \Sigma$, then $\sigma\cap\tau$ is a face of both $\sigma$ and $\tau$.
\end{enumerate}
	
	A \emph{generalized cone complex} $\Sigma$ is the colimit (as a topological space) of a finite diagram $\mathbf{D}$ of cones with face morphisms. 
A cone $\sigma$ in a  finite diagram of cones $\mathbf{D}$ is \emph{maximal} if there is no proper face morphism $f\col\sigma\to\tau$ in $\mathbf{D}$. 

Let $\Sigma$ and $\Sigma'$ be  generalized cone complexes.
	 A \emph{morphism of generalized cone complexes} is a continuous map of topological spaces $f\col\Sigma\to\Sigma'$ such that for every cone $\sigma\in \mathbf{D}$ there exists a cone $\sigma'\in \mathbf{D}'$ such that the induced map $\sigma\to\Sigma'$ factors through a cone morphism $\sigma\to\sigma'$. In fact,  for a continuous map $f\col\Sigma\to\Sigma'$ to be a morphism of generalized cone complexes it suffices that for every maximal cone $\sigma\in \mathbf{D}$ there exists a cone $\sigma'\in \mathbf{D}'$ such that the induced map $\sigma\to\Sigma'$ factors through a cone morphism $\sigma\to\sigma'$.
\par
		
  For a graph $\Gamma$, the open cone $\mathbb{R}_{> 0}^{|E(\Gamma)|}$ parametrizes all possible choices for the lengths of the edges of $\Gamma$. Hence, $M_\Gamma^{\trop}:=\R^{E(\Gamma)}_{>0}/\Aut(\Gamma)$ parametrizes isomorphism classes of pairs $(X,\Gamma_X)$, where $X$ is a tropical curve and $\Gamma_X$ is a model of $X$ isomorphic to $\Gamma$.  We will identify $E(\Gamma)$ with the canonical basis of $\mathbb{R}^{|E(\Gamma)|}$.

 If $\iota\col \Gamma\ra \Gamma'$ is a specialization, then there is an inclusion $\R^{E(\Gamma')}\subset \R^{E(\Gamma)}$ induced by the inclusion $E(\Gamma')\subset E(\Gamma)$. Given a refinement $\Gamma'$ of $\Gamma$, there is a map 
\begin{align}\label{eq:hat}
F\col\R^{E(\Gamma')}\to  &  \R^{E(\Gamma)}\\
      (x_{e'})_{{e'}\in E(\Gamma')}  \mapsto&(y_e)_{e\in E(\Gamma)}\nonumber,
\end{align}
where, if $b\col E(\Gamma')\to E(\Gamma)$ is the surjection induced by the refinement $\Gamma'$, 
\[
y_e:=\underset{b(e')=e}{\sum_{e'\in E(\Gamma')}}x_{e'}.
\] 

		If $X$ is a stable $n$-pointed tropical curve of genus-$g$, then $X$ admits exactly one stable model $\Gamma_X$. Hence the moduli space $M_{g,n}^{\text{trop}}$ of stable $n$-pointed tropical curves of genus $g$ is the generalized cone complex given as the colimit of the diagram whose cones are $\mathbb{R}_{\geq0}^{|E(\Gamma)|}$, where $\Gamma$ runs through all stable genus-$g$ weighted graphs with $n$ legs, with face morphisms specified by specializations. More precisely, if $\Gamma$ specializes to $\Gamma'$, then $\mathbb{R}_{\geq 0}^{|E(\Gamma')|}$ is a face of $\mathbb{R}_{\geq 0}^{|E(\Gamma)|}$ via the inclusion $E(\Gamma')\subset E(\Gamma)$.  For more details about $M_{g,n+1}^{\text{trop}}$ and its compactification $\overline{M}_{g,n+1}^{\textnormal{trop}}$, see \cite[Section 2]{M},  \cite[Sections 2.1 and 3.2]{BMV}, \cite[Section 3]{Caporaso1}, \cite[Section 3]{Caporaso}, and \cite[Section 4]{ACP}.\par

If $(X,w,\text{leg})$ is a stable $n$-pointed weighted tropical curve, then there is a unique stable weighted graph with $n$ legs $\Gamma_{st}$ that is a model of $X$ with  weights and legs  satisfying the following property:  for every $p\in X$ such that either $w(p)\neq0$ or $\l(p)\neq0$, then $p\in V(\Gamma_{st})$. We call $\Gamma_{st}$ the \emph{stable model} of $X$.\par
 For the remainder of the paper we will usually omit to denote the weight $w$ of a weighted tropical curve.

\begin{Def}
Let $(\Gamma,v_0,\mu)$ be a degree-$d$ polarized graph with $1$ leg.  For each $(v_0,\mu)$-quasistable pseudo-divisor $(\E,D)$ on $\Gamma$, we define
\[
\sigma_{(\Gamma,\E,D)}:=\mathbb{R}^{E(\Gamma^\E)}_{\geq0} \quad \text{and}\quad \sigma^\circ_{(\Gamma,\E,D)}:=\mathbb{R}^{E(\Gamma^\E)}_{>0}.
\]
If $\iota\col(\Gamma,\E,D)\ra(\Gamma',\E',D')$ is a specialization, then there is a natural inclusion $\sigma_{(\Gamma',\E',D')}\to \sigma_{(\Gamma,\E,D)}$.  
For a fixed $(\Gamma,v_0,\mu)$, we define the generalized cone complex
\[
J^\trop_{\Gamma,\mu}:=\lim_{\longrightarrow}\sigma_{(\Gamma',\E',D')},
\]
where $(\Gamma',\E',D')$ runs through all specializations $\iota\col\Gamma\ra\Gamma'$ and  $(\iota(v_0),\iota_*(\mu))$-quasistable pseudo-divisors $(\E',D')$ on $\Gamma'$. We note that every $(\iota(v_0),\iota_*(\mu))$-quasistable pseudo divisor $(\E',D')$ on $\Gamma'$ is the specialization of a $(v_0,\mu)$-quasistable pseudo-divisor on $\Gamma$ (see \cite[Proposition 4.10]{AP1}).\par
 Let $\mu$ be a universal genus-$g$ polarization. The \emph{universal tropical Jacobian (with respect to $\mu$)} is defined as the generalized cone complex 
\[
J^\trop_{\mu,g}:=\lim_{\longrightarrow} \sigma_{(\Gamma,\E,D)}=\coprod_{(\Gamma,\E,D)} \sigma_{(\Gamma,\E,D)}^\circ/\text{Aut}(\Gamma,\E,D)
\] 
where $(\Gamma,\E,D)$ runs through all elements of $\mathcal{QD}_{\mu,g}$. 
\end{Def}

In the next result we collect some properties of the universal tropical Jacobian. 

\begin{Thm}
\label{thm:maintrop} 
The generalized cone complex $J^\trop_{\mu,g}$ has pure dimension $4g-2$ and is connected in codimension $1$. The natural forgetful map $\pi^{trop}\col J^\trop_{\mu,g}\to M^\trop_{g,1}$ is a morphism of generalized cone complexes. For every stable pointed tropical curve $X$ of genus $g$, there is a homeomorphism
\[
(\pi^{trop})^{-1}([X])\cong J^\trop(X)/\Aut(X),
\]
where $J^\trop(X)$ is the tropical Jacobian of $X$.
\end{Thm}

\begin{proof}
See \cite[Theorem 5.14]{AP2}.
\end{proof}
	
\subsection{The universal tropical Abel map}\label{sesc:Abeltrop}

We introduce the universal tropical Abel map and study its properties. The universal tropical Abel map is not, in general, a morphism of generalized cone complexes. We will provide its ``resolution" within the category of generalized cone complexes.

Fix a nonnegative integer $n$, a sequence $\mathcal A=(a_0, a_1,\ldots,a_n,m)$ of $n+2$ integers and a universal genus-$g$ polarization $\mu$ of degree $d$ over $M_{g,1}^{\trop}$, where 
\[
d:=\sum_{0\le i\le n}a_i+m(2g-2).
\] 

The \emph{universal tropical Abel map} is defined as
\begin{align}
\label{eq:atrop}\alpha^\trop_{\mathcal A,\mu}\col M_{g,n+1}^{trop}\to& J^\trop_{\mu,g}\\
           (X,p_0,\ldots,p_n)\mapsto& (\st(X,p_0), \D) \label{eq:Abeltrop}\nonumber,
\end{align}
where $\st(X,p_0)$ is the stable reduction of $(X,p_0)$ and, if $\omega_X$ is the canonical divisor of $X$, then $\D$ is the unique $(p_0,\mu)$-quasistable divisor on $\st(X)$ equivalent to  
\[
\red_{X,*}(m\omega_X+\sum_{0\le i\le n} a_ip_i).
\]
(Recall Theorem \ref{thm:quasistable} and that $\st(X)$ can be viewed as a subcurve of $X$ and then $\D$ can be viewed as a divisor on $X$, which is equivalent to $m\omega_X+\sum_{0\le i\le n} a_ip_i$.)\par

In general, this map is not a morphism of cone complexes. We want to construct a refinement $\beta_{\mathcal A,\mu}^\trop\col M_{g,\mathcal A}^{\trop}\to M_{g,n+1}^{\trop}$ such that the composition $\alpha_{\mathcal A,\mu}^\trop\circ\beta_{\mathcal A,\mu}^\trop\col M_{g,\mathcal A}^{trop}\ra J_{\mu,g}^{trop}$ is a morphism of generalized cone complexes.\par
  Fix a degree-$d$ polarized graph with $1$ leg $(\Gamma,v_0,\mu)$. Consider a degree-$d$ divisor $D_0$ on $\Gamma$. For each $(v_0,\mu)$-quasistable pseudo-divisor $(\E,D)$ of degree $d$ on $\Gamma$ there exists an orientation $\ora{\Gamma^\E}$ on $\Gamma^\E$ and a flow $\phi$ on $\ora{\Gamma^\E}$ such that $D_0^\E+\div(\phi)=D$. Indeed, $D_0^\E$ and $D$ have the same degree and the image of the map $d^*\col C_1(\Gamma,\mathbb{Z})\to C_0(\Gamma,\mathbb{Z})$ consists precisely of the divisors of degree $0$.   By Proposition \ref{prop:flow}, there is  a finite number of pairs $(\E,\phi)$ with $\phi$ an acyclic flow on $\Gamma^\E$ such that $D_0^\E+\div(\phi)$ is $(v_0,\mu)$-quasistable.\par

\begin{Def}\label{def:CGamma}
For each pair $(\E,\phi)$ where $\phi$ is an acyclic flow on $\Gamma^\E$, we define the cone $C_{\Gamma,\E,\phi}\subset\R_{\geq0}^{E(\Gamma^\E)}$ (respectively, $C_{\Gamma,\E,\phi}^\circ\subset\R_{>0}^{E(\Gamma^\E)}$) as the cone (respectively, the open cone) given, in the coordinates $(x_e)_{e\in E(\Gamma^\E)}$ of $\R^{E(\Gamma^\E)}$, by intersecting $\R^{E(\Gamma^\E)}_{\geq0}$ (respectively, $\R^{E(\Gamma^\E)}_{>0}$) with the linear subspace given by the following equations: 
\begin{equation}
\label{eq:C}
\sum_{e\in E(\Gamma^\E)} \gamma(e)\phi(e)x_e=0
\end{equation}
where $\gamma$ runs over the set of cycles in $\Gamma$ or, equivalently, over the elements of a basis of $H_1(\Gamma,\mathbb{Z})$ (recall the definition \eqref{eq:gamma} of $\gamma(e)$ for $e\in E(\Gamma^\E)$).  
\end{Def}

 Note that there is a natural bijection between the cycles in $\Gamma$ and the cycles of $\Gamma^\E$, taking a cycle $\gamma$ in $\Gamma$ to its subdivision induced by $\E$. In the sequel, we will identify the cycles in $\Gamma$ and the cycles in $\Gamma^\E$.

\begin{Rem}
\label{rem:functionflow}
The cone $C_{\Gamma,\E,\phi}^\circ$ parametrizes tropical curves $X$ with model $\Gamma$ such that $X$ admits a rational function $f$ whose induced flow $\phi_f$  is equal to the flow $\phi$ on $\Gamma^\E$.
\end{Rem}

\begin{Not}\label{not:KGamma}
We let $K_{\Gamma,\E,\phi}$ (respectively, $K^\circ_{\Gamma,\E,\phi}$) be the cone obtained as the image of $C_{\Gamma,\E,\phi}$ (respectively, $C^\circ_{\Gamma,\E,\phi}$) via the map $F\col \R_{\geq0}^{E(\Gamma^\E)}\to \R_{\geq0}^{E(\Gamma)}$  in   \eqref{eq:hat}.  
\end{Not}

We have a natural map 
\[
 C_{\Gamma,\E,\phi}\to K_{\Gamma,\E,\phi}.
\] 

\begin{Prop}
\label{prop:isoCK}
If $\E\subset E(\Gamma)$ is nondisconnecting and $\phi$ is an acyclic flow on $\Gamma^\E$ such that $\div(\phi)(v)=-1$ for every exceptional vertex $v\in V(\Gamma^\E)$, then the  induced map $C_{\Gamma,\E,\phi}\to K_{\Gamma,\E,\phi}$ is an isomorphism of cones. 
\end{Prop}

\begin{proof}
If $\E=\emptyset$, then the result is clear, hence we can assume $\E\ne\emptyset$. \par
 Fix a spanning tree $T$ not intersecting $\E$. Hence, for each $e'\in \E$, there exists a unique cycle $\gamma_{e'}$ on $\Gamma$ such that $E(\gamma_{e'})\setminus T=\{e'\}$. Let $H$ be the linear subspace of $\R^{E(\Gamma^\E)}$ given in the coordinates $(x_e)_{e\in E(\Gamma^\E)}$ of $\R^{E(\Gamma^\E)}$ by the equations
\[
\sum_{e\in E(\Gamma^\E)} \gamma_{e'}(e)\phi(e)x_e=0
\]
for every $e'\in \E$. We will prove that the restriction $F|_H\col H\to\R^{E(\Gamma)}$ is an integral isomorphism.\par

Consider a point $(x_e)_{e\in E(\Gamma^\E)}\in H$ and write $F((x_e)_{e\in E(\Gamma^\E)})=(u_{e'})_{e'\in E(\Gamma)}$, with  $(u_{e'})_{e'\in E(\Gamma)}\in \mathbb{R}^{E(\Gamma)}$. Fix an edge $e_0$ of $\Gamma^\E$.  Since $\Gamma^\E$ is a subdivision of $\Gamma$, the natural surjection $b\col E(\Gamma^\E)\to E(\Gamma)$ satisfies $|b^{-1}(e')|=1$ for every $e'\in E(\Gamma)\setminus \E$, and $|b^{-1}(e')|=2$ for every $e'\in \E$. So, if $b(e_0)\notin \E$, then $b^{-1}(b(e_0))=\{e_0\}$, and hence $x_{e_0}=u_{b(e_0)}$. On the other hand, if $b(e_0)\in \E$, we consider $\gamma:=\gamma_{b(e_0)}$ (the cycle defined above), and we let $b^{-1}(b(e_0))=\{e_0,e_1\}$. Consider the equation
\[
\sum_{e\in E(\Gamma^\E)\setminus b^{-1}(\E)} \gamma(e)\phi(e)x_e+\sum_{e\in b^{-1}(\E)}\gamma(e)\phi(e)x_e=0.
\]
Recall that $x_e=u_{b(e)}$ for $e\notin b^{-1}(\E)$. Then
\begin{equation}\label{eq:sum}
\sum_{e\in b^{-1}(\E)}\gamma(e)\phi(e)x_e=-\sum_{e\in E(\Gamma^\E)\setminus b^{-1}(\E)}\gamma(e)\phi(e)u_{b(e)}.
\end{equation}
By the definition of $\gamma$, we have $\gamma(e)=0$ for every $e\in E(\Gamma^\E)$ such that $b(e)\in \E\setminus\{e_0\}$, thus \eqref{eq:sum} translates into 
\[
\gamma(e_0)\phi(e_0)x_{e_0}+\gamma(e_1)\phi(e_1)x_{e_1}=-\sum_{e\in E(\Gamma^\E)\setminus b^{-1}(\E)}\gamma(e)\phi(e)u_{b(e)}.
\]
Note that $x_{e_0}+x_{e_1}=u_{b(e_0)}$ because  $F((x_e)_{e\in E(\Gamma^\E)})=(u_{e'})_{e'\in E(\Gamma)}$, hence
\[
(\gamma(e_0)\phi(e_0)-\gamma(e_1)\phi(e_1))x_{e_0}=-\gamma(e_1)\phi(e_1)u_{b(e_0)}-\sum_{e\in E(\Gamma^\E)\setminus b^{-1}(\E)}\gamma(e)\phi(e)u_{b(e)}.
\]
However, since $\div(\phi)$ has degree $-1$  on the exceptional vertices, it follows that  
\[
\gamma(e_0)\phi(e_0)-\gamma(e_1)\phi(e_1)\in\{1,-1\},
\]
and we get 
\[
x_{e_0}=\pm\left(\gamma(e_1)\phi(e_1)u_{b(e_0)}+\sum_{e\in E(\Gamma^\E)\setminus b^{-1}(\E)}\gamma(e)\phi(e)u_{b(e)}\right),
\]
and $x_{e_1}=u_{b(e_0)}-x_{e_0}$. This gives rise to the inverse $(F|_H)^{-1}\col\mathbb{R}^n\to H$ as an integral linear map. Since $K_{\Gamma,\E,\phi}=F(C_{\Gamma,\E,\phi})$, the cones $K_{\Gamma,\E,\phi}$ and $C_{\Gamma,\E,\phi}$ are isomorphic by Remark \ref{rem:primitive}.
\end{proof}

Let $C_{\Gamma,\E,\phi}$ and $C_{\Gamma',\E',\phi'}$ be cones as in Definition \ref{def:CGamma}. Assume that 
   $\iota\col(\Gamma^\E,\phi)\ra (\Gamma'^{\E'},\phi')$ is a specialization. Then there is a natural inclusion $C_{\Gamma',\E',\phi'}\subset C_{\Gamma,\E,\phi}$ induced by the inclusion $\R^{E(\Gamma'^{\E'})}\subset \R^{E(\Gamma^\E)}$. Since the diagram
\[
\begin{CD}
\R^{E(\Gamma'^{\E'})}@>>> \R^{E(\Gamma^\E)}\\
@VVV @VVV\\
\R^{E(\Gamma')}@>>>\R^{E(\Gamma)}
\end{CD}
\]
is commutative, we also have an induced inclusion $K_{\Gamma',\E',\phi'}\subset K_{\Gamma,\E,\phi}$.

  Let $(\Gamma,v_0,\mu)$ be a degree-$d$ polarized graph with one leg. For a degree-$d$ divisor $D$ on $\Gamma$, we let $\A_\Gamma(D)$ be the set whose elements are pairs $(\E,\phi)$ with $\E$ a nondisconnecting subset of $E(\Gamma)$ and $\phi$ an acyclic flow on $\Gamma^\E$ such that the pseudodivisor $(\E,D^\E+\div(\phi))$ is $(v_0,\mu)$-quasistable on $\Gamma$; we call the elements of $\A_\Gamma(D)$ the \emph{admissible pairs (with respect to $D$)}.

 We note that if $\iota\col(\Gamma,\E,\phi)\ra (\Gamma',\E',\phi')$ is a specialization with $\phi$ and $\phi'$ acyclic and $(\E,\phi)\in \A_\Gamma(D)$, then $(\E',\phi')\in \A_{\Gamma'}(\iota_*(D))$.

\begin{Prop}
\label{prop:union}
Let $(\Gamma,v_0,\mu)$ be a degree-$d$ polarized graph with one leg  and $D_0$ be a degree-$d$ divisor on $\Gamma$. There is a partition
\[
\R_{>0}^{E(\Gamma)}=\coprod_{(\E,\phi)\in \A_\Gamma(D_0)} K^\circ_{\Gamma,\E,\phi}.
\]
 Moreover, if $(\E,\phi)\in \A_\Gamma(D_0)$ then
\[
K_{\Gamma,\E,\phi}=\coprod K^\circ_{\Gamma',\E',\phi'}
\]
where the union runs through all specializations $(\Gamma,\E,\phi)\ra (\Gamma',\E',\phi')$ with $\phi'$ acyclic.

\end{Prop}
\begin{proof}
Let us prove the existence of the partition. Let $(x_e)$ be an element of $\R_{>0}^{E(\Gamma)}$. Then $(x_e)$ induces a tropical curve $X$ associated to the metric graph $(\Gamma,\l)$ with metric $\l$ such that $\l(e)=x_e$ for every $e\in E(\Gamma)$. Note that $D_0$ induces a divisor on $X$, which we will call $\D_0$. By Theorem \ref{thm:quasistable}, there is a $(v_0,\mu)$-quasistable divisor $\D'$ on $X$ such that $\D'-\D_0=\div(f)$ for some rational function $f$ on $X$. By Corollary \ref{cor:quasiquasi}, the divisor $\D'$ induces an $\E$-subdivision of $\Gamma$ for some $\E\subset E(\Gamma)$ such that $(\E,D')$ is a $(v_0,\mu^\E)$-quasistable pseudo-divisor on $\Gamma$, where $D'$ is the divisor on $\Gamma^\E$ induced by $\D'$. Moreover, since $\supp(f)\subset \supp(D_0^\E)\cup\supp(D')$, the rational function $f$ induces an acyclic flow $\phi_f$ on $\Gamma^\E$ such that $D_0^\E+\div(\phi_f)=D'$. This proves that $(x_e)\in K_{\Gamma,\E,\phi_f}$, and hence $\R_{>0}^{E(\Gamma)}$ is the union of the cones $K^\circ_{\Gamma,\E,\phi}$. By \cite[Remark 3.2]{AP2}, there is a unique rational function $f$ (up to scalars) such that $\D'-\D_0=\div(f)$. This means that $(x_e)$ belongs only to $K_{\Gamma,\E,\phi_f}$, and hence there is a partition as stated.\par

  We now prove the second statement. We begin by showing that if $(\Gamma'_1,\E'_1,\phi'_1)$ and $(\Gamma'_2,\E'_2,\phi'_2)$ are distinct then $K^\circ_{\Gamma'_1,\E'_1,\phi'_1}\cap K^\circ_{\Gamma'_2,\E'_2,\phi'_2}=\emptyset$. This is clear if $\Gamma_1'\neq \Gamma_2'$, because, in this case,
\[
K^\circ_{\Gamma'_i,\E'_i,\phi'_i}\subset \mathbb{R}^{E(\Gamma'_i)}_{>0}\subset \mathbb{R}^{E(\Gamma)}_{\geq0}
\]
for $i=1,2$, and $\R^{E(\Gamma'_1)}_{>0}\cap \R^{E(\Gamma'_2)}_{>0}=\emptyset$ (when viewed as subsets of $\R^{E(\Gamma)}$). On the other hand, if $\Gamma'_1=\Gamma'_2=:\Gamma'$, then $K^\circ_{\Gamma',\E'_1,\phi'_1}\subset \R^{E(\Gamma')}$ and $K^\circ_{\Gamma',\E'_2,\phi'_2}\subset \R^{E(\Gamma')}$, hence the first statement implies that $K^\circ_{\Gamma',\E'_1,\phi'_1}\cap K^\circ_{\Gamma',\E'_2,\phi'_2}=\emptyset$.\par

  There remains to show that, if $(x_e)_{e\in E(\Gamma)}\in K_{\Gamma,\E,\phi}$, then $(x_e)_{e\in E(\Gamma)}\in K^\circ_{\Gamma',\E',\phi'}$ for some specialization $(\Gamma,\E,\phi)\ra(\Gamma',\E',\phi')$ as in the statement. By Proposition \ref{prop:isoCK}, the induced map $C_{\Gamma,\E,\phi}\ra K_{\Gamma,\E,\phi}$ is bijective, hence there exists a unique point $(y_e)_{e\in E(\Gamma^\E)}$ in $C_{\Gamma,\E,\phi}$ over $(x_e)_{e\in E(\Gamma)}$. Let $(\Gamma',\E',\phi')$ be the specialization of $(\Gamma,\E,\phi)$ contracting the edges $e$ such that $y_e=0$. Then $(y_e)_{e\in E(\Gamma')}\in C^\circ_{\Gamma',\E',\phi'}$, which means that $(x_e)_{e\in E(\Gamma)}$ belongs to the image of $C^\circ_{\Gamma',\E',\phi'}$ which is $K^\circ_{\Gamma',\E',\phi'}$.
\end{proof}

Given a degree-$d$ polarized graph $(\Gamma,v_0,\mu)$ with $1$ leg and a degree-$d$ divisor $D_0$ on $\Gamma$, we define 
\begin{equation}\label{def:fan}
\Sigma_{\Gamma,D_0,\mu}:=\{K_{\Gamma',\E',\phi'}\}
\end{equation}
where $(\Gamma',\E',\phi')$ runs through all specializations $\iota\col \Gamma\ra \Gamma'$ and admissible pairs $(\E',\phi')\in \A_{\Gamma'}(\iota_*(D_0))$. We now prove that $\Sigma_{\Gamma,D_0,\mu}$ is a fan.

\begin{Thm}
\label{thm:fan}
Let $(\Gamma,v_0,\mu)$ be a degree-$d$ polarized graph with $1$ leg and $D_0$ be a degree-$d$ divisor of $\Gamma$.  If $\iota\col (\Gamma,\E,\phi)\ra (\Gamma',\E',\phi')$ is a specialization, then $K_{\Gamma',\E',\phi'}$ is a face of $K_{\Gamma,\E,\phi}$, and conversely, every face of $K_{\Gamma,\E,\phi}$ arises in this way. In particular, 
 $\Sigma_{\Gamma,D_0,\mu}$ is a fan in $\R^{E(\Gamma)}$ whose support is precisely $\R^{E(\Gamma)}_{\geq0}$.
\end{Thm}

\begin{proof}
The fact that $\Sigma_{\Gamma,D_0,\mu}$ is a fan follows by the first statement and Proposition \ref{prop:union}.
We note that, by Proposition \ref{prop:isoCK}, the fact that $K_{\Gamma',\E',\phi'}$ is a face of $K_{\Gamma,\E,\phi}$ is equivalent to $C_{\Gamma',\E',\phi'}$ being a face of $C_{\Gamma,\E,\phi}$. If $\iota$ is an isomorphism, then $C_{\Gamma',\E',\phi'}=C_{\Gamma,\E,\phi}$, and in particular $C_{\Gamma',\E',\phi'}$ is a face of $C_{\Gamma,\E,\phi}$. Otherwise, we have that $E(\Gamma'^{\E'})\subsetneq E(\Gamma^{\E})$ and we consider the hyperplane $H\subset \R^{E(\Gamma^\E)}$ defined in the coordinates $(x_e)_{e\in E(\Gamma^\E)}$ as
\[
H:\sum_{e\in E(\Gamma^\E)\setminus E(\Gamma'^{\E'})}x_e=0.
\]
Note that $C_{\Gamma,\E,\phi}$ lies in one of the two half-spaces determined by $H$. Moreover the intersection $H\cap C_{\Gamma,\E,\phi}$ is precisely the set of points $(x_e)_{e\in E(\Gamma^\E)}\in C_{\Gamma,\E,\phi}$ such that $x_e=0$ for every $e\in E(\Gamma^\E)\setminus E(\Gamma'^{\E'})$. But, setting $x_e=0$ for every $e\in E(\Gamma^\E)\setminus E(\Gamma'^{\E'})$ in \eqref{eq:C}, we  get the equations of $C_{\Gamma',\E',\phi'}$, so $C_{\Gamma',\E',\phi'}$ is a face of $C_{\Gamma,\E,\phi}$. Finally, using Proposition \ref{prop:union} we deduce that every face of $K_{\Gamma',\E',\phi'}$ arises in this way.
\end{proof}

\begin{Exa}
\label{exa:fan}
Let $\Gamma$ be a graph with $2$ vertices $v_0, v_1$ and $3$ edges $e_0,e_1,e_2$ connecting $v_0,v_1$, as in Figure \ref{fig:poset}. Let $\mu$ be the degree-$0$ polarization given by $\mu(v)=0$ for every $v \in V(\Gamma)$ and $D_0$ be the divisor $D_0(v_0)=4,  D_0(v_1)=-4$.
In Figure \ref{fig:fan} we depict a section of the fan $\Sigma_{\Gamma,D_0,\mu}$.
\begin{figure}[h]
\begin{tikzpicture}[scale=4]
\coordinate (a) at ($2*(0,0)$);
\coordinate (b) at ($2*(1,0)$);
\coordinate (c) at ($2*(0.5,0.866025)$);
\coordinate (d) at ($1/5*(a)+2/5*(b)+2/5*(c)$);
\coordinate (e) at ($2/5*(a)+1/5*(b)+2/5*(c)$);
\coordinate (f) at ($2/5*(a)+2/5*(b)+1/5*(c)$);
\coordinate (g) at ($1/3*(a)+1/3*(b)+1/3*(c)$);
\coordinate (h) at ($1/4*(a)+1/4*(b)+1/2*(c)$);
\coordinate (i) at ($1/4*(a)+1/2*(b)+1/4*(c)$);
\coordinate (j) at ($1/2*(a)+1/4*(b)+1/4*(c)$);
\coordinate (k) at ($1/7*(a)+3/7*(b)+3/7*(c)$);
\coordinate (l) at ($3/7*(a)+1/7*(b)+3/7*(c)$);
\coordinate (m) at ($3/7*(a)+3/7*(b)+1/7*(c)$);
\draw (a) to (b);
\draw (a) to (c);
\draw (c) to (b);
\draw[fill] (a) circle [radius=0.01];
\draw[fill] (b) circle [radius=0.01];
\draw[fill] (c) circle [radius=0.01];
\draw[fill] (d) circle [radius=0.01];
\draw[fill] (e) circle [radius=0.01];
\draw[fill] (f) circle [radius=0.01];
\draw[fill] (g) circle [radius=0.01];
\draw[fill] (h) circle [radius=0.01];
\draw[fill] (i) circle [radius=0.01];
\draw[fill] (j) circle [radius=0.01];
\draw[fill] (k) circle [radius=0.01];
\draw[fill] (l) circle [radius=0.01];
\draw[fill] (m) circle [radius=0.01];
\draw (a)-- (j);
\draw (a)-- (l);
\draw (a)-- (m);
\draw (a)-- (e);
\draw (a)-- (f);
\draw (b)-- (i);
\draw (b)-- (k);
\draw (b)-- (m);
\draw (b)-- (d);
\draw (b)-- (f);
\draw (c)-- (h);
\draw (c)-- (l);
\draw (c)-- (k);
\draw (c)-- (e);
\draw (c)-- (d);
\draw (d)-- (k);
\draw (e)-- (l);
\draw (f)-- (m);
\draw (g)-- (d);
\draw (g)-- (e);
\draw (g)-- (f);
\draw (d)-- (h);
\draw (d)-- (i);
\draw (e)-- (h);
\draw (e)-- (j);
\draw (f)-- (j);
\draw (f)-- (i);
\draw[fill, gray] (g)--(d)--(h)--(e)--(g);
\end{tikzpicture}
\caption{A section of the fan $\Sigma_{\Gamma,D_0,\mu}$.}
\label{fig:fan}
\end{figure}
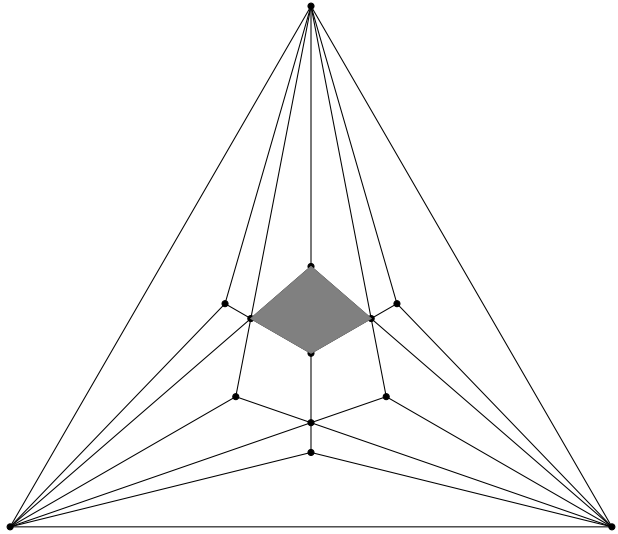

Let $\E=\{e_0,e_1\}$ and $\phi$ be the flow in Figure \ref{fig:flow}. Then $D:=D_{0}^\E+\div(\phi)$ is $(v_0,\mu)$ quasistable (in fact, $D$ is the uppermost divisor in Figure \ref{fig:poset}). Consider the basis for $H_1(\Gamma,\mathbb{Z})$ given by $e_0-e_2$ and $e_1-e_2$ (keeping the orientation as in Figure \ref{fig:flow}).
\begin{figure}[h]
\begin{tikzpicture}[scale=3]
\draw[decorate, decoration={markings, mark=at position 0.25 with {\arrow{>}}, mark=at position 0.75 with {\arrow{>}}}]   (0,0) to [out=30, in=150] (1,0);
\draw (0,0) to [out=30, in=150] (1,0);
\draw[decorate, decoration={markings, mark=at position 0.5 with {\arrow{>}}}] (0,0) to (0.5,0);
\draw (0,0) to (0.5,0);
\draw (0.5,0) to (1,0);
\draw[decorate, decoration={markings, mark=at position 0.5 with {\arrow{>}}}] (0.5,0) to (1,0);
\draw (0,0) to [out=-30, in=-150] (1,0);
\draw[decorate, decoration={markings, mark=at position 0.5 with {\arrow{>}}}] (0,0) to [out=-30, in=-150] (1,0);
\draw[fill] (0,0) circle [radius=0.02];
\draw[fill] (1,0) circle [radius=0.02];
\draw[fill] (0.5,0.144) circle [radius=0.02];
\draw[fill] (0.5,0) circle [radius=0.02];
\node[above] at (0.25,0.1) {1};
\node[below] at (0.25,0.01) {\tiny{1}};
\node[above] at (0.75,0.1) {2};
\node[below] at (0.75,0.01) {\tiny{2}};
\node[below] at (0.5,-0.144) {1};
\node[below] at (0,0) {$v_0$};
\node[below] at (1,0) {$v_1$};
\end{tikzpicture}
\caption{The flow $\phi$.}
\label{fig:flow}
\end{figure}
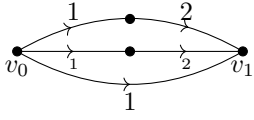

 We denote by $(z_0,z_1,z_2)$ the coordinates of $\R^{E(\Gamma)}$, where $z_i$ is associated with $e_i$ for $i=0,1,2$, and by $(x_0,y_0,x_1,y_1,z_2)$ the coordinates of $\R^{E(\Gamma^\E)}$, where $x_i$ (respectively, $y_i$) is associated with the edge over $e_i$ with flow $1$ (respectively, flow $2$) for $i=0,1$. 
  Then the cone $C_{\Gamma,\E,\phi}\subset \R^{E(\Gamma^\E)}_{\geq0}$ is given by the equations
\[
\begin{cases}
x_0+2y_0-z_2=0;\\
x_1+2y_1-z_2=0,
\end{cases}
\]
and hence the cone $K_{\Gamma,\E,\phi}\subset \R^{E(\Gamma)}$ is given by inequalities
\[
\begin{cases}
0\leq z_2-z_0\leq z_0;\\
0\leq z_2-z_1\leq z_1.
\end{cases}
\]
The cone $K_{\Gamma,\E,\phi}$ has extremal rays $(1,1,1), (1,2,2), (2,1,2), (1,1,2)$ and  the gray quadrilateral in Figure \ref{fig:fan} is a section of $K_{\Gamma,\E,\phi}$. The faces of $K_{\Gamma,\E,\phi}$ are obtained by taking all the specializations of $(\Gamma,\E,\phi)$ contracting some of the edges in Figure \ref{fig:flow}.
\end{Exa}
  
  Let $(\Gamma,v_0,\mu)$ be a degree-$d$ polarized graph with $1$ leg. Given a subdivision $\widehat{\Gamma}$ of $\Gamma$ and a degree-$d$ divisor $D_0$ on $\widehat{\Gamma}$, there is a continuous map 
\begin{align*}
\alpha_{D_0,\mu}^\trop\col \R^{E(\widehat{\Gamma})}_{\geq0}\to&  J^\trop_{\Gamma,\mu}\\
          (x_e)_{e\in E(\widehat{\Gamma})} \mapsto& (X,\D'),
\end{align*}
where $X$ is the tropical curve with model $\widehat{\Gamma}/\widehat{S}$, with 
\[
\widehat{S}:=\{e\in E(\widehat\Gamma) ; x_e=0\},
\]
 and lengths $x_e$, for $e\in E(\widehat{\Gamma})\setminus \wh S$,  and $\D'$ is the unique $(v_0,\mu)$-quasistable divisor on $X$ equivalent to $\D_0$ (see Theorem \ref{thm:quasistable}). Here, as usual, we let $\D_0$ be the divisor on $X$ induced by $\iota_*(D_0)$, where $\iota\col \widehat{\Gamma}\to \widehat{\Gamma}/\widehat{S}$ is the natural specialization. Note that a model for $X$ is also $\Gamma/S$, where 
\[
S:=\{e\in E(\Gamma);x_{e'}=0\text{ for every }e'\in E(\widehat{\Gamma})\text{ over }e\},
\]
hence, in fact, $(X,\D')$ corresponds to a point in $J_{\Gamma,\mu}^\trop$.\par

 The map $\alpha_{D_0,\mu}^\trop$ is rarely a morphism of generalized cone complexes. For instance, each cone in $\Sigma_{\Gamma,D_0,\mu}$ in Example \ref{exa:fan} maps to a cone $\sigma_{(\Gamma,\E,D)}$ and their images are not all contained in the same maximal cone of $J^\trop_{\Gamma,\mu}$. However,  since $\Sigma_{\wh \Gamma,D_0,\mu}$ is a refinement of $\R^{E(\wh\Gamma)}$, there is a natural morphism of generalized cone complexes 
\[
\beta^\trop_{D_0,\mu}\col \Sigma_{\wh \Gamma,D_0,\mu}\to\R^{E(\wh\Gamma)}_{\geq0}.
\] 
Next we prove that $\alpha_{D_0,\mu}^\trop\circ\beta^\trop_{D_0,\mu}$ is in fact a morphism of generalized cone complexes. 

\begin{Thm}
\label{thm:map}
Let $(\Gamma,v_0,\mu)$ be a degree-$d$ polarized graph with $1$ leg  and $D_0$ a degree-$d$ divisor on 
 a subdivision $\widehat{\Gamma}$ of $\Gamma$. Then the map 
\[
\alpha^\trop_{D_0,\mu}\circ\beta^\trop_{D_0,\mu}\col\Sigma_{\widehat{\Gamma},D_0,\mu}\to J^{trop}_{\Gamma,\mu}.
\]
is a morphism of generalized cone complexes.
\end{Thm}
\begin{proof}
We set $\alpha:=\alpha^\trop_{D_0,\mu}$ and $\beta:=\beta^\trop_{D_0,\mu}$. Let $K_{\widehat{\Gamma}',\E,\phi}$ be a cone of $\Sigma_{\widehat{\Gamma},D_0,\mu}$. We need to show that the restriction of $\alpha\circ\beta$ to $K_{\widehat{\Gamma}',\E,\phi}$ factors as the composition of the natural map $\sigma_{(\Gamma',\E',D')}\to J^{\trop}_{\Gamma,\mu}$ and a cone morphism $K_{\widehat{\Gamma}',\E,\phi}\to\sigma_{(\Gamma',\E',D')}$.\par

By the definition of $\Sigma_{\widehat{\Gamma},D_0,\mu}$ the cone $K_{\widehat{\Gamma}',\E,\phi}$ is associated to a specialization $\widehat{\iota}\col\widehat{\Gamma}\ra\widehat{\Gamma}'$.  Moreover, if we set $D:=\widehat{\iota}_*(D_0)^\E+\div(\phi)$, we know that the pseudo-divisor $(\E,D)$ on $\widehat{\Gamma}'$  is $(\wh \iota(v_0),\wh \iota_*(\widehat{\mu}))$-quasistable. Let $\iota\col\Gamma\ra\Gamma'$ be the specialization of $\Gamma$ such that an edge $e\in E(\Gamma)$ is contracted by $\iota$ if and only if all edges in $\widehat{\Gamma}$ over $e$ are contracted by $\widehat{\iota}$. Since the $\E$-subdivision $\widehat{\Gamma}'^\E$ of $\wh{\Gamma}'$ is also a refinement of $\Gamma'$, we can consider the $(\iota(v_0),\iota_*(\mu))$-quasistable pseudo-divisor $(\E',D')$ on $\Gamma'$ induced by $(\E,D)$ as described in  Remark \ref{rem:subdivision}.\par

Let us prove that the restriction of $\alpha\circ\beta$   to $K_{\widehat{\Gamma}',\E,\phi}$ is the composition of the cone morphisms
\begin{equation}
\label{eq:KCR}
K_{\widehat{\Gamma}',\E,\phi}\to C_{\widehat{\Gamma}',\E,\phi}\subset \R^{E(\widehat{\Gamma}'^{\E})}_{\geq0}\stackrel{F}{\rightarrow} \R^{E(\Gamma'^{\E'})}_{\geq0}=\sigma_{(\Gamma',\E',D')}(\to J^\trop_{\Gamma,\mu}),
\end{equation}
where the first map is the isomorphism of Proposition \ref{prop:isoCK} and $F$ is the map  in  \eqref{eq:hat}. 

By Remark \ref{rem:functionflow} and Propositions \ref{prop:isoCK} and \ref{prop:union}, every point in $K_{\widehat{\Gamma}',\E,\phi}$ corresponds to a tropical curve $X$, and $\phi$ induces a rational function $f$ on $X$ (possibly after a specialization). If $X$ is in $K^\circ_{\widehat{\Gamma}',\E,\phi}$, then $\widehat{\Gamma}'$ and $\Gamma'$ are models for $X$ and there are divisors $\D_0$, $\D$ and $\D'$ on $X$ induced respectively by $\widehat{\iota}_*(D_0)$, $D$ and $D'$.  By construction, we obtain
\begin{equation}\label{eq:DDprime}
\D'=\D=\D_0+\div(f).
\end{equation}
 Since $(\E,D)$ is $(v_0,\mu)$-quasistable on $\widehat{\Gamma}'$, it follows from Proposition \ref{prop:quasiquasi} that $\D$, and hence $\D'$, are $(p_0,\mu)$-quasistable divisors on $X$, where $p_0\in X$ is the point corresponding to $v_0$. 
On the other hand, if $X$ is in the boundary of $K_{\widehat{\Gamma}',\E,\phi}$, then $X$ is a specialization of a curve in the interior $K^\circ_{\widehat{\Gamma}',\E,\phi}$. Since quasistability is preserved under pushforwards of specializations (see Propositions \ref{prop:spec} and \ref{prop:quasiquasi}), we can still define $\D'$, $\D$ and $\D_0$ as above and apply the same reasoning to prove that \eqref{eq:DDprime} still holds. Therefore, we conclude that
\[
\alpha\circ\beta (X)=(X,\D)=(X,\D')\in \sigma_{(\Gamma',\E',D')},
\]
which is what we wanted.
\end{proof}

\begin{Def}
Fix nonnegative integers $g$ and $n$.
 Let $\mathcal A=(a_0,\dots,a_n,m)$ be a sequence of $n+2$ integers. 
Let $\mu$ be a genus-$g$ universal  polarization of degree  $d:=m(2g-2)+\sum_{0\le i\le n}a_i$. For every genus-$g$ weighted graph $\Gamma$ with $n+1$ legs, we consider the divisor on $\Gamma$:
\begin{equation}\label{eq:DGamma}
D_{\Gamma,\mathcal A}:=m\omega_{\Gamma}+\sum_{0\le i\le n} a_i\text{leg}_\Gamma(i).
\end{equation}
where $\omega_\Gamma$ is the canonical divisor of $\Gamma$. Recall that if $\iota\col (\Gamma,\E,\phi)\to (\Gamma',\E',\phi')$ is a specialization, then there exists an induced cone morphism $K_{\Gamma',\E',\phi'}\to K_{\Gamma,\E,\phi}$.  We define the generalized cone complex
\[
M_{\mathcal A,\mu}^{\trop}:=\lim_{\longrightarrow} K_{\Gamma,\E,\phi}
\]
where $(\Gamma,\E,\phi)$ runs through all genus-$g$ stable weighted graphs $\Gamma$ with $n+1$ legs and an admissible pair $(\E,\phi)\in \A_\Gamma(D_{\Gamma,\mathcal A})$.
\end{Def}
 
There is a natural forgetful morphism of generalized cone complexes
\[
\beta^\trop_{\mathcal A,\mu}\col M_{\mathcal A,\mu}^{\trop}\to M_{g,n+1}^{\trop}
\]
given by the composition $K_{\Gamma,\E,\phi}\subset \R^{E(\Gamma)}_{\geq0}\to M_{\Gamma}^{\trop}$.
Recall the definition \eqref{eq:atrop} of universal tropical Abel map $\alpha^\trop_{\mathcal A,\mu}\col M_{g, n+1}^{\trop}\to  J^\trop_{\mu,g}$.

\begin{Thm}
\label{thm:abeltrop}
The map $\alpha^\trop_{\mathcal A,\mu}\circ \beta_{\mathcal A,\mu}^\trop\col M_{\mathcal A,\mu}^{\trop}\to J^\trop_{\mu,g}$ is a morphism of generalized cone complexes.
\end{Thm}
\begin{proof}
Let $\Gamma$ be a genus-$g$ stable weighted graph with $n+1$ legs. Consider the divisor $D_0:=D_{\Gamma,\mathcal A}$ on $\Gamma$ defined in \eqref{eq:DGamma}. By forgetting the last $n$ legs of $\Gamma$, we get a graph of genus $g$ with $1$ leg which, abusing notations, we still call $\Gamma$. By the universal property satisfied by the stable model $\st(\Gamma)$, there exists a unique refinement $\widehat{\Gamma}$ of $\st(\Gamma)$ endowed with a specialization $\text{red}_\Gamma\col\Gamma\ra\widehat{\Gamma}$. Recall that the edges contracted by $\text{red}_\Gamma$ are separating, hence  no cycle of $\Gamma$ contains an edge contracted by $\text{red}_\Gamma$. By the equations \eqref{eq:C} defining $C_{\Gamma,\E,\phi}$, we obtain
\[
C_{\Gamma,\E,\phi}=C_{\widehat{\Gamma},\E,\text{red}_{\Gamma*}(\phi)}\times \R^{E(\Gamma)\setminus E(\widehat{\Gamma})}_{\geq0},
\]
and hence, by Proposition \ref{prop:isoCK},
\[
K_{\Gamma,\E,\phi}=K_{\widehat{\Gamma},\E,\text{red}_{\Gamma*}(\phi)}\times \R^{E(\Gamma)\setminus E(\widehat{\Gamma})}_{\geq0}.
\]

We now follow the notations and arguments of the proof of Theorem \ref{thm:map}, in particular Equation \eqref{eq:KCR} (with $\widehat{\Gamma}$ and $\st(\Gamma)$ instead of $\widehat{\Gamma}'$ and $\Gamma'$). We have that the cone $K_{\widehat{\Gamma},\E,\text{red}_{\Gamma*}(\phi)}$ maps to the cone $\sigma_{(st(\Gamma),\E',D')}$  via a cone morphism $K_{\widehat{\Gamma},\E,\text{red}_{\Gamma*}(\phi)}\ra\sigma_{(st(\Gamma),\E',D')}$, where $(\E',D')$ is the pseudo-divisor on $\st(\Gamma)$ induced by the pseudo-divisor $(\E,\text{red}_{\Gamma*}(D_0)+\div(\phi))$ on $\widehat{\Gamma}$. Since $(\E,\text{red}_{\Gamma*}(D_0)+\div(\phi))$ is $(\iota(v_0),\iota_*(\mu))$-quasistable then so is $(\E',D')$, which means that $\sigma(\st(\Gamma),\E',D')$ is a cone of $J_{g,\mu}^\trop$. Hence $K_{\Gamma,\E,\phi}$ maps to $\sigma_{(st(\Gamma),\E',D')}$ via the cone morphism which is the composition of  the projection $K_{\Gamma,\E,\phi}\ra K_{\widehat{\Gamma},\E,\text{red}_{\Gamma,*}(\phi)}$ and the cone morphism $K_{\widehat{\Gamma},\E,\text{red}_{\Gamma*}(\phi)}\ra \sigma_{(st(\Gamma),\E',D')}$. 
\end{proof}

   \section{Toric varieties and binomial ideals}\label{sec:torbin}
	
	\subsection{Toric varieties}\label{sec:toricsec}

In this section we discuss some results and constructions in toric geometry.  
	We refer to \cite{CLS} for the basic results. \par
	A \emph{toric variety} $X$ of dimension $n$ is a normal variety with an embedding $(k^*)^n\hookrightarrow X$ of the $n$-dimensional torus $(k^*)^n$ together with an action of $(k^*)^n$ on $X$ that preserves the natural action of $(k^*)^n$ on itself. A map of toric varieties $X\to X'$ is called \emph{toric} if it respects the torus actions on both $X$ and $X'$.\par

Let $N\cong\mathbb{Z}^n$ be a lattice and $M:=\Hom(N,\mathbb{Z})$ be its dual. We let $N_\R:=N\otimes\R\cong \R^n$ and $M_\R:=M\otimes\R=\Hom(N,\R)$. Given a strictly convex rational polyhedral cone (or, simply, a cone) $\sigma$, we define 
\[
\sigma^\vee:=\{u\in M_\R;u(v)\geq0\text{ for every }v\in \sigma\}.
\]
We consider the finitely generated monoid $S_\sigma:=\sigma^\vee\cap M$. We associate to $\sigma$ the finitely generated $k$-algebra
\begin{equation}\label{eq:Asigma}
A_\sigma:=k[S_\sigma]=\bigoplus_{u\in S_\sigma}k\cdot \chi^u 
\end{equation}
where the multiplication is given by $\chi^u\cdot\chi^{v}=\chi^{u+v}$, and we define the affine toric variety $U_\sigma:=\Spec(A_\sigma)$.  If $\tau$ is a face of $\sigma$, then there is a natural inclusion $S_\sigma\subset S_\tau$ which makes the natural map $\lambda\col A_\sigma\to A_\tau$ a localization homomorphism with respect to the multiplicative system 
\begin{equation}\label{eq:multsys}
\{\chi^u; u\in S_\sigma,u(v)=0 \text{ for every } v\in\tau\}.
\end{equation}
In fact, by \cite[Proposition 1.3.16]{CLS}, we can choose an element $u\in S_\sigma$, such that $S_\tau=S_\sigma+\mathbb{Z}u$, and in this case $A_\tau=(A_\sigma)_{\chi^{u}}$. The element $u$ can be chosen as any element such that $\tau=\{v\in\sigma; u(v)=0\}$. Hence the induced morphism $U_\tau\to U_\sigma$ is an open immersion, in particular $U_\tau$ is a principal open subscheme of $U_\sigma$. More generally if $\tau\to\sigma$ is cone morphism, then there exists an induced ring homomorphism $\lambda\col A_\sigma\to A_\tau$, giving rise to a toric morphism
\begin{equation}
\label{eq:mortor}
U_\tau\to U_\sigma.
\end{equation}\par

For each fan $\Sigma$ we can associate a scheme $\Tor(\Sigma)$ , called the \emph{toric variety associated to} $\Sigma$, given by glueing the affine toric varieties $U_\sigma$ for $\sigma\in\Sigma$. The scheme $\Tor(\Sigma)$ is covered by the affine open subsets $U_\sigma$. Moreover each cone $\sigma\in \Sigma$ is associated to an orbit $O(\sigma)\subset\Tor(\Sigma)$; these orbits induce a stratification of $\Tor(\Sigma)$ dual to the natural stratification of $\Sigma$. For each cone $\sigma\in\Sigma$ there is a closed invariant subvariety  $V(\sigma):=\overline{O(\sigma)}\subset\Tor(\Sigma)$ such that $\codim(V(\sigma))=\dim(\sigma)$. In particular, every one-dimensional cone $\sigma\in\Sigma$ corresponds to a Weil divisor on $\Tor(\Sigma)$. We also note that the ideal $I(\sigma)$ of $V(\sigma)$ in $U_\sigma$ is 
\[
I(\sigma)=\langle \chi^u; u\in S_\sigma, u\notin \sigma^\perp \rangle.
\]
 Moreover, if $\eta_\sigma$ is the generic point of $V(\sigma)$, then the local ring of $\Tor(\Sigma)$ at $\eta_\sigma$ is
\begin{equation}
\label{eq:localsigma}
\O_{\Tor(\Sigma),\eta_\sigma}=\O_{U_\sigma,\eta_\sigma}=(A_\sigma)_{I(\sigma)}.
\end{equation}

 We denote by $\Sigma(1)$ the set of one-dimensional cones of $\Sigma$ or, abusing notation, the sets of its primitive generators. In particular, if $\sigma$ is a cone, we write $\sigma(1)$ for the set of its extremal rays.\par

  Let $N$ and $N'$ be lattices and $F\col N\to N'$ be a linear transformation. If $\Sigma$ and $\Sigma'$ are fans in $N_\R$ and $N'_\R$, respectively, such that for every $\sigma\in \Sigma$ there exists $\sigma'\in\Sigma'$ with $F(\sigma)\subset \sigma'$, then there is an induced toric morphism $\Tor(\Sigma)\to\Tor(\Sigma')$ locally given as in \eqref{eq:mortor}.\par

   If $A$ is a local Noetherian $k$-algebra, there is an inclusion
\[
k[z_1,\ldots,z_n]\to A[[z_1,\ldots,z_n]]
\]
which induces a morphism 
\[
S_{A,n}:=\Spec(A[[z_1,\ldots,z_n]])\to \Spec(k[z_1,\ldots,z_n])=\mathbb{A}^n_k.
\]

If $\Sigma$ is a fan in $\R^n$ whose cones are contained in $\R^n_{\geq0}$, then there is a toric morphism
$\Tor(\Sigma)\to \mathbb{A}^n_k$. 
In this case, for every local Noetherian $k$-algebra $A$, we define the scheme 
\begin{equation}\label{eq:TVA}
\Tor_{A}(\Sigma):=S_{A,n}\times_{\mathbb{A}^n_k}\Tor(\Sigma).
\end{equation}

\begin{Rem}
\label{rem:flat}
Note that $\Tor_A(\Sigma)$ is flat over $\Tor(\Sigma)$, because $A[[z_1,\ldots,z_n]]$ is a flat $k[z_1,\ldots,z_n]$-algebra and flatness is stable under base change (or, equivalently, under tensor products).
\end{Rem}

\begin{Cons}
\label{cons:xy}
We introduce a useful construction which we use in Section \ref{sec:combalg}.

Given a cone $\sigma\subset N_{\R}\cong \mathbb R^n$ and an element $\ol{u}\in S_\sigma$, define the hyperplane $H\subset N_{\R}\times \R^2$ by $\ol{u}=u_1+u_2$, where $\{u_1,u_2\}$ is the dual basis of $(\R^2)^\vee$ induced by the canonical basis of $\R^2$. Define the cone $\tau\subset H$ as 
\[
\tau:=(\sigma\times\R^2_{\geq0})\cap H.
\]
 We obtain the following diagram:
\[
\begin{tikzcd}
\tau \arrow[hook]{r}\arrow[hook]{d}& \sigma\times \R^2_{\geq0}\arrow{r}\arrow[hook]{d}& \sigma \arrow[hook]{d}\\
H \arrow[hook]{r} & N_\R\times \R^2 \arrow{r}& \R^n.
\end{tikzcd}
\]
Dualizing, we get a diagram:
\[
\begin{tikzcd}
\tau^\vee\arrow[hook]{d} &\arrow{l} \sigma^\vee\times (\R^2_{\geq0})^\vee \arrow[hook]{d}& \sigma^\vee \arrow[hook']{l}\arrow[hook]{d}\\
H^\vee &\arrow{l} (N_\R)^\vee\times (\R^2)^\vee &\arrow[hook']{l} (\R^n)^\vee.
\end{tikzcd}
\]
Now, $S_\tau=\tau^\vee\cap ((N_\R\times \mathbb{Z}^2)\cap H)^\vee$ is generated by (the image of) $S_\sigma$ and (the image of) $u_1$ and $u_2$. Since $\ol{u}=u_1+u_2$, we get 
\[
S_\tau\cong\frac{S_\sigma\oplus \mathbb{Z}_{\geq0} u_1\oplus \mathbb{Z}_{\geq0} u_2}{\ol{u}-u_1-u_2}.
\]
This implies that 
\begin{equation}\label{eq:Atau}
A_{\tau}\cong\frac{A_\sigma[x,y]}{\langle xy-\chi^{\ol{u}}\rangle},
\end{equation}
 where $x$ and $y$ are identified with $\chi^{u_1}$ and $\chi^{u_2}$.
 Moreover, if $r$ is an extremal ray of $\sigma$, then the preimage of the cone  $\langle r\rangle$ in $\tau$ under the induced map $\tau\to\sigma$, is the cone generated by $(r,0,v(r))$ and $(r,v(r),0)$. We denote this cone by $\tau'$. \par
   Let $u\in S_\sigma$ such that $r=\{v\in \sigma;u(v)=0\}$. By \cite[Proposition 1.3.16]{CLS}, we have $A_r=(A_\sigma)_{\chi^{u}}$. Since $\tau'$ is the preimage of $r$, it follows that $\tau'=\{\ol{v}\in \tau;u(\ol{v})=0\}$ (viewing $u$ in $S_\tau$). Again, by \cite[Proposition 1.3.16]{CLS}, we have $A_{\tau'}=(A_\tau)_{\chi^{u}}$. We will use this description in terms of localization in the proof of Theorem \ref{thm:mainlocal}. 
 \end{Cons}

   \subsection{Monomial and binomial ideals}

We recall and prove some results about monomial and binomial ideals.

A \emph{binomial} in a polynomial ring $k[x_1,\ldots,x_n]$ is a polynomial which has at most two terms, i.e., it is of the form $a x^\alpha+bx^\beta$ for $a,b\in k$ and $\alpha,\beta\in \mathbb{Z}^{n}_{\geq0}$. A \emph{binomial ideal} is an ideal generated by binomials. 

By \cite[Proposition 1.1.9]{CLS}, every ideal defining an affine toric variety is a binomial ideal. In other words, if $\sigma$ is a cone, there is a surjective homomorphism $k[y_1,\ldots, y_s]\to A_{\sigma}$ whose kernel is a binomial ideal.\par

\begin{Prop}
\label{prop:ES1.6}
If $B, I_1,\ldots, I_s$ are ideals in $k[x_1,\ldots,x_n]$ such that $B$ is a binomial ideal and $I_1,\ldots, I_s$ are monomial ideals, then
\begin{enumerate}
\item $(B+I_1)\cap\ldots\cap(B+I_s)$ is generated by monomials modulo $B$;
\item any monomial in $B+I_1+\ldots+I_s$ lies in the ideal $B+I_i$ for some $i$. In particular, if $m,m_1,\ldots,m_s$ are monomials  in $k[x_1,\ldots,x_n]$ and $m\in B+(m_1,\ldots,m_s)$, then $m\in B+(m_i)$ for some $i$.
\end{enumerate}
\end{Prop}

\begin{proof}
See \cite[Corollary 1.6]{ES}.
\end{proof}

In particular, given a cone $\sigma$, the intersection of monomial ideals in $A_\sigma$ is a monomial ideal.

\begin{Prop}
\label{prop:ES1.7}
Let $B\subset k[x_1,\ldots,x_n]$ be a binomial ideal, $m_1,\ldots,m_t$ monomials, and $f_1,\ldots, f_t$ polynomials such that $\sum_{i=1}^t f_im_i \in B$. Let $f_{i,j}$ denote the terms of $f_i$. For each term $f_{i,j}$, either $f_{i,j}m_i\in B$ or there is a term $f_{i',j'}$, distinct from $f_{i,j}$, and a scalar $a\in k$ such that $f_{i,j}m_i+af_{i',j'}m_{i'}\in B$.
\end{Prop}

\begin{proof}
See \cite[Corollary 1.7]{ES}.
\end{proof}

\begin{Prop}
\label{prop:ES1.10}
Let $B$ and $I$ be, respectively, a binomial and a monomial ideal ideal in the ring $k[x_1,\ldots,x_n]$. If $f\in B+I$ and $f'$ is the sum of the terms of $f$ that are not individually contained in $B+I$, then $f'\in B$. 
\end{Prop}

\begin{proof}
See \cite[Proposition 1.10]{ES}.
\end{proof}

\begin{Cor}
\label{cor:sigmamon}
If $\sigma$ is a cone and $I$ is a monomial ideal in $A_\sigma$, then $f\in I$ if and only if each term of $f$ is in $I$.
\end{Cor}
\begin{proof}
One direction is clear. Consider $f=\sum_{m\in \Lambda} a_m\chi^m\in I$, with $a_m\in k^*$ and $\Lambda\subset S_\sigma$ a finite subset. Let $B$ be the binomial ideal of $k[y_1,\ldots, y_s]$ such that
\[
A_\sigma=\frac{k[y_1,\ldots, y_s]}{B}.
\] 
For each $m$, choose $u_m=(u_1,\ldots, u_s)\in \mathbb{Z}^s_{\geq0}$ such that $\chi^m=[y^{u_m}]$, where $y^{u_m}:=y_1^{u_1}\cdots y_s^{u_s}$. Let $I'=\langle y^{u_m}; \chi^m\in I\rangle$, so that $I=(B+I')/B$. Since $f=\sum_{m\in \Lambda} a_m [y^{u_m}]$, we have $\sum_{m\in \Lambda} a_my^{u_m}\in B+I'$. Let $\Lambda':=\{m\in \Lambda; y^{u_m}\notin B+I'\}$. By Proposition \ref{prop:ES1.10}, we have $\sum_{m\in \Lambda'} a_my^{u_m}\in B$. If $\Lambda'\neq\emptyset$, then by Proposition \ref{prop:ES1.7} for each $m\in \Lambda'$ there exists $m'\in \Lambda'\setminus\{m\}$ and a scalar $a\in k$ such that $y^{u_m}-ay^{u_{m'}}\in B$. This would imply that $\chi^m=a\chi^{m'}$, which means that $a=1$ and $m=m'$, a contradiction. Then $\Lambda'=\emptyset$ and we are done.
\end{proof}

If $P$ is a prime ideal of a ring $A$, the \emph{$n$-th symbolic power} $P^{(n)}$ is the $P$-primary component of $P^n$, or equivalenty, $P^{(n)}$ is the preimage of the localized ideal $P_P^n$ under the localization map $A\to A_P$. Note that if $P$ is the ideal of a prime Weil divisor $\Y$ on the affine scheme $\Spec(A)$, then $P^{(n)}$ is the ideal of the Weil divisor $n\Y$. 

  The next Proposition will be used in Lemma \ref{lem:Dr} to check that a Weil divisor of the form $m\mathcal{Y}$ is actually Cartier, and to compute its local equation.

\begin{Prop}
\label{prop:symbol}
Consider the ring $A=k[x,y,u]/\langle xy-u^t\rangle$ and the ideal $I=\langle y,u\rangle$ of $A$. If $m=tn$ for some positive integers $m$ and $n$, then $I^{(m)}=\langle y^n\rangle$. In particular if $B$ is a flat $A$-algebra such that $IB$ is prime, then $(IB)^{(m)}=\langle y^n\rangle B$.
\end{Prop}

\begin{proof}
If $t=1$, then $I=\langle y\rangle$ and the statement is clear. So, we can assume $t\geq2$. Note that $I$ is prime and $I^m=\langle y^{m-a}u^a;a=0,\ldots, m\rangle$. Write $a=qt+r$, with $0\leq r\leq t-1$. Then 
\[
I^m=\langle y^nx^n, y^{(t-1)(n-q)+n-r}x^qu^r; q=0,\ldots, n-1\text{ and }r=0,\ldots, t-1 \rangle,
\]
because $u^t=xy$ and $m=tn$. Therefore
\[
I^m=\langle y^n\rangle\langle x^n, y^{(t-1)(n-q)-r}x^qu^r; q=0,\ldots, n-1\text{ and }r=0,\ldots, t-1 \rangle.
\]
However, the ideal
\[
\langle x^n, y^{(t-1)(n-q)-r}x^qu^r; q=0,\ldots, n-1\text{ and }r=0,\ldots, t-1 \rangle
\]
contains $x^n$ and $y^{m-n}$ (for $q=r=0$), hence it has codimension $2$, and therefore cannot be contained in $I$. Localizing in $I$, we get $I^mA_I=\langle y^n\rangle A_I$, and hence the preimage of $I^mA_I$ in $A$, which is precisely $I^{(m)}$ by definition, is $\langle y^n\rangle$.\par
 The last statement of the Proposition follows from \cite[Tag 05G9]{stacks-project}.
\end{proof}

The following result will be used to check that certain ideal sheaves on a family of curves are torsion-free.

\begin{Prop}
\label{prop:Ifree}
Let $A$ be a local domain with maximal ideal $\mathfrak{m}$, residual field $F$, and $u,v$ elements in $\mathfrak{m}$. Consider the ideal $I=\langle y,u\rangle$ of the ring $R:=\frac{A[[x,y]]}{\langle xy-uv\rangle}$. Then
\[
I\otimes_R\frac{F[[x,y]]}{\langle xy\rangle}\cong \langle x,y \rangle,\;\;
\quad
I\otimes_R F[[y]]\cong F[[y]]\oplus \frac{F[[y]]}{\langle y\rangle}, \;\;
\quad
I F[[y]]=\langle y\rangle,
\]
where $\langle x,y\rangle$ is seen as an ideal in $\frac{F[[x,y]]}{\langle xy\rangle}$.
\end{Prop}

\begin{proof}
Set  $B:=\frac{F[[x,y]]}{\langle xy\rangle}$. A simple computation (using that $A$ is a domain) shows that
\[
I=\frac{A\alpha\oplus A\beta}{\langle x\alpha-v\beta,u\alpha-y\beta\rangle},
\]
then we obtain
\[
I\otimes_R\frac{F[[x,y]]}{\langle xy\rangle}=\frac{B\alpha\oplus B\beta}{\langle x\alpha,y\beta\rangle}\cong\langle x,y\rangle
\]
and
\[
I\otimes_R F[[y]]=\frac{F[[y]]\alpha\oplus F[[y]]\beta}{\langle y\beta\rangle}\cong F[[y]]\alpha \oplus F\beta\cong  F[[y]]\oplus \frac{F[[y]]}{\langle y\rangle}.
\]
The natural map $I\otimes_R F[[y]]\to F[[y]]$ takes $\alpha$ to $y$ and $\beta$ to $0$, so its image is $\langle y \rangle$. 
\end{proof}

Let $\sigma$ be a cone and $\tau$ a face of $\sigma$. Recall the homomorphism $\lambda\col A_\sigma\to A_\tau$ defined before \eqref{eq:mortor}. We need to describe the preimage of a principal monomial ideal in $A_\tau$ under $\lambda$.  Before this, we need a lemma.

\begin{Lem}
\label{lem:uv}
Let $\sigma$ be a cone and $\chi^u,\chi^v\in A_\sigma$ be monomials with $u,v\in S_\sigma$. If $\chi^u\in \langle \chi^v\rangle$, then $u-v\in S_\sigma$.
\end{Lem}

\begin{proof}
Write $\chi^u=f\chi^v$, where $f=a_1\chi^{v_1}+\ldots+a_s\chi^{v_s}\in A_{\sigma}$, with $v_1,\dots,v_s\in S_\sigma$ and $a_1,\dots,a_s\in k$. By Proposition \ref{prop:ES1.7}, either $\chi^u=0$ or there exist $a\in k$ and $j\in \{1,\ldots,s\}$ such that $\chi^u-aa_j\chi^{v_j+v}=0$. Note that both $\chi^u$ and $\chi^u-aa_j\chi^{v_j+v}$ are regular functions on the torus.  Write $u=(u_1,\ldots,u_n)$. Since $\chi^u(t_1,\ldots,t_n)=t_1^{u_1}\cdots t_n^{u_n}$ is nonzero on the torus, we deduce that $\chi^u\neq0$, and hence $\chi^u/\chi^{v_j+v}=aa_j$. However 
\[
\frac{\chi^u(1,\ldots,1)}{\chi^{v_j+v}(1,\ldots,1)}=1
\]
from which we get $\chi^u=\chi^{v_j+v}$, which implies $u-v=v_j\in S_\sigma$.
\end{proof}

\begin{Prop}
\label{prop:hface}
Let $\sigma$ be a cone and $\tau$ a face of $\sigma$. Let $\lambda\col A_\sigma\to A_\tau$ be the corresponding localization homomorphism. For every $u_0\in S_\tau$, 
\[
\lambda^{-1}(\langle \chi^{u_0}\rangle)=\langle \chi^v; v\in S_\sigma\textnormal{ and }v-u_0\in S_\tau\rangle,
\]
 so, in particular $\lambda^{-1}(\langle \chi^{u_0} \rangle)$ is monomial. Moreover, $\chi^u\in \lambda^{-1}(\langle \chi^{u_0}\rangle)$ if and only if $u-u_0\in S_\tau$.
\end{Prop}
\begin{proof}
Recall that the homomorphism $\lambda\col A_\sigma\to A_\tau$ is the localization with respect to the multiplicative set $R:=\{ \chi^u; u|_\tau=0\}$ (see \eqref{eq:multsys}). Denote by $J$ the ideal on the right-hand side of the equation in the statement. Then $R^{-1}J=\langle \chi^{u_0}\rangle$, because if $v(r)\geq u_0(r)$ for every $r\in\tau$, then $v-u_0\in S_\tau$. So, by \cite[Proposition 2.2]{E}, it is enough to prove that if $u|_\tau=0$ and $f\chi^u\in J$ then $f\in J$. Since $J$ is monomial, by Corollary \ref{cor:sigmamon}, we can assume that $f$ is a monomial, namely $f=\chi^{v}$. Since $J$ is monomial and finitely generated, we can apply Proposition \ref{prop:ES1.6} to deduce that $\chi^{u+v}\in \langle\chi^{v'}\rangle$ for some $v'\in S_\sigma$ such that $v'-u_0\in S_\tau$. By Lemma \ref{lem:uv}, we have  $u+v-v'\in S_\sigma$. Hence it follows that $u+v-v'\in S_\tau$. We deduce that  
\[
u+v-u_0=(u+v-v')+(v'-u_0)\in S_\tau,
\]
because $S_\tau$ is a semigroup. But $u$ is invertible in $S_\tau$ because $u|_\tau=0$, hence $v-u_0\in S_\tau$, and this shows that $f\in J$. \par
  Let us prove the last statement.  By Proposition \ref{prop:ES1.6}, $\chi^u\in J$ if and only if there exists $v\in S_\sigma$ such that $v-u_0\in S_\tau$ and $\chi^u\in \langle \chi^v\rangle$. By Lemma \ref{lem:uv}, we have $u-v\in S_\sigma\subset S_\tau$. Then $u-u_0=(u-v)+(v-u_0)\in S_\tau$.
\end{proof}

\section{Geometric and algebraic properties of the cones $K_{\Gamma,\E,\phi}$}\label{sec:combalg}

In this section we study some properties of the cones  $K_{\Gamma,\E,\phi}$ introduced in Notation \ref{not:KGamma}. In particular the results of Propositions \ref{prop:Kdim} and \ref{prop:Icap} will be crucial in the resolution of the geometric Abel described in Section \ref{sec:Abelres}.

  We begin computing the dimension of $K_{\Gamma,\E,\phi}$.

\begin{Prop}
\label{prop:Kdim}
 Let $\Gamma$ be a graph, $\E$ be a nondisconnecting subset of $E(\Gamma)$ and $\phi$ be an acyclic flow on $\Gamma^\E$ such that $\div(\phi)$ has degree $-1$ on every exceptional vertex. If $\F=\{e\in E(\Gamma)\setminus\E;\phi(e)\neq 0\}$, then
\begin{equation}
\label{eq:Kdim}
\dim(K_{\Gamma,\E,\phi})= |E(\Gamma)|-|\F|+b_0(\Gamma_{\E\cup\F})-1.
\end{equation}
Moreover, $\dim(K_{\Gamma,\E,\phi})=0$ if and only if $\Gamma$ has a single vertex, and no edges. Finally, $\dim(K_{\Gamma,\E,\phi})=1$ if and only if one of the following conditions hold:
\begin{enumerate}
\item $|V(\Gamma)|=2$, $\Gamma$ has no loops, $\E=\emptyset$ and $\phi\ne0$.
\item $|V(\Gamma)|=2$, $\Gamma$ has no loops, $\E=\emptyset$, $|E(\Gamma)|=1$ and $\phi=0$.
\item $|V(\Gamma)|=|E(\Gamma)|=1$ and  $\E=\emptyset$.
\end{enumerate}
In particular every extremal ray of $K_{\Gamma,\E,\phi}$ corresponds to a specialization $(\Gamma,\E,\phi)\leadsto(\Gamma',\emptyset,\phi')$ where $\Gamma'$ and $\phi'$ satisfies  one of (1), (2), and (3).
\end{Prop}

\begin{proof} 
By Proposition \ref{prop:isoCK}, $\dim(K_{\Gamma,\E,\phi})=\dim(C_{\Gamma,\E,\phi})$. We write
\[
\dim(C_{\Gamma,\E,\phi})=|E(\Gamma)|+|\E|-\delta,
\]
 where $\delta$ is the number of independent equations in \eqref{eq:C}. Since $\E$ is nondisconnecting, for each $e\in \E$ we can choose a cycle $\gamma_e$ of $\Gamma$ such that $|\E|\cap \gamma_e=\{e\}$. Choose a basis $\gamma_1,\ldots, \gamma_b$ of $H_1(\Gamma_\E,\mathbb{Z})$, where $b=b_1(\Gamma_\E)$. It is clear that the collection $\{\gamma_e\}_{e\in\E}\cup\{\gamma_i\}_{i=1,\ldots,b}$ is a basis of $H_1(\Gamma,\mathbb{Z})$.\par
   Since $\div(\phi)(v_e)=-1$ for every exceptional vertex $v_e$ with $e\in \E$, the flow $\phi$ does not vanish on at least one edge of $\Gamma^\E$ over $e$, so the equations induced by $\{\gamma_e\}_{e\in \E}$ are independent, and also independent of the equations induced by $\{\gamma_i\}_{i=1\ldots,b}$. \par
	If $\phi$ vanishes nowhere on $E(\Gamma_\E)$, then  $\delta-|\E|=b$. In fact, we can argue as before, choosing distinct edges $e_1,\ldots, e_b$ and cycles $\gamma_1,\ldots, \gamma_b$ such that $e_j\in \gamma_i$ if and only if $i=j$. More generally, if we let $\F_0=\{e\in E(\Gamma)\setminus\E;\phi(e)=0\}$, the set of edges of $\Gamma_\E$ over which $\phi$ vanishes, then $\delta-|\E|=b_1(\Gamma_\E/\F_0)$. In fact we can choose a basis of $H_1(\Gamma_\E/\F_0,\mathbb{Z})$, extend it to a set $A$ of cycles of $\Gamma_\E$ (adding edges of $\F_0$) and complete $A$ to a basis of $H_1(\Gamma_\E,\mathbb{Z})$ by adding a set $B$ of cycles whose edges are contained in $\F_0$. The equations induced by the cycles in $A$ are independent (cf. the case where $\phi$ is non-vanishing), while the equation induced by each cycle in $B$ is trivial.\par
	So we deduce that
\begin{align*}
\dim(K_{\Gamma,\E,\phi})&=|E(\Gamma)|+|\E|-\delta\\
                        &=|E(\Gamma)|+|\E|-(|\E|+b_1(\Gamma_E/\F_0))\\
												&=|E(\Gamma)|-|\F|+b_0(\Gamma_{\E\cup\F})-1,
\end{align*}												
where, since $E(\Gamma_\E)=\F\coprod \F_0$, the last equality follows from Propostion \ref{prop:b0b1}.

	If $\dim(K_{\Gamma,\E,\phi})=0$, then $(|E(\Gamma)|-|\F|)+(b_0(\Gamma_{\E\cup\F})-1)=0$, which means that $\F=E(\Gamma)$ and $b_0(\Gamma_{\E\cup\F})=1$. Since $\F=E(\Gamma)$, we get $b_0(\Gamma_{\E\cup\F})=|V(\Gamma)|$, hence $\Gamma$ has a single vertex. Since $\phi$ is acyclic, no loops can be in $\F$, so $\Gamma$ has no loops.

	If $\dim(K_{\Gamma,\E,\phi})=1$, then, by the same argument above, either $\F=E(\Gamma)$ and $|V(\Gamma)|=b_0(\Gamma_{\E\cup\F})=2$, or $|E(\Gamma)\setminus\F|=1$ and $b_0(\Gamma_{\E\cup\F})=1$.\par
	In the first case, $\Gamma$ has two vertex and, since $\F=E(\Gamma)$ and $\F$ contains no loops, $\Gamma$ has no loops. Note that $\E=\emptyset$, because $\E\subset E(\Gamma)\setminus \F$, and $\phi\neq0$.\par
  In the second case, write $E(\Gamma)\setminus \F=\{e_0\}$. Note that $\E=\emptyset$, otherwise $e_0\in \E$ and $|V(\Gamma)|=b_0(\Gamma_{\E\cup\F})=1$, implying that $e_0$ is a loop, which contradicts that, since $\phi$ is acyclic, $\E$ contains no loops. Since $b_0(\Gamma_{\F})=1$, the graph $\Gamma_\F$ is connected. Since $E(\Gamma_\F)=\{e_0\}$, it follows that $\Gamma_\F$, and so $\Gamma$, have at most two vertices. \par
	If $\Gamma$ has two vertices, then $e_0$ must be an edge connecting them. Note that $\Gamma$ has no loops, because $\F$ has none. Moreover, since $\phi(e_0)=0$ and $\phi$ is acyclic, it follows that $\phi$ must vanish on every edge of $\Gamma/\{e_0\}$, hence $\F=\emptyset$ and $E(\Gamma)=\{e_0\}$.\par
	If $\Gamma$ has one vertex, then $e_0$ is a loop. Since $\F$ contains no loops, then $\F=\emptyset$ and hence $E(\Gamma)=\{e_0\}$.\par
  The ``if'' parts of the statement easily follow from equation \eqref{eq:Kdim}.
\end{proof}

From now on we fix a graph $\Gamma$, a nondisconnecting subset $\E$ of $E(\Gamma)$ and an acyclic flow $\phi$ on $\Gamma^\E$ such that $\div(\phi)$ has degree $-1$ on every exceptional vertex $v_e$, $e\in \E$. 

\begin{Not}
\label{not:eset}
Consider an edge $e\in\E$. Since $\div(\phi)(v_e)=-1$, the flow $\phi$ induces an orientation of $e$. We denote by $e^s$ and $e^t$ the edges of $E(\Gamma^\E)$ over $e$, such that $t(e^s)=s(e^t)$. Note that $e^s$ and $e^t$ inherit the orientation of $e$, and
\begin{equation}\label{eq:plus1}
\phi(e^t)=\phi(e^s)+1. 
\end{equation}
\end{Not}

We will also fix an edge $e_0\in\E$.  Fix a spanning tree $T$ of $\Gamma$ in the complement of $\E$. There exists a unique cycle $\gamma$ of $\Gamma$ containing $e_0$ and with all other edges contained in $T$. The cycle $\gamma$ induces a cycle on $\Gamma^\E$, which, abusing notation, we will still call $\gamma$. We assume that $\gamma(e_0^s)=\gamma(e_0^t)=1$ (recall equation \eqref{eq:gamma}). We define $u_{e_0}', u_{e_0}''\in (\R^{E(\Gamma)})^\vee$ as the linear functionals acting on a vector $e$ of the base $E(\Gamma)$ of $\R^{E(\Gamma)}$, respectively, via
\begin{equation}
\label{eq:u'u''}
u_{e_0}'(e):=\begin{cases}
             0, &\text{ if } e\notin T\cup\{e_0\};\\
	   -\gamma(e)\phi(e), &\text{ if } e\in T;\\
	   -\phi(e_0^s),&\text{ if } e=e_0;
						\end{cases} 
\end{equation}
and
\begin{equation}
\label{eq:u'u''bis}
u_{e_0}''(e):=\begin{cases}
             0, &\text{ if } e\notin T\cup\{e_0\};\\
             \gamma(e)\phi(e), &\text{ if } e\in T;\\
	   \phi(e_0^t),&\text{ if } e=e_0.
						\end{cases}
\end{equation}
	
\begin{Lem}
\label{lem:ue0}
We have $u_{e_0}',u_{e_0}''\in K_{\Gamma,\E,\phi}^\vee$ and $u_{e_0}'+u_{e_0}''=e_0^\vee$. If $r=K_{\Gamma',\emptyset,\phi'}$ is an extremal ray of $K_{\Gamma,\E,\phi}$ and $\iota\col\Gamma^\E\to \Gamma'$ is the induced specialization, then either $u_{e_0}'(r)=0$ or $u''_{e_0}(r)=0$. Moreover, the following properties hold
\begin{enumerate}
\item if $e_0^t$ is contracted by $\iota$, then $u_{e_0}'(r)=0$;
\item if $e_0^s$ is contracted by $\iota$, then $u_{e_0}''(r)=0$.
\end{enumerate}
In particular, if $u_{e_0}'(r)=u''_{e_0}(r)=0$, then $e_0$ is contracted in the induced specialization $\Gamma\to \Gamma'$.
\end{Lem}	

\begin{proof}
Write $u':=u'_{e_0}$ and $u''=u''_{e_0}$. The fact that $u'+u''=e_0^\vee$ comes from $\phi(e_0^t)=\phi(e_0^s)+1$, see \eqref{eq:plus1}. Since $e_0^\vee\in K_{\Gamma,\E,\phi}^\vee$ and $u'(r)+u''(r)=e_0^\vee(r)$, the property $u',u''\in K_{\Gamma,\E,\phi}^\vee$ follows once we show that for every extremal $r$ of $K_{\Gamma,\E,\phi}$, either $u'(r)=0$ or $u''(r)=0$. Indeed, if this holds, then for every extremal ray $r$ of $K_{\Gamma,\E,\phi}$ either $u'(r)=0$ or 
\[
u'(r)=e_0^\vee(r)-u''(r)=e_0^\vee(r)\geq0.
\]
In any case, $u'(r)\geq0$, hence $u'\in K_{\Gamma,\E,\phi}^\vee$, and similarly $u''\in K_{\Gamma,\E,\phi}^\vee$.

Let $r$ be an extremal ray as in the statement.
 By Proposition \ref{prop:Kdim}, either $\Gamma'$ has two vertices and no loops or $\Gamma'$ has one vertex and a single loop. In the latter case, if $e$ is the single loop of $\Gamma'$, then $r=\langle e\rangle$ and, since $\phi'(e)=0$ (recall that $\phi'$ is acyclic),   one of the following cases hold
\begin{itemize}
\item $e\notin T\cup\{e_0\}$, then $e_0^s,e_0^t$ are contracted by $\iota$ and $u'(e)=u''(e)=0$;
\item $e\in T$, then $e_0^s,e_0^t$ are contracted by $\iota$ and $\phi(e)=\phi'(e)=0$, hence $u'(e)=u''(e)=0$;
\item $e=e_0$, then $\phi'(e)$ is equal to either $\phi(e_0^s)$ or $\phi(e_0^t)$ but, since $\phi(e_0^t)>0$, the edge $e_0^t$ is contracted by $\iota$ and $\phi(e_0^s)=\phi'(e)=0$ and hence $u'(e)=0$.
\end{itemize}

  Assume now that $\Gamma'$ has two vertices and no loops. If $\Gamma'$ has a single edge $e$, then $e$ is a separating edge, hence $e\neq e_0$ and $\gamma(e)=0$, from which we deduce that $e_0^s$ and $e_0^t$ are contracted by $\iota$ and  $u'(e)=u''(e)=0$. So, by Proposition \ref{prop:Kdim}, we can assume that $\Gamma'$ has at least two edges and $\phi'$ is nonzero. 

Now, let $e_1,\ldots, e_k$ be the edges of $\Gamma'$.  By \eqref{eq:C}, we can write
	\[
	r=\left\langle e_1+\frac{\phi'(e_1)}{\phi'(e_2)}e_2+\ldots\frac{\phi'(e_1)}{\phi'(e_k)}e_k\right\rangle,
	\]
(recall we can view $e_1,\ldots, e_k$ as edges of $\Gamma^\E$ as well). If the cycle $\gamma$ does not contain any edge $e_1,\ldots, e_k$ (so, in particular, $e_0\neq e_i$ for every $i=1,\ldots,k$), then $e_0^s$ and $e_0^t$ are contracted by $\iota$ and $\gamma(e_i)=0$, and hence $u'(r)=u''(r)=0$. Otherwise, we can assume that $\gamma$ contains precisely the edges $e_1$ and $e_2$. We compute:
\[
u'(r)=u'(e_1)+\frac{\phi'(e_1)}{\phi'(e_2)}u'(e_2)\;
\text{ and } 
\; u''(r)=u''(e_1)+\frac{\phi'(e_1)}{\phi'(e_2)}u''(e_2).
\]
We now have two cases to check. In the first case, $e_1,e_2\in T$. Then $e_1,e_2\neq e_0$, hence $e_0^s$ and $e_0^t$ are contracted by $\iota$. Moreover, $\phi'(e_i)=\phi(e_i)$ for $i=1,2$ and, since $\phi'$ is acyclic, we have $\gamma(e_1)=-\gamma(e_2)$. We deduce that 
\[
u''(r)=-u'(r)=\gamma(e_1)\phi(e_1)+\frac{\phi(e_1)}{\phi(e_2)}\gamma(e_2)\phi(e_2)=\phi(e_1)(\gamma(e_1)+\gamma(e_2))=0.
\]
In the second case, $e_1\in T$ and $e_2\not\in T$. Note that either $e_0^s$ or $e_0^t$ must be contracted by $\iota$, because $r=K_{\Gamma',\emptyset,\phi'}$. Moreover, either $e_2=e_0^s$ (if $e_0^t$ is contracted by $\iota$) or $e_2=e_0^t$ (if $e_0^s$ is contracted by $\iota$). Then $\gamma(e_1)=-1$ (because $\gamma(e_0)=1$) and $\phi'(e_1)=\phi(e_1)$. Hence
\begin{align*}
u'(r)=u'(e_1)+\frac{\phi'(e_1)}{\phi'(e_0)}u'(e_0)=\phi(e_1)-\frac{\phi(e_1)\phi(e_0^s)}{\phi'(e_0)}\\
u''(r)=u''(e_1)+\frac{\phi'(e_1)}{\phi'(e_0)}u''(e_0)=-\phi(e_1)+\frac{\phi(e_1)\phi(e_0^t)}{\phi'(e_0)}.
\end{align*}
If $e_0^t$ is contracted by $\iota$, then $\phi'(e_0)=\phi(e_0^s)$ and hence $u'(r)=0$. If $e_0^s$ is contracted by $\iota$, then $\phi'(e_0)=\phi(e_0^t)$ and hence $u''(r)=0$.\par
 Finally, if $u'(r)=u''(r)=0$, then $e_0^\vee(r)=u'(r)+u''(r)=0$, and the last statement follows.
\end{proof}

\begin{Rem}
\label{rem:uinvertible}
There are rays $r_1$ and $r_2$ such that $u_{e_0}'(r_1)>0$ and $u''_{e_0}(r_2)>0$, hence neither $u'_{e_0}$ or $u''_{e_0}$ are invertible in $S_{K_{\Gamma,\E,\phi}}$. Indeed, for $j=1,2$, let us find specializations $\iota_j\col(\Gamma,\E,\phi)\to (\Gamma_j,\E_j,\phi_j)$ with $\phi_j$ acyclic, such that $\iota_1$ contracts $e_0^s$ but does not contract $e_0^t$, and vice-versa for $\iota_2$.
The specialization $\iota_1$ is the one contracting $e_0^s$, since $\phi(e_0^t)>0$, then $\phi_1$ will remain acyclic. If $\phi(e_0^s)>0$, the specialization $\iota_2$ is given similarly, contracting $e_0^t$, while if $\phi(e_0^s)=0$, we can just contract all the edges of $\Gamma^\E$ but $e_0^s$.
\end{Rem}

\begin{Lem}
\label{lem:uaubis}
 Let $u\in K_{\Gamma,\E,\phi}^\vee$ be such that $u(r)\geq e_0^\vee(r)$ for every extremal ray $r$ of $K_{\Gamma,\E,\phi}$ such that $u_{e_0}'(r)=0$. Then $u-u_{e_0}''\in K_{\Gamma,\E,\phi}^\vee$.
\end{Lem}

\begin{proof}
Let $r$ be an extremal ray $r$ of $K_{\Gamma,\E,\phi}$.
It suffices that $u(r)\geq u_{e_0}''(r)$. If $u_{e_0}''(r)=0$, then we are done because $u\in K^\vee_{\Gamma,\E,\phi}$ (hence $u(r)\geq0$). If $u_{e_0}''(r)\neq0$ then Lemma \ref{lem:ue0} implies $u_{e_0}'(r)=0$ and $u_{e_0}''(r)=e_0^\vee(r)$, concluding the proof.
\end{proof}

Recall that $\Gamma$, $\E$, $\phi$ and $e_0\in E(\Gamma)$ are fixed. Let $\tau$ be the cone obtained from $K_{\Gamma,\E,\phi}$ and $e_0^\vee$ as in Construction \ref{cons:xy} (see also \eqref{eq:Atau}). Note that $\tau$ depends on $\Gamma$, $\E$, $\phi$ and  $e_0$, which will be omitted in the notation. Recall \eqref{eq:Asigma}. We see that
\begin{equation}
\label{eq:AtauK}
A_\tau=\frac{A_{K_{\Gamma,\E,\phi}}[x_{e_0},y_{e_0}]}{\langle x_{e_0}y_{e_0}-\chi^{e_0^\vee}\rangle},
\end{equation}
where $x_{e_0},y_{e_0}$ are variables. We can view $A_\tau$ as the datum of a family of nodal curves over $\Spec(A_{K_{\Gamma,\E,\phi}})$ locally around a node of a fiber (corresponding to $e_0$). \par
   Let $r$ be an extremal ray in $K_{\Gamma,\E,\phi}(1)$. Let $\tau'$ be the face of $\tau$ over $r$, i.e., the preimage of $\langle r\rangle $ under the map $\tau\to K_{\Gamma,\E,\phi}$. There is a localization homomorphism $\lambda\col A_\tau\to A_{\tau'}$, and we define 
\begin{equation}
\label{eq:defIer}
I_{e_0,r}:=\lambda^{-1}(\langle y_{e_0}\rangle).
\end{equation}
 The $n$-th symbolic power of $I_{e_0,r}$ is given by 
\begin{equation}
\label{eq:defIe0rn}
I^{(n)}_{e_0,r}=\lambda^{-1}(\langle y_{e_0}^n\rangle). 
\end{equation}

Recall that we set $N\cong\mathbb Z^n$ and $M=\text{Hom}(N,\mathbb Z)$.

\begin{Lem}
\label{lem:Ie0}
The $n$-th symbolic  power $I_{e_0,r}^{(n)}$ of the ideal $I_{e_0,r}$ of $A_\tau$ is monomial and a monomial $x_{e_0}^ay_{e_0}^b\chi^u$ is in $I_{e_0,r}^{(n)}$ if and only if $u(r)\geq (n-b)e_0^\vee(r)$. In particular 
\[
I^{(n)}_{e_0,r}=\langle y_{e_0}^b\chi^u; u\in K^\vee_{\Gamma,\E,\phi}\cap M\text{ with }u(r)\geq (n-b)e_0^\vee(r)\rangle.
\]
\end{Lem}
\begin{proof}
The statement follows from Proposition \ref{prop:hface}: we need only to translate the notation. Keep the notation in Construction \ref{cons:xy}. Set $x:=x_{e_0}$ and $y:=y_{e_0}$. Since $y=\chi^{v_2}$, we can apply Proposition \ref{prop:hface} and get
\[
I^{(n)}_{e_0,r}=\langle \chi^v; v\in S_\tau\text{ such that }v-nv_2\in S_{\tau'} \rangle.
\]
Every $v\in S_\tau$ can be written as $v=u+av_1+bv_2$, with $u\in S_\sigma$ and $a,b$ nonnegative integers, which means that $\chi^v=\chi^ux^ay^b$ and, vice versa, every monomial $\chi^ux^ay^b$ is equal to $\chi^v$ for some $v=u+av_1+bv_2$. However, the condition $v-nv_2\in S_{\tau'}$ is equivalent to 
\[
(v-nv_2)(r,0,e_0^\vee(r))\geq0\quad\text{ and }\quad(v-nv_2)(r,e_0^\vee(r),0)\geq0.
\]
The first condition is the same as $u(r)+(b-n)e_0^\vee(r)\geq0$ which gives precisely our statement, while the second condition is $u(r)+ae_0^\vee(r)\geq0$ which is always true.
\end{proof}

In the proof of Lemma \ref{lem:Dr}, the previous result will translate to the following geometric property. First, recall the open immersion $U_{\tau'}=\Spec(A_{\tau'})\to \Spec(A_\tau)=U_\tau$ in \eqref{eq:mortor}. If $\Y$ is the Cartier divisor on $U_{\tau'}$ given by the ideal $\langle y_{e_0}\rangle$ then the ideal of the closure $\ol{n\Y}$ in $U_\tau$ of the Cartier divisor $n\Y$ is the ideal $I_{e_0,r}^{(n)}$.\par

The following is a key result used in the proof of Theorem \ref{thm:mainlocal}.
Recall  \eqref{eq:defIer} and the notation in Section \ref{sec:toricsec}. 

\begin{Prop}
\label{prop:Icap}
 Let $\Gamma$ be a graph, $\E$ a nondisconnecting subset of $E(\Gamma)$ and $\phi$ an acyclic flow on $\Gamma^\E$ such that $\div(\phi)$ has degree $-1$ on every exceptional vertex. Let $e_0$ be an edge in $E(\Gamma)$. If $e_0\in \E$, then
\[
\left(\underset{u_{e_0}'(r)=0}{\bigcap_{r\in K_{\Gamma,\E,\phi}(1)}} I_{e_0,r}^{(\phi(e_0^t))}\right)\cap\left(\underset{u_{e_0}''(r)=0}{\bigcap_{r\in K_{\Gamma,\E,\phi}(1)}} I_{e_0,r}^{(\phi(e_0^s))}\right)=\left\langle y_{e_0}\right\rangle^{\phi(e_0^s)}\left\langle y_{e_0},\chi^{u_{e_0}''}\right\rangle.
\]
Otherwise, if $e_0\not\in\E$, then
\[
\left({\bigcap_{r\in K_{\Gamma,\E,\phi}(1)}} I_{e_0,r}^{(\phi(e_0))}\right)=\left\langle y_{e_0}\right\rangle^{\phi(e_0)}.
\]
\end{Prop}
\begin{proof}

 First assume that $e_0\in \E$. We set $n:=\phi(e_0^s)$, so that $\phi(e_0^t)=n+1$,  $K:=K_{\Gamma,\E,\phi}$, $x:=x_{e_0}$, and $y:=y_{e_0}$. We  denote by $J$ (respectively, $J'$) the ideal on the left-hand side (respectively, right-hand side) of the equation in the statement. By Proposition \ref{prop:ES1.6}, the intersection of monomial ideals  in toric rings are monomial ideals, hence $J=J'$ once we prove the equality at the level of monomials. \par

 We first prove that $J\subset J'$. Let $x^ay^b\chi^u$ be a monomial in $J$, where $a,b$ are nonnegative integers and $u\in K^\vee$. If $b\geq n+1$, then we are done. So we reduce to the case $b\leq n$.  By Lemma \ref{lem:Ie0}, we have $x^ay^b\chi^u\in J$ if and only if
\begin{equation}
\label{eq:ur}
\begin{cases}
u(r)\geq(n-b+1)e_0^\vee(r), & \forall \; r\in  K(1) \text{ such that } u_{e_0}'(r)=0,\\
u(r)\geq(n-b)e_0^\vee(r),  & \forall\; r\in  K(1) \text{ such that } u_{e_0}''(r)=0.
\end{cases}
\end{equation}
This implies that $u-(n-b)e_0^\vee\in K^\vee$ because, by Lemma \ref{lem:ue0}, it is nonnegative over all extremal rays $r\in K(1)$. Moreover the hypotheses in Lemma \ref{lem:uaubis} are satisfied, hence $u-(n-b)e_0^\vee-u''\in K^\vee$. However, $\chi^{e_0^\vee}=xy$, hence 
\[
y^b\chi^u=y^b(xy)^{n-b}\chi^{u''}\chi^{u-(n-b)e_0^\vee-u''}\in J'.
\]
  The other inclusion $J'\subset J$ follows from the fact that $y^n\chi^{u''}$ satisfies  \eqref{eq:ur}.\par
	Assume now that $e_0\notin \E$ and set $n:=\phi(e_0)$. Following through the proof of the previous case, the only difference appears in \eqref{eq:ur} which becomes
\[
u(r)\geq(n-b)e_0^\vee(r), \forall\; r\in K_{\Gamma,\E,\phi}(1).
\]
Hence $u-(n-b)e_0^\vee\in K_{\Gamma,\E,\phi}^\vee$ and we get
\[
y^b\chi^u=y^b(xy)^{n-b}\chi^{u-(n-b)e_0^\vee}\in J'.
\]
The other inclusion $J'\subset J$ follows from the fact that $y^n\in I_{e_0,r}^{(n)}$ for every $r\in K_{\Gamma,\E,\phi}(1)$.
	\end{proof}

\begin{Exa}
\label{exa:ideal}
Maintain the notation of Example \ref{exa:fan}. We let $(\Gamma,\E,\phi)$ be the triple depicted in Figure \ref{fig:flow}. By Example \ref{exa:fan}, 
\[
K_{\Gamma,\E,\phi}=\langle (1,1,1),(1,2,2),(2,1,2),(1,1,2)\rangle\subset \R^3_{\geq0}.
\] 
Using \cite[Sagemath]{SM} we compute
\[
K^\vee_{\Gamma,\E,\phi}=\langle (0,-1,1),(2,0,-1),(0,2,-1),(-1,0,1)\rangle\subset (\R^3)^\vee
\]
and
\begin{equation}\label{eq:SKGamma}
S_{K_{\Gamma,\E,\phi}}=\langle (0,-1,1),(2,0,-1),(0,2,-1),(-1,0,1),(1,1,-1)\rangle\subset (\R^3)^\vee.
\end{equation}
Each extremal ray of $K_{\Gamma,\E,\phi}$ corresponds to one of the flows in Figure \ref{fig:fray} which are of the type predicted by Proposition \ref{prop:Kdim}.
\begin{figure}[h]
\begin{tikzpicture}[scale=2.3]
\begin{scope}
\draw[decorate, decoration={markings, mark=at position 0.5 with {\arrow{>}}}]   (0,0) to [out=30, in=150] (1,0);
\draw (0,0) to [out=30, in=150] (1,0);
\draw[decorate, decoration={markings, mark=at position 0.5 with {\arrow{>}}}] (0,0) to (1,0);
\draw (0,0) to (1,0);
\draw (0,0) to [out=-30, in=-150] (1,0);
\draw[decorate, decoration={markings, mark=at position 0.5 with {\arrow{>}}}] (0,0) to [out=-30, in=-150] (1,0);
\draw[fill] (0,0) circle [radius=0.02];
\draw[fill] (1,0) circle [radius=0.02];
\node[above] at (0.5,0.16) {\tiny 1};
\node[above] at (0.5,-0.02) {\tiny 1};
\node[below] at (0.5,-0.144) {\tiny 1};
\node[below] at (0,0) {$v_0$};
\node[below] at (1,0) {$v_1$};
\end{scope}
\begin{scope}[shift={(1.3,0)}]
\draw[decorate, decoration={markings, mark=at position 0.5 with {\arrow{>}}}]   (0,0) to [out=30, in=150] (1,0);
\draw (0,0) to [out=30, in=150] (1,0);
\draw[decorate, decoration={markings, mark=at position 0.5 with {\arrow{>}}}] (0,0) to (1,0);
\draw (0,0) to (1,0);
\draw (0,0) to [out=-30, in=-150] (1,0);
\draw[decorate, decoration={markings, mark=at position 0.5 with {\arrow{>}}}] (0,0) to [out=-30, in=-150] (1,0);
\draw[fill] (0,0) circle [radius=0.02];
\draw[fill] (1,0) circle [radius=0.02];
\node[above] at (0.5,0.16) {\tiny 2};
\node[above] at (0.5,-0.02) {\tiny{1}};
\node[below] at (0.5,-0.144) {\tiny 1};
\node[below] at (0,0) {$v_0$};
\node[below] at (1,0) {$v_1$};
\end{scope}
\begin{scope}[shift={(2.6,0)}]
\draw[decorate, decoration={markings, mark=at position 0.5 with {\arrow{>}}}]   (0,0) to [out=30, in=150] (1,0);
\draw (0,0) to [out=30, in=150] (1,0);
\draw[decorate, decoration={markings, mark=at position 0.5 with {\arrow{>}}}] (0,0) to (1,0);
\draw (0,0) to (1,0);
\draw (0,0) to [out=-30, in=-150] (1,0);
\draw[decorate, decoration={markings, mark=at position 0.5 with {\arrow{>}}}] (0,0) to [out=-30, in=-150] (1,0);
\draw[fill] (0,0) circle [radius=0.02];
\draw[fill] (1,0) circle [radius=0.02];
\node[above] at (0.5,0.16) {\tiny 1};
\node[above] at (0.5,-0.02) {\tiny{2}};
\node[below] at (0.5,-0.144) {\tiny 1};
\node[below] at (0,0) {$v_0$};
\node[below] at (1,0) {$v_1$};
\end{scope}
\begin{scope}[shift={(3.9,0)}]
\draw[decorate, decoration={markings, mark=at position 0.5 with {\arrow{>}}}]   (0,0) to [out=30, in=150] (1,0);
\draw (0,0) to [out=30, in=150] (1,0);
\draw[decorate, decoration={markings, mark=at position 0.5 with {\arrow{>}}}] (0,0) to (1,0);
\draw (0,0) to (1,0);
\draw (0,0) to [out=-30, in=-150] (1,0);
\draw[decorate, decoration={markings, mark=at position 0.5 with {\arrow{>}}}] (0,0) to [out=-30, in=-150] (1,0);
\draw[fill] (0,0) circle [radius=0.02];
\draw[fill] (1,0) circle [radius=0.02];
\node[above] at (0.5,0.16) {\tiny 2};
\node[above] at (0.5,-0.02) {\tiny{2}};
\node[below] at (0.5,-0.144) {\tiny 1};
\node[below] at (0,0) {$v_0$};
\node[below] at (1,0) {$v_1$};
\end{scope}
\end{tikzpicture}
\caption{Specializations of $(\Gamma,\E,\phi)$}
\label{fig:fray}
\end{figure}
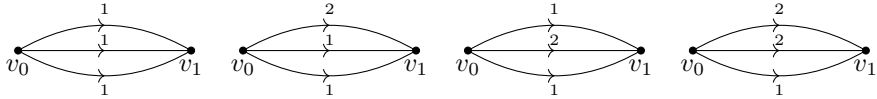

We compute the linear functionals $u'_{e_i},u''_{e_i}\in(\R^3)^\vee$ for $i=0,1$. It is clear that  $T=\{e_2\}$ is the unique spanning tree in the complement of $\E$. Hence 
\[
u'_{e_0}=(-1,0,1), \; u''_{e_0}=(2,0,-1), \; u'_{e_1}=(0,-1,1), \; u''_{e_1}=(0,2,-1).
\]
 These are precisely the extremal rays of $K_{\Gamma,\E,\phi}^\vee$ (this could not happen in general). As stated by Lemma \ref{lem:ue0}, we have $u'_{e_0}+u''_{e_0}=(1,0,0)$ and $u'_{e_1}+u''_{e_1}=(0,1,0)$. It is easy to check that for every extremal ray of $K_{\Gamma,\E,\phi}$, either $u'_{e_i}$ or $u''_{e_i}$ vanishes on it, for $i=0,1$. Moreover, $u'_{e_i}$ and $u''_{e_i}$ are the normal vectors of opposing faces of $K_{\Gamma,\E,\phi}$, for $i=0,1$. In fact, Lemma \ref{lem:ue0} shows that this is always the case, i.e., there are opposing faces (not necessarily  of the same dimension, and whose intersection is not necessarily 0-dimensional). These faces correspond precisely to the two specializations of $(\Gamma,\E,\phi)\to (\Gamma,\E\setminus\{e_0\},\phi')$.  \par
  Now we describe the algebra $A_{K_{\Gamma,\E,\phi}}$. By \eqref{eq:SKGamma}, we can write 
	\[
	A_{K_{\Gamma,\E,\phi}}=k[\chi^{(0,-1,1)},\chi^{(2,0,-1)},\chi^{(0,2,-1)},\chi^{(-1,0,1)},\chi^{(1,1,-1)}]
	\]
	or, equivalently (again, using SageMath),
	\[
	A_{K_{\Gamma,\E,\phi}}=\frac{k[z_0,z_1,z_2,z_3,z_4]}{\langle z_0z_4-z_1z_3,z_4^2-z_1z_2,z_3z_4-z_0z_2,z_1z_3^2-z_0^2z_2\rangle}.
	\]
	We also have $\chi^{e_0^\vee}=\chi^{(1,0,0)}=z_0z_4=z_1z_3$ and $\chi^{e_1^\vee}=\chi^{(0,1,0)}=z_0z_2=z_3z_4$.

	 Let us concentrate on the edge $e_0$. Let $r_1=(1,1,1)$, $r_2=(2,1,2)$, $r_3=(1,2,2)$ and $r_4=(1,1,2)$. Note that $u'_{e_0}(r_i)=0$ if and only if $i=1,2$, and $u''_{e_0}(r_i)=0$ if and only if $i=3,4$. Moreover,
\begin{align*}
A:=&\frac{A_{K_{\Gamma,\E,\phi}}[x,y]}{\langle xy-\chi^{(1,0,0)}\rangle}\\
 =&\frac{k[z_0,z_1,z_2,z_3,z_4,x,y]}{\langle z_0z_4-z_1z_3,z_4^2-z_1z_2,z_3z_4-z_0z_2,z_1z_3^2-z_0^2z_2,xy-z_0z_4\rangle}.
\end{align*}
Let us describe the localization map $h_i\col A\to A_{r_i}$ for each $r_i$. The ring $A_{r_i}$ is the localization of $A$ on the multiplicative system $R_i=\{\chi^u; u(r_i)=0\}$, for $i=1,2,3,4$. More precisely,
\begin{align*}
R_1=&\{ z_0^az_3^b; a,b\in \mathbb{Z}_{\geq0}\}&
R_2=&\{ z_2^az_3^b; a,b\in \mathbb{Z}_{\geq0}\}\\
R_3=&\{ z_0^az_1^b; a,b\in \mathbb{Z}_{\geq0}\}&
R_4=&\{ z_1^az_2^bz_4^c; a,b,c\in \mathbb{Z}_{\geq0}\}.
\end{align*}
Then, after eliminating some variables, we get
\begin{align*}
A_{r_1}=&\frac{k[z_0^{\pm1},z_3^{\pm1},z_4,x,y]}{\langle xy-z_0z_4\rangle}&
A_{r_2}=&\frac{k[z_2^{\pm1},z_3^{\pm1},z_0,x,y]}{\langle xy-z_0^2z_2z_3^{-1}\rangle}\\
A_{r_3}=&\frac{k[z_0^{\pm1},z_1^{\pm1},z_4,x,y]}{\langle xy-z_0z_4\rangle}&
A_{r_4}=&\frac{k[z_1^{\pm1},z_4^{\pm1},z_0,x,y]}{\langle xy-z_0z_4\rangle}.
\end{align*}
Define $I_1:=I_{e_0,r_1}^{(2)}=h_1^{-1}(y^2)$, $I_2:=I_{e_0,r_2}^{(2)}=h_2^{-1}(y^2)$, $I_3:=I_{e_0,r_3}^{(1)}=h_3^{-1}(y)$ and $I_4:=I_{e_0,r_4}^{(1)}=h_4^{-1}(y)$. 
Using \cite[CoCoA]{CoCoA} we get, as predicted by Lemma \ref{lem:Ie0},
\begin{align*}
I_1=&\langle y^2, yz_1,yz_2,yz_4,z_1^2,z_1z_4,z_4^2,z_2z_4,z_2^2\rangle\\
I_2=&\langle y^2, yz_0z_4,  yz_1, yz_4^2, yz_0^2, z_0^2z_4^2, z_0z_1z_4, z_0z_4^3, z_1^2, z_1z_4^2, z_4^4, z_0^2z_1, z_0^3z_4, z_0^4\rangle\\
I_3=&\langle y, z_2, z_3, z_4\rangle\\
I_4=&\langle y,z_0,z_3 \rangle
\end{align*}
and, as predicted by Proposition \ref{prop:Icap}, 
\[
I_1\cap I_2\cap I_3\cap I_4=\langle y^2, yz_1\rangle=\langle y\rangle\langle y,\chi^{u''_{e_0}}\rangle.\]
\end{Exa}

\section{The resolution of the geometric Abel map}\label{sec:Abelres}

\subsection{Compactified Jacobians and Abel maps}\label{sec:compjac}
Let $k$ be an algebraically closed field. 
 Let $(X,p_0,\ldots,p_n)$ be a pointed nodal curve defined over $k$. The \emph{dual graph} 
$\Gamma_X$ of $(X,p_0,\ldots, p_n)$ is the usual weighted graph with $n+1$ legs, where $V(\Gamma_X)$ is the set of irreducible components of $X$, the vertex-set $E(\Gamma_X)$ is the set of nodes of $X$, the weight of a vertex of $\Gamma_X$ is the genus of the normalization of the corresponding component, and $\text{leg}_{\Gamma_X}(i)$ is the vertex corresponding to the component containing $p_i$, for $i=0,\ldots, n$. We say that $(X,p_0,\ldots, p_n)$ is \emph{stable} if so is $\Gamma_X$.\par 

 The \emph{Jacobian} $\J(X)$ of $X$ is the scheme parametrizing the equivalence classes of invertible sheaves on $X$.  Given an integer $d$, the \emph{degree-$d$ Jacobian} $\J_d(X)$ of $X$ is the subscheme of $\J(X)$ parametrizing the equivalence classes of invertible sheaves of degree $d$ on $X$. In general, $\J_d(X)$  is neither proper nor of finite type. A better behaved parameter space is obtained by resorting to torsion-free  rank-$1$ sheaves and to stability conditions.\par

 Recall that a coherent sheaf $I$ on $X$ is \emph{torsion-free} if it has no embedded components, \emph{rank-1} if it is invertible on a dense open subset of $C$, and \emph{simple} if $\text{Hom}(I,I)=k$. The \emph{degree} of $I$ is $\deg I=\chi(I)-\chi(\O_X)$. There exists a scheme $\mathcal{S}pl_d(X)$ parametrizing simple torsion-free rank-$1$ sheaves of degree $d$ on $X$ (see \cite{AK}). Recall that  $\mathcal{S}pl_d(X)$   satisfies the existence part of the valuative criterion  and is connected  but, in general, not separated and only locally of finite type. To deal with a 
manageable piece of it, one can consider polarizations. \par

  We denote by $\{X_v;v\in V(\Gamma_X)\}$ the set of irreducible components of $X$. Let $d$ be an integer. A \emph{polarization of degree $d$ on $X$} is any function $\mu\col V(\Gamma_X)\ra \mathbb Q$ such that $\sum_{v\in V(\Gamma_X)}\mu(v)=d$.

\begin{Rem}
\label{rem:polgf}
The datum of a  degree-$d$ polarization on $X$ is equivalent to the datum of a  degree-$d$ polarization on $\Gamma_X$.
\end{Rem}

Let $I$ be a torsion-free rank-$1$  sheaf on $X$ of degree $d$. We define a pseudo-divisor $(\E_I,D_I)$ on $\Gamma_X$ as follows. The set $\E_I$ is precisely the set of edges corresponding to nodes where $I$ is not locally free and, for every $v\in V(\Gamma^\E)$, we set
\[
D_I(v)=\begin{cases}
        \deg(I|_{X_v}),&\text{ if $v\in V(\Gamma_X)$};\\
				-1,&\text{ if $v$ is exceptional}.
				\end{cases}
\]
We call $(\E_I,D_I)$ the \emph{multidegree} of $I$. Given a degree-$d$ polarization $\mu$ on $X$, we say that $I$ is \emph{$(p_0,\mu)$-quasistable} if so is $(\E_I,D_I)$. \par

The above notions naturally extend to families. 
 Let $\pi\col\X\to T$ be a family of pointed nodal curves with section $\Delta\col T\to\X$. 
A polarization $\mu$ on $\X$ is the datum of polarizations on the fibers of $\pi$ that are compatible with specializations. A coherent sheaf $\I$ over $\X$ is \emph{$(\Delta,\mu)$-quasistable} if, for any closed point $t\in T$, the restriction of $\I$ to the fiber $\pi^{-1}(t)$ is a torsion-free rank-$1$ and $(\Delta(t),\mu)$-quasistable sheaf. There is an algebraic space $\ol{\J}_{\pi,\mu}$ parametrizing $(\Delta,\mu)$-quasistable sheaves over $\X$. This algebraic space is proper and of finite type   (\cite[Theorems A and B]{Es01}) and it represents the contravariant functor $\mathbf{J}_{\pi,\mu}$ from the category of locally Noetherian $T$-schemes to sets, defined on a $T$-scheme $B$ by
\[
\mathbf{J}_{\pi,\mu}(B):=\{(\Delta_B,\mu_B)\text{-quasistable sheaves over } \X\times_T B\stackrel{\pi_B}\lra B\}/\sim
\]
where $\Delta_B$ and $\mu_B$ are the pullbacks to $\X\times_T B$ of the section $\Delta$ and the polarization $\mu$, and where $\sim$ is the equivalence relation given by $\I_1\sim \I_2$ if and only if there exists an invertible sheaf $\L$ on $B$ such that $\I_1\cong \I_2\otimes \pi_B^*\L$. \par

  The construction  of the relative Jacobian  can be extended to the universal setting. Let $\overline{\M}_{g,n}$ be the Deligne-Mumford stack parametrizing stable $n$-pointed genus-$g$ curves, and let $\M_{g,n}$ be its open locus. Let $\overline{\M}_{g,n+1}\to\overline{\M}_{g,n}$ be the universal family over $\overline{\M}_{g,n}$. Let $\J_{d,g,n}\ra\M_{g,n}$ be the universal degree-$d$ Jacobian parametrizing invertible sheaves of degree $d$ on smooth fibers of $\overline{\M}_{g,n+1}\to\overline{\M}_{g,n}$. 

For each universal degree-$d$ polarization $\mu$ over $\overline{\M}_{g,n+1}\to\overline{\M}_{g,n}$, there is  a proper and separated Deligne-Mumford stack $\ol{\J}_{\mu,g,n}$ over $\overline{\M}_{g,n}$ containing $\J_{d,g,n}$ as an open dense subset. For every scheme $S$,  
\[
\ol{\J}_{\mu,g,n}(S)=\frac{\left\{(\pi,\Delta,\I);\begin{array}{l} \pi\col \X\to S\text{ is a family of stable $n$-pointed genus-$g$ curves,}\\ \I \text{ is a $(\Delta,\mu)$-quasistable torsion-free rank-1 sheaf on $\X$}\end{array}\right\}}{\sim}
\] 
where $(\pi_1,\Delta_1,\I_1)\sim(\pi_2,\Delta_2,\I_2)$ if there exists an $S$-isomorphism $f\col \X_1\to \X_2$ and an invertible sheaf $\L$ on $S$, such that $\Delta_2=f\circ\Delta_1$ and $\I_1\cong f^*\I_2\otimes\pi_1^*\L$. We refer to \cite[Theorems A and B]{M15} and \cite[Corollary 4.4 and Remark 4.6]{KP} for more details on the stack $\ol{\J}_{\mu,g,n}$. 
	
	Given a family of $n+1$-pointed nodal curves $\pi\col \X\to T$, there exists a family of stable curves $\pi_{\st} \col \X_{\st}\to T$, called the \emph{stable reduction} of $\X$. This family is endowed with a morphism $\red\col \X\to \X_{\st}$ contracting chains of  rational curves. If $\mu$ is a polarization of $\X_{\st}$, then $\mu$ induces a polarization $\mu'$ on $\X$ (as in Remark \ref{rem:subdivision}). Then, we deduce the existence of a natural morphism (see \cite[Theorem 4.1 ]{EP}):
\begin{align}\label{eq:pipist}
\text{red}\col\overline{\J}_{\pi,\mu'}\ra& \overline{\J}_{\pi_{\st}, \mu}\\
  \I\mapsto & \red_*(\I)\nonumber.
\end{align}
	
  Let $\pi\col\X\to T$ be a pointed family of nodal curves with smooth generic fiber and $T$ irreducible. Let $\L$ be a degree-$d$ invertible sheaf on $\X$. For every degree-$d$ polarization $\mu$ on $\X$, there is a rational map 
\begin{align}\label{eq:abelmap}
\alpha_{\L,\mu}\col T\dasharrow &\ol{\J}_{\pi,\mu}\\
       t \mapsto& [\L|_{X_t}]\nonumber.
\end{align}

Note that $\alpha_\L$ is defined whenever $\L|_{X_t}$ is $(\Delta(t),\mu)$-quasistable. In particular,  $\alpha_{\L,\mu}$ is defined at every $t\in T$ such that $\pi^{-1}(t)$ is smooth. We call $\alpha_{\L,\mu}$ the \emph{Abel map relative to $\pi$ and $\L$} or, if no confusion may arise, simply the \emph{Abel map}.

  \subsection{Resolving the local Abel map}

In this section we construct a local resolution of the Abel map. The key ingredient is the  refinement map $\beta_\Gamma^\trop$ appearing in the resolution of the tropical Abel map $\alpha_{D_0}^\trop$, see Theorem \ref{thm:map}. The fan $\Sigma_{\Gamma,D_0,\mu}$ in Theorem \ref{thm:fan} will give rise to the toric blowup resolving the local Abel map.

Throughout the section we let $(X,p_0)$ be a pointed nodal curve and $(\Gamma,v_0)$ its dual graph. We will write $N_e$ for the node of $X$ corresponding to an edge $e\in E(\Gamma)$. We let $A$ be a local Noetherian $k$-algebra. We set $T:=\Spec A[[t_e]]_{e\in E(\Gamma)}$ and we let $0$ be the closed point of $T$. We fix a family  $\pi\col\X\to T$ of pointed curves with section $\Delta\col T\ra \X$ such that 
\begin{itemize}
\item[(1)] $T$ is integral;
\item[(2)] $\pi^{-1}(0)\cong X$ and $\Delta(0)=p_0$;
\item[(3)] the nodes of the central fiber $X$ can be smoothed independently, i.e., for each node $N_{e_0}$ of $X$ there are local \'etale coordinates $x_{e_0},y_{e_0}$ around $N_{e_0}\in \X$ such that
\begin{equation}
\label{eq:locnode}
\widehat{\O}_{\X,N_{e_0}}\cong\frac{\widehat{A}[[x_{e_0},y_{e_0},t_e]]_{e\in E(\Gamma)}}{(x_{e_0}y_{e_0}-t_{e_0})}.
\end{equation}
\end{itemize}
The prototype family to have in mind is the versal family over the versal deformation space of an $(n+1)$-pointed nodal curve.\par

  Let $\L$ be a degree-$d$ invertible sheaf on $\X/T$ and $\mu$ a degree-$d$ polarization on $\Gamma$. As explained in Remark \ref{rem:polgf}, the polarization $\mu$ induces a unique polarization $\mu$ on $X$ and hence on $\X$. Indeed, the dual graph of every fiber of $\X$ is given by a specialization $\iota$ of $\Gamma$, and the polarization on this fiber is simply $\iota_*(\mu)$. Consider the relative Jacobian $\ol{\J}_{\pi,\mu}$ parametrizing $(\Delta,\mu)$-quasistable torsion-free rank-$1$ sheaves on the fibers of $\pi$. 
The fibers of $\pi$ are smooth precisely over the locus in $T$ where $t_e\neq0$ for every $e\in E(\Gamma)$; this locus is a non-empty open subset of $T$. Therefore, the Abel map defined  in \eqref{eq:abelmap}: 
\[
\alpha_{\L,\mu}\col T\dasharrow\ol{\J}_{\pi,\mu},
\]
 is defined over this open subset and takes a point $\eta\in T$ to $[\L|_{\pi^{-1}(\eta)}]$ whenever the fiber $\pi^{-1}(\eta)$ is smooth.
	 Let $D_0$ be the multidegree of $\L|_X$. We consider 
\[
T_{\L,\mu}=\Tor_{A}(\Sigma_{\Gamma,D_0,\mu})
\quad\text{ and }\quad
\X_{\L,\mu}=\X\times_T T_{\L,\mu},
\]
where $\Sigma_{\Gamma,D_0,\mu}$ is the fan \eqref{def:fan} and $\Tor_A(-)$ is the toric variety \eqref{eq:TVA}. We denote by
\[
\beta_{\L,\mu}\col T_{\L,\mu}\ra T
\quad \text{ and } \quad 
\pi_{\L,\mu}\col \X_{\L,\mu}\ra T_{\L,\mu}
\]
 the natural maps.

	The goal of this section is to prove the following result.

\begin{Thm}
\label{thm:mainlocal}
	The rational map $\alpha_{\L,\mu}\circ\beta_{\L,\mu}\col T_{\L,\mu}\dashrightarrow \ol{\J}_{\pi,\mu}$ is defined everywhere, and, as a morphism of schemes, it is finite. In particular $T_{\L,\mu}$ is the normalization of the closure of the image of $\alpha_{\L,\mu}$.
\end{Thm}

We need some auxiliary results to prove Theorem \ref{thm:mainlocal}.
	As explained in Section \ref{sec:toricsec}, for every $\sigma\in \Sigma_{\Gamma,D_0,\mu}$ there exists an open immersion
\[
T_{\sigma}:=\Tor_{A}(\sigma)\to T_{\L,\mu}.
\]
  For consistency with the notation for toric varieties, we write $\chi^{e^\vee}:=t_e$ for every $e\in E(\Gamma)$. Recall that the map $f_\sigma\col T_\sigma\to T$ is induced by the ring homomorphism
\begin{equation}
\label{eq:fsigma}
f_\sigma^{\#}\col A[[\chi^{e^\vee}]]_{e\in E(\Gamma)}\to B_\sigma:=A[[\chi^{e^\vee}]]_{e\in E(\Gamma)}\otimes_{k[\chi^{e^\vee}]_{e\in E(\Gamma)}}k[S_{\sigma}].
\end{equation}
If $\chi^u\in k[\chi^{e^\vee}]_{e\in E(\Gamma)}$, then $\chi^u\otimes 1=1\otimes \chi^u$: we abuse notation and write $\chi^u$ for $\chi^u\otimes 1$. We note that $B_{\sigma}$ is a flat $A_\sigma=k[S_\sigma]$-algebra. Let $V(\sigma)\subset T_{\L,\mu}$ be the invariant subvariety of $T_{\L,\mu}$ associated to $\sigma$.
  Let $\eta_\sigma$ be the generic point of $V(\sigma)$ (note that $\eta_\sigma\in T_\sigma$). We define 
\begin{equation}
\label{eq:Xsigma}
\X_\sigma= \X\times_T T_\sigma
\quad \text{ and }\quad 
X_{\sigma}=\pi^{-1}_{\L,\mu}(\eta_\sigma).
\end{equation}

\begin{Lem}
\label{lem:sigmadual}
Let $\sigma\in \Sigma_{\Gamma,D_0,\mu}$ and $\Gamma'$ be the graph such that $\sigma=K_{\Gamma',\E,\phi}$. Then  $\Gamma'$ is the dual graph of $X_{\sigma}$.
\end{Lem}

\begin{proof}
Let $\Gamma\to\Gamma'$ be the specialization corresponding to $\sigma$. Note that $e^\vee(\sigma)\neq0$ if and only if $e\in E(\Gamma')$, while $e^\vee(\sigma)=0$ in the other cases (recall that $K_{\Gamma',\E,\phi}^\circ\subset \R^{E(\Gamma')}_{>0}$). Hence  $f_\sigma^{\#}(t_e)=f_\sigma^{\#}(\chi^{e^\vee})$ is invertible in $B_\sigma$ if and only if $e\notin E(\Gamma')$. This implies that the image of $T_\sigma$ via $f_\sigma$  meets the locus $t_e=0$ in $T$ if and only if $e\in E(\Gamma')$. Moreover, the locus $V(\sigma)\cap T_\sigma$ is given by the ideal $\langle \chi^u; u(\sigma)\neq0\rangle$, which contains $\chi^{e^\vee}$ for every $e\in E(\Gamma')$ because $e^\vee(\sigma)\neq0$ for any such $e$. In particular, the image of $V(\sigma)$ is contained in the locus $t_e=0$ for every $e\in E(\Gamma')$. This implies that the nodes of $X_\sigma$ are precisely the ones associated to the edges of $\Gamma'$.
\end{proof}
If $\sigma=K_{\Gamma',\E,\phi}$ is as in Lemma \ref{lem:sigmadual} and $e\in E(\Gamma')$, we denote as usual by $N_{e,\sigma}$ the node of $X_{\sigma}$ corresponding to $e$. 
  By Equations \eqref{eq:locnode} and \eqref{eq:Xsigma}, we have the following description of the completion of the local ring of $\X_{\sigma}$ at $N_{e,\sigma}$:
\begin{equation}
\label{eq:OXN}
\widehat{\O}_{\X_\sigma,N_{e,\sigma}}=\frac{\widehat{\O}_{T_\sigma,\eta_\sigma}[[x_e,y_e]]}{\langle x_ey_e-\chi^{e^\vee}\rangle}.
\end{equation}
 Let $\tau$ and $A_\tau$ be as in Equation \eqref{eq:AtauK} (also recall Construction \ref{cons:xy}). Note that $\widehat{\O}_{\X_\sigma,N_{e,\sigma}}$ is a flat $A_\tau$-algebra: indeed, $B_\sigma$ is a flat $A_\sigma$-algebra, and localizations and completions are flat, hence $\widehat{\O}_{T_\sigma,\eta_\sigma}$ is a flat $A_\sigma$-algebra.\par

	Let $r=(r_e)_{e\in E(\Gamma)}$  be a ray of $\Sigma_{\Gamma,D_0,\mu}(1)$ (we will abuse notation and write $r$ for both the ray and the cone associated to it) and $V(r)$ be the invariant subvariety of $T_{\L,\mu}$ associated to $r$. Then $V(r)\cap T_r$ is a Cartier divisor of $T_r$, since $T_r$ is smooth. By Proposition \ref{prop:Kdim}, the ray $r$ induces a specialization $\iota\col\Gamma\to \Gamma'$ where either $\Gamma'$ has a single vertex and a single loop, or two vertices and no loops; moreover, in the latter case, there is an acyclic flow $\phi$ on $\Gamma'$  such that $\iota_*(D_0)+\div(\phi)$ is $(\iota(v_0),\iota_*(\mu))$-quasistable and, by \eqref{eq:C}, for every $e,e'\in E(\Gamma')$,
\begin{equation}
\label{eq:phire}
\phi(e)r_e=\phi(e')r_{e'}.
\end{equation}
 By Lemma \ref{lem:sigmadual}, the dual graph of $X_r$ is $\Gamma'$. Moreover, if  $u_0(r)=1$ for some $u_0\in S_{r}$, then every $v\in S_{r}$ can be written as 
\[
v=v(r)u_0+(v-v(r)u_0), \text{ where } \pm(v-v(r)u_0)\in S_{r}.
\]
  This means that we can view $\chi^{u_0}$ as a local parameter for $T_r$, i.e., every $\chi^{v}$ is equal to $(\chi^{u_0})^{v(r)}$, up to multiplication by an invertible. In particular, for every $e\in E(\Gamma)$, 
\[
f_r^{\#}(\chi^{e^\vee})=\chi^u\chi^{e^{\vee}(r)u_0}
\]
 for some $u\in S_{r}$ such that $u(r)=0$ (recall \eqref{eq:fsigma}). This, together with Lemma \ref{lem:sigmadual}, implies that the induced family
 \
\[
\pi_r\col\X_r=\mathcal{X}\times_T T_r\to T_r
\]
 has fibers which are either smooth or have dual graph $\Gamma'$. More precisely, a fiber of $\pi_r$ is smooth over the locus in $T_r$ where $\chi^{u_0}\neq0$, while it has dual graph $\Gamma'$ over the locus in $T_r$ where $\chi^{u_0}=0$, and the last locus is precisely $V(r)\cap T_r$. \par
If $\sigma\in \Sigma_{\Gamma,D_0,\mu}$ and $r$ is a ray of $\sigma$, we have a Cartesian diagram
\[
\begin{tikzcd}
\X_r\arrow{r}\arrow{d}{\pi_r} &\X_\sigma\arrow{d}{\pi_\sigma}\\
T_r\arrow{r} & T_\sigma.
\end{tikzcd}
\]
  The open immersion $T_r\to T_\sigma$, is induced by the localization $A_\sigma\to A_r$. In fact, $T_r$ is a principal open subscheme of $T_\sigma$ induced by a regular function $\chi^{u_1}$ (see the last paragraph of Construction \ref{cons:xy}). Therefore, we can enrich the previous  diagram as follows:
\[
\begin{tikzcd}
 \Spec((\widehat{\O}_{\X_\sigma,N_{e,\sigma}})_{\chi^{u_1}})\arrow{rr}\arrow{dd}\arrow{rd}&&  \Spec(\widehat{\O}_{\X_\sigma, N_{e,\sigma}})\arrow{dd}\arrow{rd} &\\
& \X_r\arrow[rr, crossing over] &&\X_{\sigma}\arrow{dd}\\
\Spec((\widehat{\O}_{T_\sigma,\eta_\sigma})_{\chi^{u_1}})\arrow[rr]\arrow{rd} &&\Spec(\widehat{\O}_{T_\sigma,\eta_\sigma})\arrow{rd} &\\
&T_r \arrow[from=uu, crossing over]  \arrow{rr}&&T_\sigma\\
\end{tikzcd}
\]
where every face of the cube is Cartesian, except the left and right faces.

In figure \ref{fig:XrXsigma} below, we illustrate  all the revelant loci playing a role in what follows.
\begin{figure}[ht]
\begin{overpic}{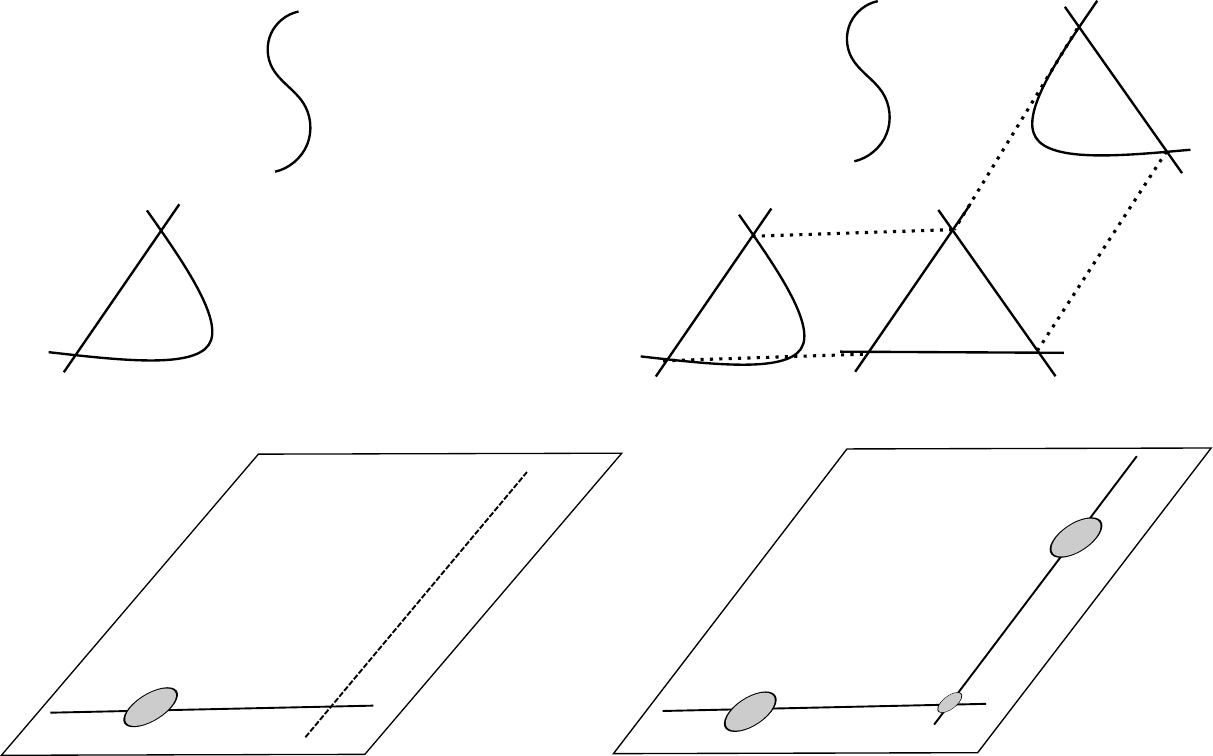}
\put(23,15){\huge$T_{r}$}
\put(70,15){\huge$T_\sigma$}
\put(8,6){$\eta_{r}$}
\put(58,6){$\eta_{r}$}
\put(16,5){\tiny$V(r)\cap T_{r_1}$}
\put(68,5){\small$V(r)$}
\put(78,2){$\eta_\sigma$}
\put(77,12){\small$V(r')$}
\put(85,20){$\eta_{r'}$}
\put(13,30){$X_{r}$}
\put(60,29.5){$X_{r}$}
\put(79,30.5){$X_\sigma$}
\put(89,47){$X_{r'}$}
\put(80,43){\tiny$N_{e_0,\sigma}$}
\put(91,60){\tiny$N_{e_0,r'}$}
\put(55,43){\tiny$N_{e_0,r}$}
\put(6,42.5){\tiny$N_{e_0,r}$}
\put(72,31){\tiny$N_{e_1,\sigma}$}
\put(50,35){\tiny$N_{e_1,r}$}
\put(1,35){\tiny$N_{e_1,r}$}
\put(87,31){\tiny$N_{e_2,\sigma}$}
\put(96,51){\tiny$N_{e_2,r'}$}
\end{overpic}
\caption{The Cartesian diagram induced by $\X_{r}\to \X_\sigma$.}
\label{fig:XrXsigma}
\end{figure}

Let $\Delta_r$ be the section of $\pi_r\col \X_r\ra T_r$ obtained by the pullback of the section $\Delta$ of $\pi\col \X\ra T$ and denote by 
\begin{equation}\label{eq:agar}
h_r\col\X_r=\X\times_T T_r\longrightarrow \X
\end{equation}
 the first projection.  

We write $I_{\Y|\X}$ for the ideal sheaf of a divisor $\Y$ of a scheme $\X$. A sheaf of ideals over $\X_r$, respectively over $\X_\sigma$, pullback to a sheaf of ideals over $\Spec((\widehat{\O}_{\X_\sigma,N_{e,\sigma}})_{\chi^{u_1}})$, respectively over $\Spec(\widehat{\O}_{\X_\sigma,N_{e,\sigma}})$. We will identify the last sheaf of ideals with an ideal of $(\widehat{\O}_{\X_\sigma,N_{e,\sigma}})_{\chi^{u_1}}$, respectively of $\widehat{\O}_{\X_\sigma,N_{e,\sigma}}$. We will write $I_{\Y|\X_r}(\widehat{\O}_{\X_\sigma,N_{e,\sigma}})_{\chi^{u_1}}$ to denote the ideal in $(\widehat{\O}_{\X_\sigma,N_{e,\sigma}})_{\chi^{u_1}}$ corresponding to the sheaf $I_{\Y|\X_r}$.

 For every $e\in E(\Gamma')$, the local equation of $\X_r$ in $(\widehat{\O}_{\X_\sigma, N_{e,\sigma}})_{\chi^{u_1}}$,  is $x_ey_e=\chi^u(\chi^{u_0})^{e^\vee(r)}$, where $u=e^\vee-e^\vee(r)u_0$ or, since $\chi^u$ is invertible, 
\begin{equation}
\label{eq:xyu0}
(x_e\chi^{-u})y_e=(\chi^{u_0})^{e^\vee(r)}.
\end{equation}
More precisely, we have (recall equation \eqref{eq:OXN}):
\begin{equation}\label{eq:xyu0bis}
(\widehat{\O}_{\X_\sigma,N_{e,\sigma}})_{\chi^{u_1}}\cong\frac{(\widehat{\O}_{T_\sigma,\eta_\sigma})_{\chi^{u_1}}[[x_e,y_e]]}{\langle x_ey_e-\chi^u(\chi^{u_0})^{e^\vee(r)}\rangle}.
\end{equation}

If $\Gamma'$ has a single vertex, then the fibers of $\X_r\to T_r$ are all irreducible.  If $\Gamma'$ has two vertices $v_0$ and $v_1$, then 
\[
\pi_r^{-1}(V(r)\cap T_r)=\Y_{v_0}\cup \Y_{v_1}
\]
 where $\Y_{v_0}$ (respectively, $\Y_{v_1}$) is the component corresponding to the vertex $v_0$ (respectively, $v_1$). 	 \par
	Finally, recall the definitions of $\tau$ in Construction \ref{cons:xy}, of $A_\tau$ in  \eqref{eq:AtauK}, and of $I_{e,r}$ in \eqref{eq:defIer}. Also, recall that $\widehat{\O}_{X_\sigma,N_{e,\sigma}}$ is a flat $A_\tau$-algebra.

\begin{Lem}
\label{lem:Dr}
Let $\sigma$ be a cone in $\Sigma_{\Gamma,D_0,\mu}$ and $r$ a ray of $\sigma$ associated to a cone $K_{\Gamma',\emptyset,\phi}$. Up to the pullback of a Cartier divisor of $T_r$,  there is an effective Cartier divisor $\Y_r$ on $\X_r$ such that $h_r^*\L|_{\X_r}(-\Y_r)$ is $(\Delta_r,\mu)$-quasistable and whose ideal sheaf satisfies
\[
I_{\Y_r|\X_r}(\widehat{\O}_{\X_r,N_{e,\sigma}})_{\chi^{u_1}}=\langle y_e^{\phi(e)}\rangle.
\]
 Moreover, if $\overline{\Y}_r$ is the closure of $\Y_r$ in $\X_\sigma$, then 
\[
I_{\overline{\Y}_r|\X_\sigma}\widehat{\O}_{\X_{\sigma},N_{e,\sigma}}=I^{(\phi(e))}_{e,r}\widehat{\O}_{\X_{\sigma},N_{e,\sigma}}.
\]
\end{Lem}

\begin{proof}

If $\Gamma'$ has a single vertex or $\phi=0$, then we define $\Y_r=0$, and we are done.\par
If $\Gamma'$ has two vertices, we define $\Y_r=m\Y_{v}$, where $m=\phi(e)e^{\vee}(r)$ and $v$ is the target of the edges in $\Gamma'$, with orientation induced by $\phi$. Note that $m$ and $v$ are well defined because all the edges of $\Gamma'$ have the same target ($\phi$ is acyclic), and $\phi(e)e^{\vee}(r)$ is independent of $e$ (by  \eqref{eq:phire}).\par
 Let us find the equations for $\Y_r$ locally around $N_{e_0,\sigma}$. By \eqref{eq:OXN}, the completion of the local ring of $\X_\sigma$ at a node $N_{e_0,\sigma}$ of the fiber $X_\sigma$ is
\[
\widehat{\O}_{\X_\sigma,N_{e_0,\sigma}}\cong\frac{\widehat{\O}_{T_\sigma,\eta_\sigma}[[x_{e_0},y_{e_0}]]}{\langle x_{e_0}y_{e_0}-\chi^{e_0^\vee}\rangle}.
\]
Localizing in $\chi^{u_1}$, we get
\[
(\widehat{\O}_{\X_\sigma,N_{e_0,\sigma}})_{\chi^{u_1}}\cong\frac{(\widehat{\O}_{T_\sigma,\eta_\sigma})_{\chi^{u_1}}[[x_{e_0},y_{e_0}]]}{\langle x_{e_0}y_{e_0}-\chi^{e_0^\vee}\rangle}.
\]

The Weil divisor $\Y_v$ is given locally in $(\widehat{\O}_{\X_\sigma,N_{e_0,\sigma}})_{\chi^{u_1}}$ by 
\[
I_{\Y_v|\X_r}(\widehat{\O}_{\X_\sigma,N_{e_0,\sigma}})_{\chi^{u_1}}=\langle y_{e_0},\chi^{u_0}\rangle
\]
(recall that $V(r)\cap T_r$ has equation $\langle \chi^{u_0}\rangle$ in $T_r$). The quotient
\[
\frac{(\widehat{\O}_{\X_\sigma,N_{e_0,\sigma}})_{\chi^{u_1}}}{I_{\Y_v|\X_r}(\widehat{\O}_{\X_\sigma,N_{e_0,\sigma}})_{\chi^{u_1}}}=\frac{(\widehat{\O}_{T_\sigma,\eta_\sigma})_{\chi^{u_1}}[[x_e]]}{\langle \chi^{u_0}\rangle}
\]
is a domain, because $V(r)\cap T_r$ is irreducible and smooth, then $I_{\Y_v|\X_r}(\widehat{\O}_{\X_\sigma,N_{e_0,\sigma}})_{\chi^{u_1}}$ is a prime ideal. Hence the ideal of $\Y_r$ in $(\widehat{\O}_{\X_\sigma,N_{e_0,\sigma}})_{\chi^{u_1}}$ is given by the $m$-th symbolic power of $\langle y_{e_0},\chi^{u_0}\rangle$. Since $(\O_{T_\sigma,\eta_\sigma})_{\chi^{u_1}}$ is  torsion free, and hence flat as a $k[t]$-module via the homomorphism taking $t$ to $\chi^{u_0}$, the homomorphism (recall \eqref{eq:xyu0}):
\begin{align*}
\frac{k[x,y,t]}{\langle xy-t^{e_0^\vee(r)}\rangle}&\to (\widehat{\O}_{\X_\sigma,N_{e_0,\sigma}})_{\chi^{u_1}}\\
t&\mapsto \chi^{u_0}\\
x&\mapsto x_{e_0}\chi^{-u}\\
y&\mapsto y_{e_0}
\end{align*}
is flat, because flatness is preserved under base change and completion. By Proposition \ref{prop:symbol},
\begin{equation}
\label{eq:IYXr}
I_{\Y_r|\X_r}(\widehat{\O}_{\X_\sigma,N_{e_0,\sigma}})_{\chi^{u_1}}=\langle y_{e_0}^{\phi(e_0)}\rangle\subset (\widehat{\O}_{\X_\sigma,N_{e_0,\sigma}})_{\chi^{u_1}}.
\end{equation}
Note that $Y_r$ is given by a unique equation locally around $N_{e_0,\sigma}$. To prove that $\Y_r$ is Cartier, it is enough to repeat the argument above, substituting $(\wh{\O}_{T_\sigma,\eta_\sigma})_{\chi^{u_1}}$ by $\wh{\O}_{T_\sigma, p}$ for every point $p\in V(r)\cap T_r$. This finishes the proof of the first equality of the lemma and the fact that $\Y_r$ is Cartier.\par
 Now, we prove the second equality of the statement. We know that
\begin{equation}
\label{eq:IYXsigma}
I_{\ol{\Y}_{r}|\X_{\sigma}}\widehat{\O}_{\X_\sigma,N_{e_0,\sigma}}=(I_{\Y_r|\X_r}(\widehat{\O}_{\X_\sigma,N_{e_0,\sigma}})_{\chi^{u_1}})\cap \widehat{\O}_{\X_\sigma,N_{e_0,\sigma}}.
\end{equation}
 Recall \eqref{eq:AtauK}, \eqref{eq:defIer}, \eqref{eq:defIe0rn} and Construction \ref{cons:xy}. Since $\widehat{\O}_{\X_\sigma,N_{e_0,\sigma}}$ is a flat $A_\tau$-algebra and $A_{\tau'}=(A_\tau)_{\chi^{u_1}}$ we have a commutative diagram
\begin{equation}
\label{eq:diaAtau}
\begin{tikzcd}
\widehat{\O}_{\X_\sigma, N_{e_0,\sigma}}\ar{r} & (\widehat{\O}_{\X_\sigma,N_{e_0,\sigma}})_{\chi^{u_1}}\\
A_\tau \ar{u}\ar{r} & A_{\tau'}\ar{u}
\end{tikzcd}
\end{equation}
By the definition \eqref{eq:defIe0rn} of $I_{e_0,r}^{(n)}$, and by \cite[Proposition 2.2]{E}, the map 
\[
\frac{A_{\tau}}{I_{e_0,r}^{(\phi(e_0))}}\stackrel{\cdot \chi^{u_1}}{\longrightarrow}\frac{A_{\tau}}{I_{e_0,r}^{(\phi(e_0))}}
\]
is injective. Since $\widehat{\O}_{\X_\sigma,N_{e_0,\sigma}}$ is flat over $A_\tau$, we have that
\[
\frac{\widehat{\O}_{\X_\sigma,N_{e_0,\sigma}}}{I_{e_0,r}^{(\phi(e_0))}\widehat{\O}_{\X_\sigma,N_{e_0,\sigma}}}\stackrel{\cdot \chi^{u_1}}{\longrightarrow}\frac{\widehat{\O}_{\X_\sigma,N_{e_0,\sigma}}}{I_{e_0,r}^{(\phi(e_0))}\widehat{\O}_{\X_\sigma,N_{e_0,\sigma}}}
\]
is also injective. We deduce
\begin{align*}
I_{e_0,r}^{(\phi(e_0))}\widehat{\O}_{\X_\sigma,N_{e_0,\sigma}}&=(I_{e_0,r}^{(\phi(e_0))}(\widehat{\O}_{\X_\sigma,N_{e_0,\sigma}})_{\chi^{u_1}})\cap \widehat{\O}_{\X_\sigma,N_{e_0,\sigma}} &\text{ by \cite[Prop 2.2]{E} }\\
&=((I_{e_0,r}^{(\phi(e_0))}A_{\tau'})(\widehat{\O}_{\X_\sigma,N_{e_0,\sigma}})_{\chi^{u_1}})\cap \widehat{\O}_{\X_\sigma,N_{e_0,\sigma}}& \text{ by \eqref{eq:diaAtau} }\\
&=(\langle y_{e_0}^{\phi(e_0)}\rangle(\widehat{\O}_{\X_\sigma,N_{e_0,\sigma}})_{\chi^{u_1}})\cap \widehat{\O}_{\X_\sigma,N_{e_0,\sigma}}& \text{ by \eqref{eq:defIe0rn} }\\
&=(I_{\Y_r|\X_r}(\widehat{\O}_{\X_\sigma,N_{e_0,\sigma}})_{\chi^{u_1}})\cap \widehat{\O}_{\X_\sigma,N_{e_0,\sigma}}&\text{ by \eqref{eq:IYXr} }\\
&=I_{\ol{\Y}_r,\X_\sigma}\widehat{\O}_{\X_\sigma,N_{e_0,\sigma}}&\text{ by \eqref{eq:IYXsigma}},
\end{align*}
concluding the proof of the second equality of the statement.

 It remains to show that $h_r^*\L|_{\X_r}(-\Y_r)$ is $(\Delta_r,\mu)$-quasistable. By the local equations \eqref{eq:IYXr} of $\Y_r$, the multidegree of $I_{\Y_r|\X_r}\otimes \O_{X_r}$ is precisely $\div(\phi)$. This means that the multidegree of $h_r^*\L|_{\X_r}(-\Y_r)\otimes\O_{X_r}$ is $D_0+\div(\phi)$, which is $(v_0,\mu)$-quasistable by the construction of $\phi$. Since quasistability is an open numerical condition (see \cite[Proposition 34]{Es01}), the sheaf $h_r^*\L|_{\X_r}(-\Y_r)$ is $(\Delta_r,\mu)$-quasistable. 
\end{proof}

We are now ready to prove Theorem \ref{thm:mainlocal}.

\begin{proof}[Proof of Theorem \ref{thm:mainlocal}]
Set $\alpha:=\alpha_{\L,\mu}$ and $\beta:=\beta_{\L,\mu}$. First note that, by  \eqref{def:fan}, $T_{\L,\mu}$ is covered by open subsets of the form $\Tor_A(\sigma)$, where $\sigma\in \Sigma_{\Gamma,\D_0,\mu}$ is a maximal cone, i.e., $\sigma=K_{\Gamma,\E,\phi}$ and $D_0+\div(\phi)$ is $(v_0,\mu)$-quasistable. It suffices  to show that $\alpha\circ\beta|_{\Tor_A(\sigma)}$ is a morphism.\par

 Fix $\sigma=K_{\Gamma,\E,\phi}$, for some $\E$, $\phi$.  For each extremal ray $r\in \sigma$ there is an open immersion $T_r\to T_\sigma$ (recall that $T_\sigma=\Tor_A(\sigma)$ and $T_r=\Tor_A(r)$). Recall \eqref{eq:agar}. By Lemma \ref{lem:Dr}, there is a Cartier divisor $\Y_r\subset \X_r$ such that $h_r^*\L|_{\X_r}(-\Y_r)$ is $(\Delta, \mu)$-quasistable.
 Let $\ol{\Y_r}$ be the closure of $\Y_r$ in $\X_\sigma=\X\times_{T}T_\sigma$, and define $\Y$ as the Weil divisor of $\X_\sigma$:
\begin{equation}
\label{eq:YX}
\Y:=\bigcup_{r\in \sigma(1)}\overline{\Y_r}.
\end{equation}

Recall that $e_0^s$ and $e_0^t$ denote the edges of $\Gamma'^\E$ over an edge $e_0\in\E$, and the definition \eqref{eq:u'u''} of $u_{e_0}',u_{e_0}''$. If an extremal ray $r\in \sigma(1)$ is associated to the cone $K_{\Gamma',\emptyset,\phi'}$ and $e_0^s$ (respectively, $e_0^t$) is contracted in the specialization $\Gamma^\E\to \Gamma'$, then, by Lemma \ref{lem:ue0}, we have $u'_{e_0}(r)=0$  (respectively, $u''_{e_0}(r)=0$); in this case, $\phi'(e_0)=\phi(e_0^t)$ (respectively, $\phi'(e_0)=\phi(e_0^s)$). Moreover, if $e_0\in \E$, then either $e_0^s$ or $e_0^t$ is contracted.\par
   By Lemma \ref{lem:Dr}, the divisor $\overline{\Y_r}$ is given, locally around the point $N_{e_0,\sigma}\in \X_\sigma$, by the ideal:
\begin{itemize}
\item $I_{e_0,r}^{(\phi(e_0))}\widehat{\O}_{\X_\sigma,N_{e_0,\sigma}}$, if $e_0\notin\E$;
\item $I_{e_0,r}^{(\phi(e_0^t))}\widehat{\O}_{\X_\sigma,N_{e_0,\sigma}}$, if $e_0\in \E$ and $u'_{e_0}(r)=0$;
\item $I_{e_0,r}^{(\phi(e_0^s))}\widehat{\O}_{\X_\sigma,N_{e_0,\sigma}}$, if $e_0\in \E$ and $u''_{e_0}(r)=0$.
\end{itemize}
If $u'_{e_0}(r)=u''_{e_0}(r)=0$, then, by Lemma \ref{lem:ue0}, $e_0$ is contracted in the specialization $\Gamma\to\Gamma'$, and in this case $e_0$ is not and edge of $\Gamma'$, which means, by Lemma \ref{lem:sigmadual}, that the node $N_{e_0}$ of $X$ is smoothed in $X_r$, and $I_{e_0,r}^{(\phi(e_0^t))}=I_{e_0,r}^{(\phi(e_0^s))}=A_\tau$.\par
  Therefore, by Proposition \ref{prop:Icap} and the fact that $\widehat{\O}_{\X_\sigma,N_{e_0,\sigma}}$ is a flat $A_\tau$-algebra, the divisor $\Y$ is given locally around the point $N_{e_0,\sigma}$ by the ideal:
\begin{itemize}
\item $\langle y_{e_0}\rangle^{\phi(e_0)}\widehat{\O}_{\X_\sigma,N_{e_0,\sigma}}$, if $e_0\notin\E$;
\item $\langle y_{e_0}\rangle^{\phi(e_0^s)}\left\langle y_{e_0},\chi^{u_{e_0}''}\right\rangle\widehat{\O}_{\X_\sigma,N_{e_0,\sigma}}$, if $e_0\in\E$.
\end{itemize}

We claim that $\L_\sigma:=h_\sigma^*\L\otimes\I_{\Y|\X_\sigma}$ is a $(\Delta,\mu)$-quasistable torsion-free rank-1 sheaf, where $h_\sigma\col \X_\sigma\ra \X$ is the first projection.  Recall \eqref{eq:Xsigma}. It is enough that $\L_{\sigma}\otimes \O_{X_\sigma}$ is $(\Delta(\eta_\sigma),\mu)$-quasistable, because quasistability is an open numerical condition and the multidegree of $\L_\sigma$ on any fiber is the specialization of the multidegree of $\L_\sigma$ on the fiber $X_\sigma$. By equation \eqref{eq:locnode} and Lemma \ref{lem:ue0}, 
\[
\widehat{\O}_{\X_\sigma,N_{e_0,\sigma}}\cong\frac{\widehat{\O}_{T_\sigma,\eta_\sigma}[[x_{e_0},y_{e_0}]]}{\langle x_{e_0}y_{e_0}-\chi^{e_0^\vee}\rangle}=\frac{\widehat{\O}_{T_\sigma,\eta_\sigma}[[x_{e_0},y_{e_0}]]}{\langle x_{e_0}y_{e_0}-\chi^{u'_{e_0}}\chi^{u''_{e_0}}\rangle}.
\]
By Remark \ref{rem:uinvertible}, we have that both $\chi^{u'_{e_0}}$ and $\chi^{u''_{e_0}}$ are in the maximal ideal of $\widehat{\O}_{T_\sigma,\eta_\sigma}$. It follows from Proposition \ref{prop:Ifree} that $\I_{\Y|\X_\sigma}\otimes \O_{X_\sigma}$ is a torsion-free rank-$1$ sheaf. We now show that, locally around the node $N_{e_0,\sigma}$, the following properties hold:
\begin{enumerate}
\item  If $e_0\notin\E$, then $\deg(\I_{\Y|\X_\sigma}\otimes \O_{X_\sigma^s})=-\phi(e_0)$, and so $\deg(\I_{\Y|\X_\sigma}\otimes \O_{X_\sigma^t})=\phi(e_0)$;
\item  if $e_0\in\E$, then $\deg(\I_{\Y|\X_\sigma}\otimes \O_{X_\sigma^s})=-\phi(e_0^s)$, and so $\deg(\I_{\Y|\X_\sigma}\otimes \O_{X_\sigma^t})=\phi(e_0^s)+1=\phi(e_0^t)$,
\end{enumerate}
where $X_\sigma^s$ and $X_\sigma^t$ are, respectively, the components of $X$ corresponding to the source and the target of $e$ (if $\phi(e)=0$ then $\I_{\Y|\X_\sigma}\otimes \O_{X_\sigma}=\langle 1\rangle$ locally around $N_{e,\sigma}$, hence there is no ambiguity in the choice of the source and target).  Indeed, (1) follows from the fact that, locally around $N_{e_0,\sigma}$, we have $\I_{\Y|\X_\sigma}=\langle y_{e_0}\rangle^{\phi(e_0)}$. For (2), note that Proposition \ref{prop:Ifree} implies the existence of an exact sequence
\[
0\to \O_{N_{e_0,\sigma}}\to \I_{\Y|\X_\sigma}\otimes \O_{X_\sigma^s}\to \I_{\Y|\X_\sigma}\cdot\O_{X_\sigma^s}\to 0
\]
where, locally around $N_{e_0,\sigma}$, we have $\I_{\Y|\X_\sigma}\cdot\O_{X_\sigma^s}=\langle y_{e_0}\rangle^{\phi(e_0^s)+1}$. Hence we get 
\[
\deg(\I_{\Y|\X_\sigma}\otimes\O_{X_\sigma^s})=-(\phi(e_0^s)+1)+1=-\phi(e_0^s).
\] 
Therefore, when restricted to $X_\sigma$, the ideal sheaf $\I_{\Y|\X_\sigma}$ has multidegree precisely $\div(\phi)$. This means that $\L_\sigma\otimes\O_{X_\sigma}$ has multidegree $D_0+\div(\phi)$ and hence it is $(\Delta(\eta_\sigma),\mu)$-quasistable. This completes the proof of  the claim. \par

We deduce from the claim that $\alpha\circ\beta$ is a morphism, so  there remains to prove that $\alpha\circ\beta$ is finite. Note that $\alpha\circ\beta$ is proper, because $\beta$ and $\ol{\J}_{\pi,\mu}\to T$ are proper. Then it suffices  that $\alpha\circ\beta$ is quasi-finite (i.e., it has zero dimensional fibers).\par
		Let $Z$ be the locus in $\ol{\J}_{\pi,\mu}$ where the fibers of $\alpha\circ\beta$ are positive-dimensional. By Chevalley's upper semi-continuity theorem, $Z$ is closed. Since $\ol{\J}_{\pi,\mu}\to T$ is proper, the image of $Z$ in $T$ must contain the closed point $0$ or be empty. \par
		 Let $\sigma:=K_{\Gamma,\E,\phi}$ be a cone in $\Sigma_{\Gamma,D_0,\mu}$; in particular, $\sigma\cap \R^{E(\Gamma)}_{>0}\neq\emptyset$. If $O(\sigma)$ is the orbit in $T_{\L,\mu}$ associated to $\sigma$, then $\beta(O(\sigma))=0$. Moreover, $V(\sigma)$ is proper by \cite[Proposition 3.2.7]{CLS}. Let $W$ be a fiber of the map 
\[
(\alpha\circ\beta)|_{O(\sigma)}\col O(\sigma)\to \ol{\J}_{\pi,\mu}.
\]

Let us prove that $\dim(W)=0$. By contradiction, assume that $\dim(W)\geq 1$. Then the closure $\ol{W}$ of $W$ in $V(\sigma)$ cannot be contained in $O(\sigma)$, since $O(\sigma)$ is a torus. Hence $\ol W$ must intersect another orbit $O(\tau)$, with $\sigma \prec \tau$ and $\sigma\ne \tau$. Since $\alpha\circ\beta(\ol{W})=\alpha\circ\beta(W)$, the multidegree of the sheaves $\I_{\Y|\X_\sigma}\otimes \O_{X_\sigma}$ and $\I_{Y|\X_\tau}\otimes \O_{X_\tau}$ are equal. However, $\sigma\prec\tau$ corresponds to a specialization $(\Gamma,\E',\phi')\to (\Gamma,\E,\phi)$ (via Proposition \ref{prop:union}).		
But the multidegrees of  $\I_{\Y|\X_\sigma}\otimes \O_{X_\sigma}$ and $\I_{Y|\X_\tau}\otimes \O_{X_\tau}$  are, respectively, $(\E,\div(\phi))$ and $(\E',\div(\phi'))$. Hence  $\E=\E'$, which implies $\phi=\phi'$ and $\sigma=\tau$, contradicting that $\sigma\ne\tau$.\par

		The preimage $\beta^{-1}(0)$ of the closed point is the union of a finite number of orbits. More precisely, it is the union of the orbits $O(\sigma)$, for every $\sigma\in \Sigma_{\Gamma,D_0,\mu}$ such that $\sigma\cap \R^{E(\Gamma)}_{>0}\neq \emptyset$, and all these cones are of the form $\sigma=K_{\Gamma,\E,\phi}$. By the argument in the previous paragraph, the closed point $0$ of $T$ cannot be in the image of $Z$. Hence $Z$ is empty, so we deduce that $\alpha\circ\beta$ is quasi-finite, and hence finite.
\end{proof}

Let $\st(\Gamma)$ be the stabilization of $\Gamma$ (recall the definition before Example \ref{exa:stgraph}). Let $\mu$ be a polarization on $\st(\Gamma)$. We can define an induced polarization $\mu'$ on $\Gamma$ by 
\[
\mu'(v)=\begin{cases}
         \mu(v),&\text{ if $v \in V(\st(\Gamma))$;}\\
	0,       &\text{ otherwise}.
					\end{cases}
\]		
We can consider the stable reduction  $\pi_{\st}\col \X_{\st}\to T$ of $\pi\col \X\ra T$ and the natural morphisms $\text{red}\col \X\ra \X_{st}$ and  
$\text{red}\col\ol{\J}_{\pi,\mu'}\to \ol{\J}_{\pi_{\st},\mu}$ described in \eqref{eq:pipist}. 

\begin{Cor}
Let $\mu$ be a degree-$d$ polarization on $\st(\Gamma)$, $\L$ a degree-$d$ invertible sheaf on $\X$ and $\alpha_{\L,\mu}\col T\dashrightarrow\ol{\J}_{\pi_{\st},\mu}$ the rational map induced by $\L$. Then the rational map $\alpha_{\L,\mu}\circ\beta_{\L,\mu'}\col T_{\L,\mu'}\dashrightarrow \ol{\J}_{\pi_{\st},\mu}$ is defined everywhere, i.e., it is a morphism of schemes.
\end{Cor}

\begin{Exa}
\label{exa:divisor}
Keep the notation in Examples \ref{exa:fan} and \ref{exa:ideal} and define $\sigma:=K_{\Gamma,\E,\phi}$. Let $X$ be a curve with dual graph $\Gamma$ and $\pi\col \X\to T=\Spec(k[[t_0,t_1,t_2]])$ be a family of curves with central fiber $X$. Let $\L$ be an invertible sheaf on $\X$ with multidegree $(4,-4)$. Assume that the local equation of $\X$ around the node $N_i$ of $X$ is $x_iy_i=t_i$.\par
  By Example \ref{exa:ideal}, $r_2=(2,1,2)$ and 
\[
(\widehat{\O}_{\X_{r_2}, N_{e_i,\sigma}})_{z_2z_3}\cong\frac{k[[z_2^{\pm1},z_3^{\pm1},z_0,x_i,y_i]]}{\langle x_iy_i-t_i\rangle}
\]
where $z_0=\chi^{(0,-1,1)}=t_1^{-1}t_2$, $z_2=\chi^{(0,2,-1)}=t_1^2t_2^{-1}$ and $z_3=\chi^{(-1,0,1)}=t_0^{-1}t_2$. Hence $t_0=z_0^2z_2z_3^{-1}$, $t_1=z_0z_2$ and $t_2=z_0^2z_2$. This means that $z_0$ is a local parameter of $T_r$ and the local equations of $\X_r$ around the nodes of $X_r$ are $x_1y_1=z_0^2$, $x_2y_2=z_0$, $x_3y_3=z_0^2$ (up to invertibles). Note that, as expected by equation \eqref{eq:xyu0}, the exponents  of $z_0$ are precisely the coordinates of $r_2$. Let $\Y_{v_1}$ be the Weil divisor in $\X_r$, where ${v_1}\in V(\Gamma)$ is the target of the edges in $\Gamma$. The divisor $\Y_{v_1}$ is given locally around the nodes $N_{e_0,r_2}$, $N_{e_1,r_2}$, $N_{e_2,r_2}$, respectively, by $\langle y_1,z_0\rangle$, $\langle y_2\rangle$, $\langle y_3, z_0\rangle$. Since the flow associated to $r_2$ is $\phi_2:=(1,2,1)$ (see Figure \ref{fig:fray}), it is clear that $\phi_2(e_i)e_i^\vee(r_2)=2$. Hence $\Y_{r_2}=2\Y_{v_1}$ is given locally around the nodes $N_{e_0,r_2}$, $N_{e_1,r_2}$, $N_{e_2,r_2}$, respectively, as $\langle y_1\rangle$, $\langle y_2^2\rangle $, $\langle y_3 \rangle$. Then the ideal sheaf $\I_{\Y_{r_2}|\X_r}$, when restricted to $X$, has multidegree precisely $(-4,4)=\div(\phi_2)$.\par

  Analogously, we can find the local equations for the ideal sheaves $\I_{\Y_{r_1}|\X_{r_1}}$, $\I_{\Y_{r_3}|\X_{r_3}}$, $\I_{\Y_{r_4}|\X_{r_4}}$ which, locally around the nodes $N_{e_0,r_1}$, $N_{e_0,r_3}$ and $N_{e_0, r_4}$, are given by the ideals $\langle y_0 \rangle$, $\langle y_0^2\rangle$ and $\langle y_0^2\rangle$, respectively. Then, the ideal sheaves $I_{\overline{\Y}_{r_i}|\X_{\sigma}}$ locally around $N_{e_0,\sigma}$ are $I_i$, where $I_i$ are defined in Example \ref{exa:ideal}. Hence we get that $\I_{\Y|\X_{\sigma}}$ is given by $\langle y\rangle \langle y,z_1\rangle$ locally around $N_{e_0,\sigma}$ (as in the proof of Theorem \ref{thm:mainlocal}).\par

  Restricting to $X_\sigma$, the sheaf $\I_{\Y}\otimes \O_{X_\sigma}$ is torsion-free rank-$1$ (by Proposition \ref{prop:Ifree}); it is locally free around $N_{e_2,\sigma}$ but not around $N_{e_0,\sigma}$ and $N_{e_1,\sigma}$. By Lemma \ref{lem:sigmadual}, we have that $X_{\sigma}$ has two components, which we denote by $X_{v_0}$ and $X_{v_1}$, where $v_0$ and $v_1$ are the vertices of $\Gamma$ with $v_0$ been the source of the edges. By Proposition \ref{prop:Ifree}, the restriction of $\I_{\Y|\X_{\sigma}}$ to $X_{v_0}$ is 
\[
\I_{\Y|\X_\sigma}\otimes \O_{X_{v_0}}=\O_{X_{v_0}}(-2N_{e_0,\sigma}-2N_{e_1,\sigma}-N_{e_2,\sigma})\oplus \O_{N_{e_0,\sigma}}\oplus \O_{N_{e_1,\sigma}},
\]
whereas
\[
\I_{\Y|\X_{\sigma}}\otimes\O_{X_{v_1}}=\O_{X_{v_1}}(N_{e_0,\sigma}+N_{e_1,\sigma}+N_{e_2,\sigma})\oplus \O_{N_{e_0,\sigma}}\oplus \O_{N_{e_1,\sigma}}.
\]
Hence, the multidegree of $\I_{\Y|\X_\sigma}\otimes\O_{X_\sigma}$ is precisely $(\{e_0,e_1\},\div(\phi))$ which means that $h_\sigma^*\L\otimes \I_{Y|\X_\sigma}\otimes\O_{X_\sigma}$ has multidegree $(\{e_0,e_1\},D)$ with $D(v_0)=D(v_1)=1$ and $D(v_{e_0})=D(v_{e_1})=-1$, hence it is $(\Delta,\mu)$-quasistable.
\end{Exa}

  \subsection{Resolving the universal Abel map}\label{sec:abelgeo}
Fix nonnegative integers $g$ and $n$. Let $\mathcal A=(a_0,\ldots,a_n,m)$ be a sequence of $n+2$ integers. Set $d:=\sum_{i=0}^na_i+m(2g-2)$. Let $\mu$ be a universal degree-$d$ polarization. 

Let $\pi\col\C_{g,n+1}\to\overline{\M}_{g,n+1}$ be the universal family and define 
\[
\L:=\omega_{\pi}^m\otimes\O_{\C_{g,n+1}}\left(\sum_{i=0}^n a_i\text{Im}(\Delta_i)\right),
\]
 where $\Delta_i$ are the canonical sections of $\pi$. Consider the chain of maps
\[
\alpha_{\mathcal A,\mu}\col \overline{\M}_{g,n+1}\stackrel{\alpha_{\L,\mu}}{\dashrightarrow} \overline{\J}_{\mu,g,n+1}\stackrel{\text{red}}{\longrightarrow}\ol{\J}_{\mu,g}
\]
where $\alpha_{\L,\mu}$ is the rational Abel map relative to $\pi$ and $\L$ (usually called \emph{Abel section}), and $\text{red}$ is the map defined in \eqref{eq:pipist}. 

Let $\overline{\M}_{\mathcal A,\mu}$ be the normalization of the closure of the image of the Abel section $\alpha_{\mathcal L,\mu}$. Since closed substacks and normalizations of Deligne-Mumford stacks are Deligne-Mumford  (see \cite[Examples 8.12 and 8.13, Chapter XII]{ACG}), the stack $\overline{\M}_{\mathcal A,\mu}$ is Deligne-Mumford. We can consider the induced birational morphism 
\[
\beta_{\mathcal A,\mu} \col \ol{\M}_{\mathcal A,\mu}\ra \ol{\M}_{g,n+1}.
\]

Consider the \'etale covering of $\overline{\M}_{g,n+1}$ given by 
\[
 T_X:=\Spec(\widehat{\O}_{\overline{\M}_{g,n+1},[X]})\stackrel{f_X}{\longrightarrow}\overline{\M}_{g,n+1}
\]
 for $[X]$ running over the points $[X]\in\overline{\M}_{g,n+1}$. Let 
\[
\pi_X\col T_X\times_{\overline{\M}_{g,n+1}} \C_{g,n+1}\to T_X 
\;\;
\text{ and }
\;\;
g_X\col T_X\times_{\overline{\M}_{g,n+1}} \C_{g,n+1}\to \C_{g,n+1}
\]
be the first projection and the second projection, respectively, and consider the invertible sheaf on the family $\pi_X$ given by 
\[
\L_X:=g_X^*\L.
\] 
We have a natural isomorphism
\begin{equation}
\label{eq:isolocalring}
\widehat{\O}_{\overline{\M}_{g,n+1},[X]}\cong A[[t_e]]_{e\in E(\Gamma_X)},
\end{equation}
where $A$ is a local $k$-algebra. Hence we can apply Theorem \ref{thm:mainlocal} to construct a blowup $\beta_X\col T_{\L_X,\mu}\to T_X$ such that 
\[
\alpha_{\mathcal{L}_X,\mu}\circ \beta_X\col T_{\L_X,\mu}\to \overline{\J}_{\pi_X,\mu}= \overline{\J}_{\mu,g,n+1}\times_{\overline{M}_{g,n+1}} T_X
\]
is a finite morphism. Here, $\alpha_{\L_X,\mu}$ is the Abel map relative to $\pi_X$ and $\L_X$. Hence  $T_{\L_X,\mu}$ is the normalization of the closure of the image of $T_X\dashrightarrow \J_{\pi_X,\mu}$ because $T_{\L_X,\mu}$ is toric and hence normal.

\begin{Thm}
\label{thm:mainglobal}
The collection $T_{\L_X,\mu}$ form an \'etale atlas for $\overline{\M}_{\mathcal A,\mu}$ and the rational map $\alpha_{\mathcal L,\mu}\circ\beta_{\mathcal A,\mu}\col \overline{\M}_{\mathcal A,\mu}\dashrightarrow\overline{\J}_{\mu,g,n+1}$ is defined everywhere, i.e., it is a morphism of Deligne-Mumford stacks. This morphism is \'etale locally presented as the morphism of schemes $T_{\L_X,\mu}\to \overline{\J}_{\pi_X,\mu}$. In particular,  the rational map
\[
\alpha_{\mathcal{A},\mu}\circ \beta_{\mathcal A,\mu}\col \ol{\M}_{\mathcal{A},\mu}\dashrightarrow \ol{\J}_{\mu,g}. 
\]
is defined everywhere, i.e., it is a morphism of Deligne-Mumford stacks. 
\end{Thm}
\begin{proof}
The result follows by the fact that the normalization commutes with \'etale base change, and by taking the composition with the map $\red\col \ol{\J}_{\mu,g,n+1}\to \ol{\J}_{\mu,g}$ in \eqref{eq:pipist}.
\end{proof}

\begin{Rem}
\label{rem:universal}
Recall that $\overline{\J}_{\mu,g,n+1}$ is functorial. Therefore, the morphism $\alpha_{\mathcal{A},\mu}\circ\beta_{\mathcal A,\mu}$ is induced by a  universal $(\Delta,\mu)$-quasistable torsion-free rank-$1$ sheaf $\I$ on the family $\pi\col\overline{\M}_{\mathcal{A},\mu}\times_{\overline{\M}_{g,n+1}}\ol{\M}_{g,n+2}\ra \ol{\M}_{\mathcal A,\mu}$. Moreover the pullback of $\I$ to $T_{\L_X,\mu}\times_{\ol{\M}_{g,n+1}}\ol{\M}_{g,n+2}$ is isomorphic to  $\L_X\otimes \I_{\Y|\X}$ (recall Equation \eqref{eq:YX}), up to tensoring with the pullback of an invertible sheaf on $T_{\L_X,\mu}$.
\end{Rem}

\begin{Rem}
Let $\mu$ be a genus-$g$ degree-$d$ canonical polarization. Consider a collection $\mathcal{A}=(a_0,\ldots,a_n,m)$ such that $d=m(2g-2)+\sum_{i=0}^n a_i$. There exists a natural diagram
\[
\begin{tikzcd}
\ol{\J}_{\mu,g} \arrow{r}\arrow{d}& \ol{\J}_{d,g}\arrow{d}\\
\ol{\M}_{g,1}\arrow{r} &\ol{\M}_g
\end{tikzcd}
\]
where $\ol{\J}_{d,g}$ is Caporaso's compactified Jacobian  (see \cite[Lemma 6.2]{AP2}). Hence we get an induced Abel map $\ol{\M}_{\mathcal{A},\mu}\to \ol{\J}_{d,g}$.
\end{Rem}

\section{The double ramification cycle and further problems}\label{sec:app}

\subsection{The double ramification cycle}\label{sec:double}
Fix nonnegative integers $g$ and $n$.
Let $\mathcal A=(a_0,\ldots, a_n,m)$ be a sequence of $n+2$ integers such that $\sum_{0\le i\le n} a_i+m(2g-2)=0$ and let $\mu$ be the canonical degree-$0$ polarization.

By Theorem \ref{thm:mainglobal}, we can consider the following commutative diagram:
\[
\begin{tikzcd}
&&\overline{\J}_{\mu,g}\arrow{d}& \Z_{\mu,g}\arrow{l}\\
\overline{\M}_{\mathcal A,\mu}\arrow{r}{\beta_{\mathcal A}}\arrow[bend left=15]{rru}{\alpha_{\mathcal A}\circ \beta_{\mathcal A}}&\overline{\M}_{g,n+1}\arrow{r}\arrow[dashed]{ur}{\alpha_{\mathcal A}}&\overline{\M}_{g,1}\arrow{ur}{\mathcal{\O}}
\end{tikzcd}
\]
where $\Z_{\mu,g}$ is the image of the zero section $\O$ of $\ol{\J}_{\mu,g}\to \overline{\M}_{g,1}$. 
 The \emph{double ramification cycle} is the cycle
\[
DRC_{\mathcal A}:=\beta_{\mathcal A*}(\alpha_{\mathcal A}\circ\beta_{\mathcal A})^*([\Z_{\mu,g}]))\in \text{Chow}^g(\overline{\M}_{g,n+1}).
\]

By \cite[Theorem 1.3]{H}, $DRC_{\mathcal A}$ coincides with the double ramification cycle as defined in \cite{L1,L2, GV} and computed in \cite{JPPZ}.
As in Remark \ref{rem:universal}, let $\I$ be the universal sheaf over the family 
\[
\pi\col\overline{\M}_{\mathcal A,\mu}\times_{\overline{\M}_{g,n+1}}\overline{\M}_{g,n+2}\ra\ol{\M}_{\mathcal A,\mu}.
\]

\begin{Thm}\label{thm:drc}
We have 
\[
DRC_{\mathcal A}=\beta_{\mathcal A*}(c_g(R^1\pi_*(\I))).
\]
\end{Thm}

\begin{proof}
By \cite[Corollary 2.2]{D}, we have $\I|_X\cong \O_X$ if and only if $h^0(\I|_X)=1$ for every stable $(n+1)$-pointed curve $X$. Hence, choosing a sufficiently relative ample divisor $D$ of relative degree $d$ on 
$\pi\col\overline{\M}_{\mathcal A,\mu}\times_{\overline{\M}_{g,n+1}}\overline{\M}_{g,n+2}\to\overline{\M}_{\mathcal A,\mu}$, we get an exact sequence
\[
0\to \I\to \I(D)\to \I(D)|_D\to 0.
\]
Taking pushforward by $\phi$, we obtain an exact sequence
\[
0\to \pi_*(\I)\to\pi_*(\I(D))\stackrel{\varphi}{\to}\pi_*(\I(D)|_D)\to R^1\pi_*(I)\to 0.
\]
Hence, $DRC_{\mathcal A}=D_{d-g}(\varphi)$, where $D_{d-g}(\varphi)$ is the degeneracy class of $\varphi$. By the Thom-Porteous formula, and the fact that $\pi_*(I)$ is either $0$ or trivial (when $\mathcal A=(0,\ldots,0,0)$), we get the result.
\end{proof}

\begin{Rem}
If instead of the canonical degree-$0$ polarization $\mu$, we consider a polarization $\mu'$ which is a ``nondegenerate small perturbation'' of $\mu$ (in the sense of \cite[Definition 2.3]{HKP}), then Theorem \ref{thm:drc} (with the polarization $\mu'$) is a consequence of \cite[Equations (5) and (6), and Theorem 3.1]{HKP}.
\end{Rem}

  Recently, in \cite{PR}, it was given an explicit formula for the Chern character $\text{ch}(R\pi'_*(\L'))$, where $\pi'\col\ol{\M}_{g,n+2}\to\ol{\M}_{g,n+1}$ is the universal family, and
	\[
	\L'=\omega_{\pi'}^m\otimes \O_{\ol{\M}_{g,n+2}}(\sum a_i\text{Im}(\Delta_i))\otimes \O_{\ol{\M}_{g,n+2}}(\Y'),
	\]
	with $\Y'$ a sum of boundary divisors. A similar formula for the universal sheaf $\I$ of $\ol{\M}_{\mathcal{A},\mu}$ would give an explicit formula for the double ramification cycle in terms of the tautological classes (see also the discussion in Subsection \ref{subsec:moduli}).\par

Now we give a tropical analogue of the locus underlying the double ramification cycle. Let $DRL_{\mathcal{A}}^\trop$ be the preimage in $\ol{\M}^\trop_{\mathcal{A},\mu}$  of the zero section of $\ol{\J}^\trop_{\mu,g}$ under the morphism of generalized cone complexes $\alpha_{\mathcal{A}}^\trop\circ\beta_{\mathcal{A}}^\trop\col \ol{\M}^\trop_{\mathcal{A},\mu}\ra \ol{\J}^\trop_{\mu,g}$ (recall Theorem \ref{thm:abeltrop}). 

\begin{Prop}
The locus $DRL_{\mathcal{A}}^\trop$ is the generalized cone complex
\[
DRL_{\mathcal{A}}^\trop=\underset{\longrightarrow}{\lim} K_{\Gamma,\emptyset,\phi}
\]
where  $(\Gamma,\emptyset,\phi)$ runs through all genus-$g$ stable weighted graphs $\Gamma$ with $n+1$ legs and a flow $\phi$ on $\Gamma$ such that
 $D_{\Gamma,\mathcal{A}}+\div(\phi)=0$ (recall \eqref{eq:DGamma}).
\end{Prop}

\begin{proof}
The divisor $\D_{X,\mathcal{A}}:=m\omega_X+\sum_{i=0}^n a_ip_i$ on a genus-$g$ stable $n+1$-pointed tropical curve $(X,p_0,\ldots,p_n)$ is equivalent to $0$ if and only if there exists a rational function $f$ on $X$ such that $\D_{X,\mathcal{A}}+\div(f)=0$.  This function $f$ will induce a flow $\phi$ on the stable model $\Gamma_{X}$ of $X$ such that  $D_{\Gamma_X, \mathcal{A}}+\div(\phi)=0$, hence $[X]\in K_{\Gamma_X,\emptyset,\phi}$ with $D_{\Gamma_X,\mathcal{A}}+\div(\phi)=0$. \par
Vice versa, let $\phi$ be a flow on $\Gamma$ such that $D_{\Gamma,\mathcal{A}}+\div(\phi)=0$, and $X$ be a tropical curve in $K_{\Gamma,\emptyset,\phi}$. Since $[X]\in K_{\Gamma,\emptyset,\phi}$, the flow $\phi$ induces a rational function $f$ on $X$. The condition $D_{\Gamma,\mathcal{A}}+\div(\phi)=0$ implies $\D_{X,\mathcal{A}}+\div(f)=0$, and we are done.
\end{proof}

\begin{Rem}
We note that $DRL_{\mathcal{A}}^\trop$ is precisely the tropical spaces implicitly considered in \cite{H} (see also Section \ref{sec:comparison}).
\end{Rem}

\subsection{Geometric properties of $\ol{\M}_{\mathcal{A},\mu}$.}
\label{subsec:moduli}
 The stack $\ol{\M}_{\mathcal{A},\mu}$ is normal but it is usually singular. One can see that $\ol{\M}_{A,\mu}$, with the log-structure induced by the atlas $T_{\L_X,\mu}$,  is log-smooth. 
  It would be interesting to give a modular description of $\ol{\M}_{\mathcal{A},\mu}$. A first guess would be the functor $\mathcal F$ from the category of schemes to the category of grupoids mapping a scheme $B$ to
\[
\mathcal F(B)=\{(\pi\col \X\to B, T)\},
\]
where $\pi\col \X \to B$ is a family of $(n+1)$-pointed stable curves with sections $\Delta_0,\ldots, \Delta_n$,  and $T$ is a twister  on $\pi$, in the sense of \cite[Subsection 3.1]{CEG}, such that $\omega_{\pi}^m\otimes \O_\X(\sum \Delta_i(B))\otimes T$ is $(\Delta_0,\mu)$-quasistable. However, our impression is that this functor only describes the image of the Abel section, because it does not take into account when a twister comes from two different flows. An alternative approach for a modular interpretation could be to extend the work of \cite{MW} for torsion-free rank-$1$ sheaves.\par
	 
	Another interesting problem could be to describe the (tautological classes of the) Chow group of $\ol{\M}_{\mathcal{A},\mu}$, or at least the boundary classes and their intersections. For instance, if $\mathcal{A}=(2,-2,0)$ and $\mu=0$, then $\ol{\M}_{\mathcal{A},\mu}$ is the blow up of $\ol{\M}_{g,2}$ along the locus of two-components/two-nodes curves and two-components/three-nodes curves. It would interesting to expand the computation of the class $DRC_{\mathcal{A}}$ in terms of the tautological classes of $\ol{\M}_{g,n+1}$ and a good description of the Chow group of $\ol{\M}_{\mathcal{A},\mu}$ would be very helpful.

\subsection{Relation with the class $H_{g,\mathcal{A}}$}
 Let $\mathcal{A}$ be as in Subsection \ref{sec:double}. In \cite{FP}, Farkas and Pandharipande introduced the locus $\widetilde{H}_{g,\mathcal{A}}$ of twisted (pluri-)canonical divisors. One can see that this locus is the locus $DRL_{\mathcal{A}}$ underlying $DRC_{\mathcal A}$ (up to a change of signs in the entries of $\mathcal{A}$). If $m\geq 1$ and $\mathcal{A}$ has at least one positive entry, then $\widetilde{H}_{g,\mathcal{A}}$ has the expected codimension $g$, see \cite{FP} and \cite{S}. Moreover they defined a class $H_{g,\mathcal{A}}$ as a weighted sum of the fundamental classes of the irreducible components of $\widetilde{H}_{g,\mathcal{A}}$. We expect that for $m\geq1$, we have an equality of cycles $DRC_{\mathcal{A}}=H_{g,\mathcal{A}}$.

\subsection{Tropicalization of the Abel map}
Following the recent result \cite{ACP} on the tropicalization of $\ol{\M}_{g,n}$, 
 we proved   in \cite{AP2} that $\ol{J}^\trop_{\mu,g}$ is isomorphic to the skeleton of the Berkovich analytification of $\ol{\J}_{\mu,g}$. We believe that the same is true for $\ol{\M}_{\mathcal{A},\mu}$, i.e., the skeleton of the Berkovich analytification of $\ol{\M}_{\mathcal{A},\mu}$ is isomorphic to $\ol{M}_{\mathcal{A},\mu}^\trop$. Moreover it is expected that the following diagram 
\[
\begin{tikzcd}
  \ol{\M}^{\an}_{\mathcal{A},\mu}\arrow{r}{\beta_{\mathcal{A}}^{\an}}\arrow[swap]{rd}{\alpha_{\mathcal A}^{\an}\circ\beta_{\mathcal A}^{\an}}\arrow[bend left=28]{rrr}{\bf p} &\ol{\M}^{\an}_{g,n+1}\arrow[dashed]{d}{\alpha^{\an}_{\mathcal{A}}}\arrow{r}{\bf p}& \ol{M}^\trop_{g,n+1}\arrow{d}{\alpha_{\mathcal{A}}^\trop} &\arrow[swap]{l}{\beta_{\mathcal{A}}^\trop}\arrow{ld}{\alpha_{\mathcal A}^\trop\circ\beta_{\mathcal A}^\trop} \ol{M}^\trop_{\mathcal{A},\mu}\\
                                            & \ol{\J}^{\an}_{\mu,g} \arrow{r}{\bf p}& \ol{J}^\trop_{\mu,g}
\end{tikzcd}
\]
is commutative, where $\text{an}$ is the Berkovich analytification and $\bf p$ denotes the natural retraction maps to the skeleton. This would extend \cite[Theorem 1.3]{BR} to the universal setting. 

\section{Comparison with previous works}\label{sec:comparison}

    As mentioned in the introduction, there are a lot of recent results concerning Abel maps and the double ramification cycle. In this section, we compare the different approaches and results. We fix a sequence $\mathcal A=(a_0,\ldots, a_n,m)$ such that 
    \[
    \sum_{0\leq i\leq n} a_i+m(2g-2)=d
    \]
    and fix a universal degree-$d$ polarization $\mu$.
    \par
    
    The works \cite{H}, \cite{MW} are closely related with the present paper.  Before giving an explicit comparison, we point out that the three papers share related techniques, although set in different languages and contexts.     The part played by tropical curves in this article is played by the \emph{thickness} in \cite[Definition 3.1]{H} and by the \emph{logarithmic structure} in \cite[Section 3.1]{MW}. The part played by flows considered here is played by \emph{weightings} in \cite[Definition 2.2]{H} and by \emph{logarithmic divisors} in \cite[Definition 4.6]{MW}.  
    
    First of all, we describe the relation between the work \cite{H} and our results. The former paper deals with the case $d=0$, so let us specialize to the case $d=0$.
    In \cite[Definition 3.12]{H}, Holmes introduces a stack $\M^{\diamond}$ with a birational morphism $\beta^\diamond\col\M^\diamond\to \ol{\M}_{g,n+1}$ such that $\alpha_{\mathcal{A}}\circ\beta^\diamond$ is a morphism. The map $\beta^\diamond$ is induced by the fans $\Sigma^H_{\Gamma,D_0}$, whose cones are of the form $K_{\Gamma,\emptyset,\phi}$, with $D_0+\div(\phi)=0$. In loc.cit., these fans are called $F_{\Gamma}$ (see \cite[Definition 3.3]{H}). Clearly, when $\mu$ is the degree-$0$ canonical polarization, we have that the divisor $0$ on $\Gamma$ is $(v_0,\mu)$-quasistable. Hence, $\Sigma_{\Gamma,D_0}^H\subset \Sigma_{\Gamma,\D_0,\mu}$, which means that $\M^\diamond$ is an open dense substack of the stack $\ol{\M}_{\mathcal{A},\mu}$ introduced in the present paper; in other words $\ol{\M}_{\mathcal{A},\mu}$ is a compactification of $\M^\diamond$.  Actually, we have the same picture for any universal polarization $\mu$ such that the divisor $0$ on $\Gamma$ is $(v_0,\mu)$-quasistable for every graph $\Gamma$. Note, that since $\ol{\M}_{\mathcal{A},\mu}$ is the normalization of the closure of the image of the Abel-Jacobi section, we recover \cite[Corollary 4.6]{H}. \par 
    Next, we describe the relation with the work \cite{MW}. Define $\mathcal{A}':=(a_0,\ldots, a_n,0)$ and $d':=\sum_{0\leq i\leq n}a_i=d-m(2g-2)$.
      Let us abandon the notation in Section \ref{sec:compjac}, and denote by $\J_{d',g,n+1}\to \overline{\M}_{g,n+1}$ the universal Jacobian parametrizing invertible sheaves of total degree $d'$ on $(n+1)$-pointed stable curves of genus $g$. Note that $\J_{d',g,n+1}$ is neither separated nor it satisfies the existence part of the valuative criterion. In \cite[Theorem A]{MW}, Marcus and Wise construct a logarithmic stack $\DivMW_{g,\mathcal{A}'}$, with an explicit modular description,  sitting in a commutative diagram 
      \[
      \begin{tikzcd}
      \DivMW_{g,\mathcal{A}'} \arrow{r}{\aj}\arrow{dr}{q}& {\J}_{d',g,n+1}\arrow{d}\\
                       &\overline{\M}_{g,n+1} \arrow[bend right=30, dashed]{u}[swap]{\alpha_{\mathcal{A}'}}
      \end{tikzcd}
      \]
      such that $\aj$ is a closed embbeding, $q$ is birational and $\aj=\alpha_{\mathcal{A}'}\circ q$. \par
           
      Let $\mu'$ be a universal degree-$d'$ polarization and  consider the Jacobian $\J_{\mu',g,n+1}$ parametrizing $(\Delta,\mu')$-quasistable invertible sheaves. Of course, we have an inclusion $\J_{\mu',g,n+1}\subset \J_{d',g,n+1}$.   We have that $\ol{\M}_{\mathcal{A}',\mu'}$ is a compactification of the normalization of $\DivMW_{g,\mathcal{A}'}\cap \J_{\mu',g,n+1}$. We note that $\DivMW_{g,\mathcal{A}'}$ is not separated, while $\DivMW_{g,\mathcal{A}'}\cap \J_{\mu',g,n+1}$ is a separated open substack of $\DivMW_{g,\mathcal{A}'}$. 
      If $m=0$ and $d=0$, then $\mathcal{A}=\mathcal{A}'$, $d=d'$,  and $\M^{\diamond}$ is the normalization of $\DivMW_{g,\mathcal{A}'}\cap \mathcal{P}ic_{g,n+1}^{\underline 0}$, where $\mathcal{P}ic_{g,n+1}^{\underline 0}\subset \J_{0,g,n+1}$ is the space parametrizing pairs $(C,L)$ consisting of a $(n+1)$-pointed stable curve $C$ of genus $g$ and an invertible sheaf $L$ on $C$ having degree $0$ on every component. \par
      
       We already saw that the spaces $\M^{\diamond}$ and $\DivMW_{g,\mathcal{A}'}$ are not proper over $\ol{\M}_{g,n+1}$. Nevertheless, the former space has the property that the pre-image of the zero section in $\J_{0,g,n+1}$ is proper over $\ol{\M}_{g,n+1}$, and the latter space (when $d=0$ and $d'=-m(2g-2)$) has the property that the pre-image of the pluricanonical section $\alpha_{(0,\ldots,0,-m)}\col \overline{\M}_{g,n+1}\to \J_{d',g,n+1}$ is proper over $\ol{\M}_{g,n+1}$.  Each one of these properties is enough to construct a double ramification cycle for any $m$ (see \cite[Theorem 1.2]{H} and \cite[Theorem C]{MW}), even though the stack $\DivMW_{g,{\mathcal{A}'}}$ is only constructed for $m=0$. As expected, both constructions give rise to the same cycle  (see \cite[Section 2.6]{HS}), which agrees,  when $m=0$, with the double ramification cycle defined in \cite{L1,L2,GV} (see \cite[Theorem 1.3]{H}). Since our space $\ol{\M}_{\mathcal{A},\mu}$ is a compactification of $\M^{\diamond}$, this is the same double ramification cycle considered in Section \ref{sec:double}.\par

      The main advantage of our approach is that we construct a proper Deligne-Mumford stack, hence better suited to do intersection theory. We describe explicitly the local structure of this stack by means of tropical geometry. Our construction works for every degree, which can be helpful if one intends to study the pull back of other Brill-Noether cycles. The price to pay is that, as expected, we have to consider torsion-free rank-$1$ sheaves instead of just invertible sheaves.\par

      Another alternative way to solve the Abel-Jacobi section, that works in any degree, is to find a suitable polarization $\mu$ for which the map is already defined. In \cite[Proposition 4.25]{M15}, Melo explicitly describes one such polarization. In \cite[Corollary 6.8]{KP}, Kass and Pagani characterize all polarizations $\mu$ such that the Abel-Jacobi section is a morphism. In this case, we have that $\ol{\M}_{\mathcal{A},\mu}=\ol{\M}_{g,n+1}$. However, it is not always the case that the zero section extends to a section  $\overline{\M}_{g,n+1}\to \overline{\J}_{\mu,g,n+1}$: this is an obstruction to construct a double ramification cycle.
    If the zero section does extend, that is, if the universal polarization $\mu$ is such that $\mathcal{P}ic_{g,n+1}^{\underline{0}}\subset \overline{\J}_{\mu,g,n+1}$, then we can define the double ramification cycle simply by pulling back the zero section. This construction agrees with the ones already mentioned: this is \cite[Theorem 3.1]{HKP}.\par

       As a final note, we remark that in \cite[Theorem 0.1]{Gu}, Gu\'er\'e  constructs a moduli space of $m$-log canonical divisors, carrying a perfect obstruction theory and allowing the definition of a double ramification cycle. This construction is more in the spirit of Li, Graber and Vakil \cite{L1,L2, GV}, and bypass the need for resolving the Abel Jacobi section. However, for $m>1$, this cycle does not coincide with the cycle constructed in \cite{H} and considered here, see \cite[Section 1.5]{HS}.

\section*{Acknowledgments}	

We wish to thank Dan Abramovich, Matthew Baker, Lucia Caporaso, Renzo Cavalieri, Eduardo Esteves, David Holmes, Margarida Melo, Nicola Pagani, Dhruv Ranganathan, and Martin Ulirsch for useful conversations and comments. We thank the referee for carefully reading the paper and for the constructive suggestions.  The second author was supported by CNPq, 301314/2016-0.

\bigskip
\noindent{\smallsc Alex Abreu and Marco Pacini,}

\noindent{\smallsc Instituto de Matem\'atica--Universidade Federal Fluminense,}

\noindent{\smallsc Rua Prof.M.W.de Freitas, s/n,  24210-201, Niter\'oi-Rio de Janeiro, Brazil.}

\noindent 
{\smallsl E-mail addresses: \small\verb?alexbra1@gmail.com?  and \small\verb?pacini.uff@gmail.com?}

\end{document}